\documentclass[12pt]{article}

\usepackage[margin=1.20in]{geometry}

\usepackage{amssymb,amsmath}
\usepackage{amsthm,enumitem}
\usepackage{hyperref}
\usepackage{mathrsfs}
\usepackage{textcomp}
\hypersetup{
colorlinks,%
citecolor=black,%
filecolor=black,%
linkcolor=black,%
urlcolor=black
}
\usepackage[abbrev,nobysame]{amsrefs}
\usepackage{todonotes}

\usepackage{verbatim}
\usepackage{tikz}
\usetikzlibrary{calc}
\usepackage{sectsty}

\makeatletter

\newdimen\bibspace
\makeatother

\newtheorem{Theorem}{Theorem}[section]
\newtheorem{Lemma}[Theorem]{Lemma}
\newtheorem{Proposition}[Theorem]{Proposition}
\newtheorem*{Assumption*}{Assumption (H)}
\newtheorem*{Assumption**}{Assumption (G)}

\newtheorem{ex}[Theorem]{Example}
\newtheorem{Corollary}[Theorem]{Corollary}
\newtheorem{Remark}[Theorem]{Remark}

\def\XXint#1#2#3{{\setbox0=\hbox{$#1{#2#3}{\int}$}
\vcenter{\hbox{$#2#3$}}\kern-.5\wd0}}

\newcommand{\Om}{\Omega}                
           \newcommand{\ud}{\mathrm{d}}
\newcommand{\be}{\begin{equation}}      \newcommand{\ee}{\end{equation}}

\newcommand{\T}{\mathcal{T}}

\newcommand{\A}{\mathcal{A}}
\newcommand{\C}{\mathcal{C}}
\newcommand{\Q}{\mathcal{Q}}
\newcommand{\R}{\mathbb{R}}              
\newcommand{\D}{\mathcal{D}}

\newcommand{\E}{\mathcal{E}}

\newcommand{\M}{\mathscr{M}}

\begin{document}

\title{\textbf{Extremal Alexandrov estimates: singularities, obstacles, and stability}\bigskip}

\author{\medskip  Tianling Jin,\footnote{T. Jin was partially supported by NSFC grant 12122120, and Hong Kong RGC grants GRF 16304125, GRF 16303624 and GRF 16303822.}\quad Xushan Tu, \quad Jingang Xiong\footnote{J. Xiong was partially supported by NSFC grants 12325104.}}

\date{}

\maketitle

\begin{center}
\vspace{-1cm}
\emph{\small{Dedicated to Henri Berestycki on the occasion of his 75th birthday, with \\ admiration and friendship.\medskip}}
\end{center}

\begin{abstract} 
The classical Alexandrov estimate controls the oscillation of a convex function by the mass of its associated Monge-Amp\`ere measure and yields, for two convex functions of $n$ variables with the same boundary values, a sup-norm bound with exponent $1/n$ in the measure discrepancy. We show that this exponent is not optimal in the small-discrepancy regime once one of the functions is non-degenerate in the sense of having Monge-Amp\`ere density bounded above and below by two positive constants.

We prove sharp quantitative estimates comparing two convex functions by the total variation of the difference of their Monge-Amp\`ere measures: in dimensions $n\ge 3$ the optimal dependence is quadratic in the natural mass scale, while in dimension $n=2$ the optimal dependence contains a logarithmic correction. These rates are shown to be optimal for all small discrepancies.

A key structural ingredient is a characterization of extremizers. We identify the pointwise minimizers and maximizers in the admissible class and prove that they are realized, respectively, by solutions to Monge-Amp\`ere equations with an isolated singularity and by solutions to Monge-Amp\`ere equations with a linear obstacle. This extremal description reduces the sharp estimates to a precise asymptotic analysis of these two model configurations.

Assuming further that the domain and the non-degenerate reference function are $C^{2,\alpha}$ and uniformly convex, we obtain sharp pointwise two-sided asymptotics at interior points with explicit leading constants. Finally, in dimensions $n\ge 3$ we establish a stability phenomenon: if the pointwise estimate is nearly saturated, then the measure discrepancy must concentrate near the point at the natural scale, quantifying rigidity of almost-extremal configurations.

\medskip


\medskip


\end{abstract}

\tableofcontents

\section{Introduction}\label{sec:introduction}
 
The Alexandrov estimate plays a fundamental role in the theory of elliptic partial differential equations. It asserts (see, e.g., Lemma 9.2 in Gilbarg-Trudinger \cite{GT}) that for every $w \in C(\overline{\Omega})$ defined on a bounded convex domain $\Omega \subset \R^n$, 
\begin{equation}\label{eq:abp classical}
\inf_{\Omega} w \geq \inf_{\partial \Omega} w-  \omega_n^{-\frac{1}{n}} \operatorname{diam}(\Omega)  \M \Gamma_w(\{w=\Gamma_{w}\})^{\frac{1}{n}},
\end{equation}
where $\Gamma_{w}$ denotes the convex envelop of $w$,  $\M \Gamma_{w}$ is the Monge-Amp\`ere measure associated with $\Gamma_{w}$, $\operatorname{diam}(\Omega)$ is the diameter of $\Omega$, and $\omega_n$ is the volume of the unit ball in $\R^n$.

Let $u, \varphi \in C(\overline{\Omega})$ be convex and satisfy $u=\varphi$ on $\partial\Omega$. Applying \eqref{eq:abp classical} to $\pm(u - \varphi)$ yields the following Alexandrov estimate:
\begin{equation} \label{eq:abp gene 1/n}
\|u - \varphi\|_{L^{\infty}(\Omega)} \leq \omega_n^{-\frac{1}{n}} \operatorname{diam}(\Omega)   |\M u - \M \varphi|(\Omega)^{\frac{1}{n}},
\end{equation}
where $|\M u - \M \varphi|$ stands for the total variation of the signed measure $\M u - \M \varphi$. 

\subsection{Main results}

While the exponent $\frac{1}{n}$ in \eqref{eq:abp gene 1/n} is optimal as the total variation $|\M u - \M \varphi|(\Omega)\to+\infty$ (corresponding to the degenerate case $\varphi\to 0$ under scaling), we show in this paper that it is not optimal when the total variation is tending to $0^+$ as long as $\varphi$ is non-degenerate in the sense that
\begin{equation}\label{eq:varphi equation}
0<\lambda \leq \det D^2 \varphi \leq \Lambda<\infty  
\end{equation} 
 for some  constants $\lambda$ and $\Lambda$, and obtain the following improved and optimal estimates.

\begin{Theorem}\label{thm:ordera2}
Let $\Omega \subset \R^n$, $n \geq 2$, be a convex domain satisfying $B_1(0) \subset \Omega \subset B_n(0)$, and let $\varphi \in C(\overline{\Omega})$ be a convex function satisfying \eqref{eq:varphi equation} in $\Omega$. Then there exists a positive constant $C$ depending only on $n,\lambda$ and $\Lambda$ such that for every convex $u \in C(\overline{\Omega})$ with $u=\varphi$ on $\partial \Omega$, there holds
\begin{equation}\label{eq:abp gene 2/n log}
\left\| u - \varphi \right\|_{L^{\infty}(\Omega)} \leq
\begin{cases}
Ca^2 (|\log a|^2+1) & \text{if } n = 2, \\
Ca^2 (|\log a|+1) & \text{if } n \geq 3,
\end{cases} 
\end{equation}
where  
\begin{equation}\label{defn:a}
a:=\omega_n ^{-\frac{1}{n}} |\M u - \M \phi|(\Omega)^{\frac{1}{n}}.
\end{equation}
If we additionally assume $\varphi$ is strictly convex and satisfies \eqref{eq:varphi equation} in an extended convex domain $\widetilde{\Omega}$ with $\Omega \subset \subset \widetilde{\Omega} \subset B_{2n}(0)$, then
\begin{equation} \label{eq:abp gene 2/n}
\left\| u - \varphi \right\|_{L^{\infty}(\Omega)} \leq
\begin{cases}
C_{\varphi}a^2 (|\log a|+1) & \text{if } n = 2, \\
C_{\varphi}a^2 & \text{if } n \geq 3,
\end{cases}
\end{equation} 
where $C_{\varphi}>0$ depends only on $n$, $\lambda$, $\Lambda$,   and the lower bound of $\displaystyle \inf_{p \in \partial \varphi(\Omega)} \inf_{x\in \partial \widetilde{\Omega}} (\varphi(x) - \ell_p(x))$, $\ell_p$ denoting the support function of $\varphi$ with slope $p$.
\end{Theorem}

The estimates in Theorem \ref{thm:ordera2} are stated for normalized convex domains and extend to arbitrary convex domains via affine transformations;  see Section \ref{sec:normalization reduction}. It is worth noting that, for $n \geq 3$, the estimate \eqref{eq:abp gene 2/n} is affine invariant.

The asymptotic order $a^2$ or $a^2|\log a|$ in \eqref{eq:abp gene 2/n} is \emph{optimal} for all small $a$, and is achieved by solutions of either Monge-Amp\`ere equations with isolated singularities or Monge-Amp\`ere equations with linear obstacles. 
%They are extremal configurations for the variational problem \eqref{eq:abp class Da}; see Theorem \ref{thm:abp extremal}. 
For example, when $\varphi=\frac12|x|^2$ and $\Omega$ is a ball, then the optimality is achieved by an additive translate of
\begin{equation}\label{eq:isolated global solution}
W_a(x)=\int_{0}^{|x|}  \left(   r^n+a^n \right)  ^{\frac{1}{n}}\,\ud r,
\end{equation}
which is a solution of  
\[ 
\det D^2 W_a= 1 + \omega_n a^n \delta_0,
\]
and by  an additive translate of its Legendre transform
\begin{equation}\label{eq:obstacle global solution}
W_a^*(x):= \int_{0}^{|x|}   \max\left\{ r^n-a^n,0\right\}^{\frac{1}{n}}\,\ud r,
\end{equation} 
which is a solution of 
\[
\det D^2 W_a^*=\chi_{\left\{W_a^*>0\right\}}
\] 
satisfying $\left| \left\{ W_a^*=0\right\} \right|= \omega_n a^n$. Here, $\delta_y$ denotes the Dirac measure centered at the point $y$,  $\chi_E$ is the characteristic function of a set $E$, and  $|E|$ is its  Lebesgue measure. These two functions $W_a$ and $W_a^*$ will serve as our model functions in our analysis.

%\textcolor{blue}{This paragraph is the old one, and is to be deleted:} The estimate \eqref{eq:abp gene 2/n log} is a global estimate, for which we do not assume the strict convexity of $\varphi$. An important step of proving \eqref{eq:abp gene 2/n log} is that if there is a point where the deviation causing the failure of \eqref{eq:abp gene 2/n} occurs, then $\varphi$ will be strictly convex and $C^{1,\alpha}$ near that point. The estimate \eqref{eq:abp gene 2/n log} then follows from an iterative argument, at the cost of an additional $|\log a|$ factor. 

If we further assume $\partial \Omega \in C^{2,\alpha}$ and $\varphi \in \C_+^{2,\alpha}(\overline{\Omega})$,
where $\C_+^{2,\alpha}$ denotes the space of $C^{2,\alpha}$ functions with \emph{positive definite Hessians}, then we obtain a sharp pointwise estimate:
 
\begin{Theorem}\label{thm:ordera2 more}
Suppose that $\Omega \subset \R^n$, $n\geq 2$, is a bounded convex domain with $\partial \Omega \in C^{2,\alpha}$, where $\alpha \in (0,1)$. Let $\varphi \in \C_+^{2,\alpha}(\overline{\Omega})$ and $x_0 \in \Omega$. Define $\lambda_0 := (\det D^2 \varphi(x_0))^{1/n}$.  
For every convex function $u \in C(\overline{\Omega})$ with $u=\varphi$ on $\partial \Omega$, we have:
\begin{itemize}
\item If $n=2$, then \begin{equation}\label{eq:perturb-result n=2}
- \frac{1}{2}\lambda_0^{-1}   -  \frac{C}{|\log a| }\leq \frac{u(x_0) - \varphi(x_0)}{ a^2|\log a|} \leq \frac{1}{2}\lambda_0  +  \frac{C}{|\log a| };
\end{equation}
\item If $n\geq 3$, then 
\begin{equation}\label{eq:perturb-result}
- \lambda_0^{-1}   - C a^{ \beta}   \leq \frac{u(x_0) - \varphi(x_0)}{d_{n,0}a^2} \leq \lambda_0  + C a^{ \beta} ,  
\end{equation}
\end{itemize}
where $a$ is the one defined in \eqref{defn:a}, $\beta = \frac{(n-2)\alpha}{n+\alpha}$,  $d_{n,0} = \frac{\Gamma\left(\frac{1}{n}\right)\Gamma\left(\frac{n-2}{n}\right)}{2n\Gamma\left(\frac{n-1}{n}\right)}  > 0$ for $n\ge 3$  is given by
\[
\int_{0}^{|x|} \left( r^n + 1 \right)^{\frac{1}{n}} \ud r = \frac{1}{2} |x|^2+ d_{n,0} + O(|x|^{-1}) \quad \text{as } |x| \to \infty,
\]
and $C> 0$ depends only on $n$, $\alpha$, $\operatorname{diam}(\Omega)$, $\left\|\partial \Omega\right\|_{C^{2,\alpha}}$,
$\left\|D^2\varphi\right\|_{C^{\alpha}(\overline{\Omega})}$ and $\left\|(D^2\varphi)^{-1}\right\|_{L^{\infty} (\Omega)}$.
\end{Theorem}

The terms $\frac{1}{2}\lambda_0^{-1}$, $\frac{1}{2}\lambda_0$, and $a^2|\log a|$ in \eqref{eq:perturb-result n=2}, as well as $\lambda_0^{-1}$, $\lambda_0$, and $d_{n,0}a^2$ in \eqref{eq:perturb-result}, are all \emph{optimal}. This optimality is again attained by the extremal configurations for the variational problem \eqref{eq:abp class Da} --namely, the isolated singularity problem and the obstacle problem for Monge-Amp\`ere equations.

Moreover, these two extremal configurations for the estimate \eqref{eq:perturb-result} are actually \emph{stable}, in the sense that if $\displaystyle\frac{u(x_0) - \varphi(x_0)}{d_{n,0}a^2}$ approaches the lower or upper bounds (for small $a>0$) in \eqref{eq:perturb-result}, then the corresponding measure 
\[
\mu:=\M u -\M \varphi 
\]
tends to concentrate around $x_0$. The details of this stability are as follows. 
\begin{Theorem}\label{thm:rigidity}
Suppose $n \geq 3$ and all assumptions in Theorem \ref{thm:ordera2 more} are satisfied. Let $x_0 \in \Omega$, $\lambda_{0}=\left( \det D^2 \varphi(x_0)\right)^{\frac{1}{n}}$, $r > 0$, and define the Hessian-scaled ellipsoid
\[
E (x_0, r):= \left\{ x \in \Omega :\; (x-x_0) \cdot D^2 \varphi (x_0) \cdot (x-x_0)^T  \leq r^2 \right\}.
\]
Then there exists $\varepsilon_0>0$  depending only on $n$, $\alpha$, $\operatorname{diam}(\Omega)$, $\left\|\partial \Omega\right\|_{C^{2,\alpha}}$,
$\left\|D^2\varphi\right\|_{C^{\alpha}(\overline{\Omega})}$ and $\left\|(D^2\varphi)^{-1}\right\|_{L^{\infty} (\Omega)}$, such that for every $\rho \in (0,1)$ and every convex function $u \in C(\overline{\Omega})$ with $u = \varphi$ on $\partial \Omega$ satisfying $a \leq \varepsilon_0$ (where $a$ is defined in \eqref{defn:a}), there exist positive constants $c_\rho$ and  $C_{\rho}$ both of which depend only on $n$, $\alpha$, $\rho$, $\operatorname{diam}(\Omega)$, $\left\|\partial \Omega\right\|_{C^{2,\alpha}}$,  $\left\|D^2\varphi\right\|_{C^{\alpha}(\overline{\Omega})}$ and $\left\|(D^2\varphi)^{-1}\right\|_{L^{\infty} (\Omega)}$,  such that:
\begin{align}
\frac{ u(x_0)-\varphi(x_0)}{d_{n,0}a^2} 
& \geq -\lambda_{0}^{-1}-C_{\rho}a^{\beta}+c_{\rho}\left(\frac{\omega_na^n-\mu \left( E \left(x_0, \rho a\right)\right) }{\omega_na^n} \right)^\frac{n}{2}, \label{eq:stabilitypositive} \\
\frac{ u(x_0)-\varphi(x_0)}{d_{n,0}a^2} 
& \leq \lambda_{0} +C_{\rho}a^{\beta}-c_{\rho}\left(\frac{\omega_na^n+\mu \left( E (x_0, (1+\rho)a)\right)}{ \omega_na^n} \right)^n. \label{eq:stabilitynegative} 
\end{align}
\end{Theorem}

Finally, we note that Theorem \ref{thm:ordera2 more} has an application to the strict convexity and regularity of solutions to Monge-Amp\`ere equations with multiple isolated singularities
\begin{equation*}
\det D^2 u=\det D^2 \varphi+\omega_n a^n \sum_{i=1}^{m}  b_i \delta_{y_i},
\end{equation*}  
as stated in Corollary \ref{prop:strictly convex cond}. 

In the companion work \cite{jin2026sharp}, we study the global version (i.e., \(\Omega=\R^n\)) of our extremal Alexandrov estimate. There we prove that, for \(n\ge 3\), if \(u:\R^n\to\R\) is convex and \(a:=\omega_n ^{-\frac{1}{n}} |\M u - 1|(\R^n)^{\frac{1}{n}}<\infty\), then there exist $A \in \A_n$, $b \in \mathbb{R}^n$, and $c \in \mathbb{R}$ such that
\[
\limsup_{|x| \to \infty} \left| u(x) - \left(\frac{1}{2}x^{\top}Ax  + b \cdot x + c\right) \right| = 0,
\]
where $\A_n$ denotes the set of positive definite symmetric $n \times n$ matrices with determinant $1$. Moreover, 
\[
\left\| u(x) - \left(\frac{1}{2}x^{\top}Ax  + b \cdot x + c\right) \right\|_{L^{\infty}(\R^n)}
\le d_{n,0}a^2
\]
with equality if and only if $u=W_a$ or $u=W_a^*$. Two of the essential ingredients of the proof are the sharp pointwise estimate in Theorem~\ref{thm:ordera2 more}
and the quantitative stability in Theorem~\ref{thm:rigidity}.

\subsection{Ideas of the proof}

The starting point of our approach is that the optimal comparison between $u$ and a non-degenerate background $\varphi$ is not captured by a single application of the classical estimate \eqref{eq:abp classical}, which is designed to be robust in degenerate scaling regimes. 
Instead, we introduce a pointwise variational viewpoint. 
Fix $x_0\in\Omega$ and a mass budget $\omega_na^n:= |\M u-\M\varphi|(\Omega)$. We consider the class of convex functions $w$ satisfying $w=\varphi$ on $\partial\Omega$ and $|\M w-\M\varphi|(\Omega)\le \omega_n a^n$, and ask how far $w(x_0)$ can deviate from $\varphi(x_0)$. 
A main conceptual contribution is that this extremal problem is \emph{explicitly solvable}: the pointwise minimizers and maximizers are realized by two canonical Monge--Amp\`ere configurations (Theorem \ref{thm:abp extremal}).

On the ``lower'' side, the extremizer is an isolated singularity solution, obtained by concentrating the entire discrepancy as a Dirac mass:
one replaces $\M\varphi$ by $\M\varphi+\omega_n a^n\delta_{x_0}$. 
On the ``upper'' side, the extremizer is an obstacle solution: one imposes a linear obstacle given by a supporting hyperplane of $\varphi$, shifted up by an additive constant, and uses the mass budget to create a nontrivial coincidence set. 
These two mechanisms are dual to each other (as illustrated by $W_a$ and its Legendre transform $W_a^*$ in the quadratic background), and they provide a structural explanation for why the small-discrepancy behavior differs from the classical exponent $1/n$ in \eqref{eq:abp gene 1/n}. The obstacle problem is typically more delicate, since one must control the location of the coincidence set within the sections of the reference function.

Once the extremizers are identified, the sharp bounds in Theorem \ref{thm:ordera2} are reduced to a quantitative analysis of the corresponding extremal problems. A useful ingredient in the proof of the sharp estimate  \eqref{eq:abp gene 2/n}  is a quantitative propagation estimate for small Monge-Amp\`ere perturbations (Subsection \ref{subsection:2/n}). 
Roughly speaking, if the discrepancy measure $\mu:=\M u-\M\varphi$ is supported in a deep interior section of $\varphi$, then the deviation $u-\varphi$ on an outer annulus of sections cannot be large: it is controlled linearly by $|\mu|(\Omega)$. 
The proof combines the Caffarelli--Guti\'errez's \cite{caffarelli1997properties}  Harnack inequality  for the linearized Monge--Amp\`ere equations on annular sections and a flux identity in divergence form (cf.\ \eqref{eq:div identity smooth}). 
Iterating this estimate across dyadic sections yields the sharp decay profile and ultimately the $a^2$ (resp.\ $a^2|\log a|$) scale in \eqref{eq:abp gene 2/n}.

A further novelty is that we obtain a \emph{global} estimate \eqref{eq:abp gene 2/n log} without assuming strict convexity of $\varphi$. We show that any point where the sharp estimate \eqref{eq:abp gene 2/n} could fail must belong to a region where $\varphi$ is automatically strictly convex and enjoys $C^{1,\alpha}$ regularity. This mechanism is quantified in Lemma \ref{lem:varphi sc from u} and Lemma \ref{lem:varphi sc from v}. We then propagate the local sharp estimate across scales by an iteration in sections, at the cost of an additional logarithmic factor.

Under higher regularity assumption on the domain and the reference function, we push this analysis to obtain sharp \emph{pointwise} asymptotics with optimal constants (Theorem \ref{thm:ordera2 more}). 
Here the isolated singularity and obstacle profiles determine not only the correct order but also the leading coefficients.  This sharp pointwise estimate is obtained by a sharp small--$a$ asymptotic analysis of these extremizers. 
After a Whitney extension, we may work in an interior regime and normalize at the point $x_0$. 
Theorem \ref{thm:abp extremal} reduces the lower bound to the isolated singularity solution $u_a(\cdot,x_0)$ and the upper bound to an obstacle solution $v_a(\cdot,p)$ whose coincidence set contains $x_0$. We compare $u_a$ and $v_a$ with carefully constructed barrier functions at the $a$--scale.
Subsection \ref{sec:singularity problem} provides precise asymptotic expansions for these two model configurations (summarized in Proposition \ref{prop:pointwise sharp}), including the optimal constants in front of $a^2|\log a|$ (when $n=2$) and $a^2$ (when $n\ge3$). 
To transfer the obstacle asymptotics to the value at $x_0$, we localize the relevant contact point $x_p$ with $D\varphi(x_p)=p$ using the section geometry of the coincidence set (Lemma \ref{lem:coincidencesetcontrol}) and the $\C_+^{2,\alpha}$ regularity of $\varphi$, which yields $|x_p-x_0|= O(a)$ and hence the required control of $(\det D^2\varphi)^{1/n}$ along this shift.

Finally, in dimensions $n\ge 3$ we prove a quantitative stability result  (Theorem \ref{thm:rigidity}): if the pointwise inequality is close to equality at $x_0$, then the signed measure $\mu=\M u-\M\varphi$ must concentrate near $x_0$ at the natural $a$-scale. The proof relies on a qualitative analysis of solutions to the Monge-Amp\`ere equations with \emph{two} isolated singularities or \emph{two} linear obstacles,  subject to the same mass budget. Lemma \ref{lem:extends to boundary} (a boundary contact lemma) shows, roughly speaking, that if two convex functions agree on the boundary and their Monge--Amp\`ere measures differ only by a compactly supported perturbation with zero total mass, then the regions where each function dominates reach the boundary. This separation mechanism is repeatedly used in the subsequent analysis. Lemma \ref{lem:two singularity} quantifies the discrepancy created by shifting a small fraction of the Dirac mass away from \(x_0\) in the two-singularity model, while Lemma \ref{lem:double obs 1}  quantifies the discrepancy created by lowering the primary obstacle and thereby activating a second obstacle in the two-obstacle model. They furnish delicate barrier functions that yield a rigidity mechanism ruling out ``macroscopic splitting'' of the discrepancy measure in almost-extremal situations.

\medskip

This paper is organized as follows. In Section \ref{sec:notation}, we introduce the basic notation and technical reductions adopted throughout the paper.
In Section \ref{sec:alexandrov estimate classical}, we revisit the classical Alexandrov estimate and characterize the extremal configurations in Theorem \ref{thm:abp extremal}—namely, the isolated singularity problem and the obstacle problem. 
In Section \ref{sec:alexandrov estimate affine}, we derive estimates for the extremal configurations and complete the proof of Theorem \ref{thm:ordera2}.
In Section \ref{sec:asymptotic behavior}, we further analyze the asymptotic behavior of the extremal configurations, demonstrate the sharpness of the order and constants in \eqref{eq:perturb-result n=2} and \eqref{eq:perturb-result}, and prove Theorem \ref{thm:ordera2 more}, along with Corollary \ref{prop:strictly convex cond}.
In Section \ref{sec:rigidity a2}, we perform a qualitative analysis of solutions to Monge-Ampère equations with two isolated singularities or two linear obstacles, leading to the proof of Theorem \ref{thm:rigidity}.

%\par\bigskip\noindent
%\textbf{Acknowledgments.}	

\medskip 

\noindent \textbf{Conflict of interest:} All authors certify that there is no actual or potential conflict of interest about this article.

\section{Notation and Conventions}\label{sec:notation}

Throughout this paper, we denote by $c$ and $C$ universal positive constants that depend at most on 
$n$, $\alpha$, $\rho$, $\lambda$, $\Lambda$, $\operatorname{dist}(\widetilde{\Omega},\Omega)$,  $\operatorname{diam}(\Omega)$, $\displaystyle \inf_{p \in \partial \varphi(\Omega)} \inf_{x\in  \partial \widetilde{\Omega}} (\varphi(x) - \ell_p(x))$, $\left\|\partial \Omega\right\|_{C^{2,\alpha}}$,  $\left\|D^2\varphi\right\|_{C^{\alpha}(\overline{\Omega})}$ and $\left\|(D^2\varphi)^{-1}\right\|_{L^{\infty} (\Omega)}$ whenever they are involved.
Relevant dependencies on parameters will be highlighted by placing them in parentheses if necessary. These constants may change from line to line.

For two non-negative (or non-positive) quantities $ b_1$ and $b_2 $, we adopt the notation 
\[
b_1 \lesssim b_2
\]
to indicate that $ b_1 \leq Cb_2 $ for some universal constant $C$. 
We also denote 
\[
b_1 \approx b_2 
\] 
if  $ b_1 \lesssim b_2 $  and $b_2 \lesssim b_1$.  
For matrices $ A $ and $B $, we write $ A \leq B $ when $ B - A $ is positive semidefinite. Similarly, we use $A \lesssim B$ for two positive semidefinite matrices $A$ and $B$  to indicate that  $A \leq CB $, 
and we use $A \approx B$ if  $A \lesssim B $  and $B \lesssim A$.  
In the case $c$ and $C$ 
depend on specific parameters, such as $ a $ or $ p $, we will use the notations $\lesssim_{a,p}$, $\approx_{a,p}$to emphasize these dependencies. The symbol $ \gtrsim$ is understood similarly.

The section of $u$ at $x_0$ with height $t$  for subgradient $p \in \partial u(x_0)$ is denoted as
\[
S_{t}(x_0)= S_{t,p}^u(x_0) = \left\{ x  \in \overline{\Omega}  :\;  u(x) < u(x_0)+p\cdot (x-x_0)+t  \right\}.
\]
A convex function $u$ is strictly convex at $x_0 \in \Omega$ if there exists a $p \in \partial u(x_0)$ such that
\[ 
u(x)>u(x_0)+p \cdot(x-x_0)  \text{ for all } x\in \Omega,\ x \neq x_0,
\]
and $u$ is strictly convex in $\Omega$ if  $u$ is strictly convex at every point in $\Omega$.   
Additionally, $u$ is uniformly convex if the eigenvalues of $D^2 u$ are greater than some positive constant in $\Omega$.
We denote by $\mathcal{C}^{1,1}_{+}$ the class of uniformly convex $C^{1,1}$ functions. For $k\ge 2$ and $\alpha\in(0,1)$, we set
\[
\mathcal{C}^{k,\alpha}_{+}:= C^{k,\alpha}\cap \mathcal{C}^{1,1}_{+}.
\]
Moreover, we equip $\mathcal{C}^{1,1}_{+}$ with the norm
\[
\left\|u \right\|_{\C_+^{1,1}}= \left\| u\right\|_{C^{1,1}}+\left\|(D^2 u)^{-1}\right\|_{L^\infty}.
\] 

A  bounded convex domain $\Omega$  is said to be strictly convex, uniformly convex, $\C_+^{1,1}$ or $\C_+^{k,\alpha}$, if it can be locally represented by functions that are strictly convex, uniformly convex, $\C_+^{1,1}$ or $\C_+^{k,\alpha}$.   
The norm $\left\|\partial \Omega\right\|_{\C_+^{1,1}}$ and $\left\|\partial \Omega\right\|_{C^{k,\alpha}}$ are defined in the usual way via local graphical parametrizations of $\partial \Omega$, together with the diameter of $\Omega$ to fix scaling.

\subsection{Reduction by normalization}\label{sec:normalization reduction}

As mentioned earlier, the estimates in Theorem \ref{thm:ordera2} are presented for normalized convex domains. By the affine invariance of the Monge–Ampère equation, these estimates, when written in an appropriate affine form, extend to arbitrary convex domains. The normalization argument presented below carries out this reduction without loss of generality; the same reasoning applies to both interior and pointwise estimates.

More precisely, let $\Omega\subset\mathbb{R}^{n}$ be an arbitrary bounded convex domain. Since all the estimates are invariant under the subtraction of affine functions, we may replace $u$ and $\varphi$ by $u-\ell$ and $\varphi-\ell$ for a suitable linear function $\ell$ without affecting any of the estimates. Moreover, after applying a normalization (affine) transformation, we have
\begin{equation}\label{eq:simplify normalization 1}
u_{\T}=(\det \T)^{-\frac{2}{n}} u(\T x),\ \  {\varphi}_{\T} =(\det \T)^{-\frac{2}{n}} \varphi(\T x) ,\ \   B_1(0) \subset {\Omega}_{\T}:= \T^{-1} \Omega \subset B_n(0),
\end{equation}
and the total variation of the associated measure $\mu = \M u - \M \varphi$ becomes 
\begin{equation}\label{eq:simplify normalization 2}
|{\mu}_{\T}| (\Omega_{\T}) = (\det \T)^{-1} |\mu|(\Omega),\quad \text{where }  {\mu}_{\T}=\M  {u }_{\T}- \M {\varphi}_{\T}.
\end{equation}
Let $a > 0$ be defined by $|\mu|(\Omega) = \omega_n a^n$. 
\begin{itemize}
\item The classical Alexandrov estimate \eqref{eq:abp gene 1/n}, applied to $\pm(u_{\T}-\varphi_{\T})$, implies:
\begin{equation}\label{eq:scaledA}
\|u - \varphi\|_{L^{\infty}(\Omega)} \leq    C(n) a |\Omega|^{\frac{1}{n}}.
\end{equation}
\item  Estimate \eqref{eq:abp gene 2/n}, applied to $u_{\T}-\varphi_{\T}$, implies:
\[
\left\| u - \varphi \right\|_{L^{\infty}(\Omega)} \leq 
\begin{cases}
C_{\varphi_{\T}}a^2 (|\log (a|\Omega|^{-\frac{1}{n}})|+1) & \text{if } n = 2, \\
C_{\varphi_{\T}}a^2 & \text{if } n \geq 3.
\end{cases}
\]
Let $\epsilon\in (0,1)$. If $\varphi$ is non-degenerate (in the sense of \eqref{eq:varphi equation}) in $S_{h}^{\varphi}$, then for $\Omega = S_{(1-\epsilon)h}^{\varphi}$, the quantity $C_{\varphi_{\T}}$ becomes a universal constant that depends only on $n$, $\epsilon$, $\lambda$, and $\Lambda$. Consequently, the estimate \eqref{eq:abp gene 2/n} for $n \geq 3$ is affine invariant in this setting.
\end{itemize}
 
\subsection{Reduction by the comparison principle}\label{app:comparison reduction}

Throughout this paper, we use the Jordan decomposition
\[
\mu = \M u - \M  \varphi = \mu_+ - \mu_-,
\]
where $\mu_+$ and $\mu_-$ are nonnegative finite measures. 
Our analysis also relies on several technical reductions that are readily justified by the comparison principle, as explained below. In the main text, we will therefore simply invoke them as ``by the comparison principle''.
 
These reductions rely on the existence of solutions to the relevant Dirichlet problems, the proof of which will be provided in Proposition~\ref{prop:finitetoinfinite}. We invoke the reductions only after Proposition~\ref{prop:finitetoinfinite} has been established.

\subsubsection{Reduction to one-signed measures}\label{sec:signed}
Let $u_+$ and $u_-$ be the convex solutions to 
\begin{align*}
& \M u_{+}=\M \varphi+\mu_+  \quad \text{in } {\Omega}, \quad u_{+} =u\quad \text{on } \partial \Omega, \\
& \M   u_{-} =\M \varphi-\mu_- \quad \text{in } {\Omega}, \quad   u_{-}  =u\quad \text{on } \partial {\Omega},  
\end{align*} 
respectively. The comparison principle yields
\[
u_{+} \leq u  \leq u_{-} \quad \mbox{and} \quad u_{+} \leq \varphi  \leq u_{-}.
\]
Consequently,
\[
u_+- \varphi\le u-\varphi\le u_--\varphi.
\]
Thus, estimating $u_+ - \varphi$ and $u_- - \varphi$ directly provides lower and upper bounds for $u - \varphi$.
In view of this reduction, we may assume without loss of generality that $\mu$ is either nonnegative or nonpositive, which implies $u \leq \varphi$ or $u \geq \varphi$, respectively. 

\subsubsection{Reduction to interior estimates} \label{sec:extension}
Suppose $\varphi$ extends to a convex function $\widetilde{\varphi}$ defined on a larger convex domain $\widetilde{\Omega}$. 
Let $\widetilde{u}_+$ and $\widetilde{u}_-$ solve 
\begin{align*}
\M \widetilde{u}_+ &= \M \widetilde{\varphi} + \mu_+ \quad \text{in } \widetilde{\Omega}, \quad \widetilde{u}_+  = \widetilde{\varphi} \quad \text{on } \partial \widetilde{\Omega}, \\
\M \widetilde{u}_- &= \M \widetilde{\varphi} - \mu_- \quad \text{in } \widetilde{\Omega}, \quad  \widetilde{u}_-  = \widetilde{\varphi} \quad \text{on } \partial \widetilde{\Omega},
\end{align*}
respectively, where we extend $\mu_+$ and $\mu_-$ to be zero outside $\Omega$. By the comparison principle on $\widetilde{\Omega}$, we have
\[
\widetilde{u}_+ \leq \widetilde{\varphi} \leq \widetilde{u}_- \quad \text{in } \widetilde{\Omega},  \quad \text{and then} \quad 
\widetilde{u}_+ \leq u \leq \widetilde{u}_- \quad \text{in } \Omega.
\]
Thus,
\[
\widetilde{u}_+ - \widetilde{\varphi} \leq u - \widetilde{\varphi} \leq \widetilde{u}_- - \widetilde{\varphi} \quad \text{in } \Omega.
\]
Therefore, estimates of $\widetilde{u}_+ - \widetilde{\varphi}$ and $\widetilde{u}_-- \widetilde{\varphi}$ in $\Omega$ yield direct bounds on $u - \varphi$.

With these simplifications, the original global estimate $\|u-\varphi\|_{L^\infty(\Omega)}$ reduces to an interior estimate. By translation invariance, it suffices to bound $|u(0)-\varphi(0)|$ under suitable scale condition $B_{c_\varphi}(0) \subset \Omega \subset B_{C_\varphi}(0)$ for some positive constant $c_\varphi$ and $C_\varphi$.

\subsubsection{Reduction to one-sided boundary inequalities}\label{sec:onesided}
To establish a one-sided bound for $u$, the boundary condition $u = \varphi$ can be relaxed to the corresponding inequality.
Let $\hat{u}_+$ and $\hat{u}_-$ be the convex solutions to 
\begin{align*}
& \M \hat{u}_{+}=\M \varphi+\mu_+  \quad \text{in } {\Omega}, \quad \hat{u}_{+} =\varphi+\inf_{\partial \Omega}(u-\varphi)\quad \text{on } \partial \Omega, \\
& \M   \hat{u}_{-} =\M \varphi-\mu_- \quad \text{in } {\Omega}, \quad   \hat{u}_{-}  =\varphi+\sup_{\partial \Omega}(u-\varphi)\quad \text{on } \partial {\Omega},
\end{align*} 
respectively. The comparison principle yields
\[
\hat{u}_{+} \leq u  \leq \hat{u}_{-} .
\]
Therefore, estimates of $\hat{u}_{+}-\varphi$ and $\hat{u}_{-}- \varphi$ yield direct lower and upper bounds on $u - \varphi$.

This leads to several equivalent formulations of the Alexandrov-type estimates.
For example, for $n \geq 3$, the estimate \eqref{eq:abp gene 2/n} in Theorem~\ref{thm:ordera2} is equivalent to 
\begin{align}
\label{eq:abp inf 2/n} \inf_{\Omega} (u-\varphi) 
&\geq \inf_{\partial \Omega} (u-\varphi)-C\left( \int_{\Omega} \ud \mu_+\right)^{\frac{2}{n}}, \\
\label{eq:abp sup 2/n} \sup_{\Omega} (u-\varphi) 
&\leq \sup_{\partial \Omega} (u-\varphi)+C \left( \int_{\Omega} \ud \mu_-\right)^{\frac{2}{n}}.
\end{align} 
Similar inequalities hold with the right-hand sides replaced by $\left( \int_{\Omega} \ud \mu_{\pm}\right) \log\left(\int_{\Omega} \ud \mu_{\pm}\right)$ if $n=2$.

\section{The extremal configurations in Alexandrov estimates}\label{sec:alexandrov estimate classical}

Let $a\ge 0$ be a constant, and $\varphi\in C(\overline{\Omega})$ be a convex function in $\Omega$. Define \begin{equation}\label{eq:abp class Da}
\D_{a, \varphi} = \left\{ u \in C(\overline\Omega):    u  \text{ is convex},\   u=\varphi \text{ on } \partial \Omega, \ \mbox{and } |\M u - \M \varphi|(\Omega)\leq \omega_n a^n \right\}.
\end{equation}
Notably, $u \in \D_{a,\varphi}$ implies $\varphi \in \D_{a,u}$. The  Alexandrov estimate \eqref{eq:abp gene 1/n} implies that
\begin{equation}\label{eq:abp simplify}
\sup_{u \in \D_{a,\varphi}} \|u - \varphi\|_{L^{\infty}(\Omega)} \leq   \operatorname{diam}(\Omega)   a.
\end{equation}

Let us consider two  kinds of special functions in the space $\D_{a,\varphi}$ defined via:
\begin{itemize}
    \item[(i).] solutions of the isolated singularity problem
\begin{equation}\label{eq:defnuay}
\M u_{a}(\cdot, y) =\M\varphi+ \omega_na^n\delta_y  \quad\text{in }   \Omega, \quad u_{a}(\cdot, y) =\varphi \quad\text{on } \partial \Omega,
\end{equation}
where $\delta_y$ represents the Dirac measure centered at $y \in \Omega$.  The graphs of these two functions can be illustrated as follows:

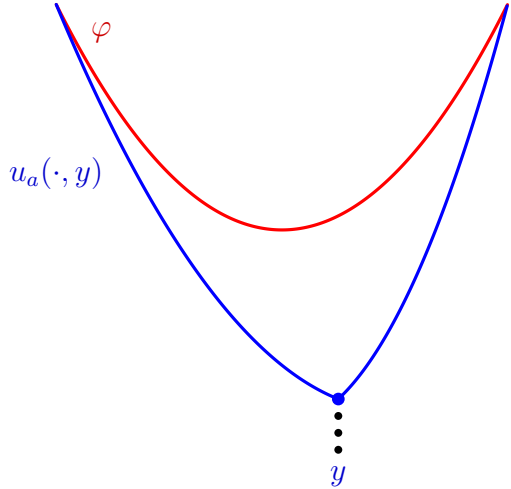
\begin{center}
\begin{tikzpicture}[scale=1.5, line cap=round, line join=round]
  \draw[very thick, red]
    plot[domain=-2:2, samples=200] (\x,{0.5*\x*\x});
  \node[red!80!black] at (-1.6,1.75) {$\varphi$};
\draw[very thick, blue]
    plot[domain=-2:0.5, samples=300] (\x,{0.4*(\x - 1)*(\x - 1) - 1.6});
\draw[very thick, blue]
  plot[domain=0.5:2, samples=300] (\x,{(14.0/15.0)*\x*\x - (26.0/15.0)});
  \node[blue!80!black] at (-2,0.5) {$u_a(\cdot,y)$};
\fill[blue] (0.5,-1.5) circle (1.6pt);
\foreach \k in {1,2,3} {
  \fill[black] (0.5,{-1.5 - 0.15*\k}) circle (1pt);
}
\node[blue!80!black, anchor=north] at (0.5,{-1.5 - 0.5}) {$y$};
\end{tikzpicture}
\end{center}

    \item[(ii).] solutions to the (hyperplane) obstacle problem of the form
\begin{equation}\label{eq:defnva}
\M v_{a}(\cdot, p) =\M\varphi \cdot \chi_{\left\{ v_{a}(\cdot, p)> \ell_{p}+h_{a,p}\right\}}   \quad\text{in }   \Omega, \quad v_{a}(\cdot, p) =\varphi \quad\text{on } \partial \Omega,
\end{equation}
where 
\begin{itemize}
\item $p \in \partial \varphi(\Omega)$ is a subgradient, 
\item $\ell_p(x) = \varphi(x_p) + p\cdot(x-x_p)$ is a support function of $\varphi$ at some $x_p \in \Omega$,  
\item the parameter $h_{a,p} \in [0, \inf_{\partial \Omega } (\varphi-\ell_p)]$  is chosen such that either $h_{a,p} = \inf_{\partial \Omega} (\varphi - \ell_p)$, or the corresponding solution $v_a(\cdot, p)$ satisfies
\begin{equation}\label{eq:defnvaq 2}
\omega_n a^n=|\M v_{a}(\cdot, p) - \M \varphi|(\Omega),
\end{equation}
\item $\chi_E$ is the characteristic function of a set $E$.  
\end{itemize}
The graph of the function $v_a(\cdot,p)$ can be illustrated as follows: 
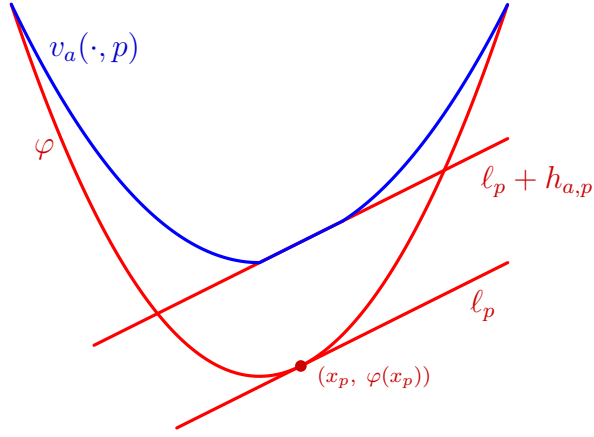
\begin{center}
\begin{tikzpicture}[scale=1.1, line cap=round, line join=round]
  % Function: \varphi(x) = x^2 / 2
  \draw[very thick, red]
    plot[domain=-3:3, samples=200] (\x,{0.5*\x*\x});

  % Label \varphi near the left side of the curve
  \node[red!80!black] at (-2.6,2.75) {$\varphi$};

  % Tangent at x_p = 0.5
  \def\xp{0.5}
  \pgfmathsetmacro{\yp}{0.5*\xp*\xp}  % \varphi(x_p)
  \pgfmathsetmacro{\m}{\xp}           % slope \varphi'(x_p) = x_p

  % Tangent line: y = yp + m (x - xp)
  \draw[very thick, red]
    plot[domain=-1:3, samples=2] (\x,{\yp + \m*(\x-\xp)});
  % Label for the tangent line \ell_p (below the line on the right)
  \node[red!80!black, below] at (2.7,{\yp + \m*(2.7-\xp) - 0.05}) {$\ell_p$};

  % Mark the tangency point
  \fill[red!80!black] (\xp,\yp) circle (2pt);
  % Smaller font for the tangent point label; slightly right and up
  \node[red!80!black, anchor=north west, xshift=2pt, yshift=4pt]
        at (\xp,\yp) {\scriptsize $(x_p,\;\varphi(x_p))$};

  % Parallel shifted line: \ell_p + 1.5, labeled as \ell_p + h_{a,p}
  \def\hoffset{1.5} % vertical offset
  \draw[very thick, red]
    plot[domain=-2:3, samples=2] (\x,{\yp + \m*(\x-\xp) + \hoffset});
  \node[red!80!black, anchor=north west, xshift=-5pt] at (2.7,{\yp + \m*(2.7-\xp) + \hoffset - 0.05})
        { $\ell_p + h_{a,p}$};
        
  % Blue function: (25/72) x^2 + 11/8 on x in [-3, 0]
  \draw[very thick, blue]
    plot[domain=-3:0, samples=200] (\x,{(25.0/72.0)*\x*\x + 11/8.0});
  \node[blue!80!black, anchor=north west, xshift=7pt]
        at (-2.9,{(25.0/72.0)*(-2.9)*(-2.9) + 11/8.0})
        {$v_a(\cdot, p)$};
  % NEW blue function: (1/3) x^2 - x/48 + 25/16 on x in [1, 3]
  \draw[very thick, blue]
    plot[domain=1:3, samples=200] (\x,{(1.0/3.0)*\x*\x - (1.0/48.0)*\x + 25/16.0});
  \node[blue!80!black, anchor=west, xshift=3pt]
        at (3,{(1.0/3.0)*3*3 - (1.0/48.0)*3 + 25/16.0})
        {};

  % NEW: Blue line y = 0.5*x + 11/8 on x in [0, 1]
  \draw[very thick, blue]
    plot[domain=0:1, samples=2] (\x,{0.5*\x + 11/8.0});
  \node[blue!80!black, anchor=west, xshift=2pt]
        at (1,{0.5*1 + 11/8.0})
        { };
        
\end{tikzpicture}
\end{center}

When $\M\varphi$ has unbounded density, the first equation in \eqref{eq:defnva} is interpreted as
\begin{equation}\label{eq:defnvaq 1}
\M v_{a}(\cdot, p) = \M \varphi \quad \text{on } \left\{ v_{a}(\cdot, p) >  \ell_{p}+h_{a,p}\right\},\quad  
\M v_{a}(\cdot, p) \leq \M \varphi .
\end{equation} 
For given $p \in \partial\varphi(\Omega)$ and $h_{a,p} \in [0, \inf_{\partial\Omega} (\varphi - \ell_p)]$, we will prove in Section~\ref{sec:obstacle problem rough} the existence and uniqueness of the solution to \eqref{eq:defnva}, and demonstrate that the mapping $h_{a,p} \mapsto a$ is a strictly increasing continuous bijection. The obstacle for \eqref{eq:defnva} is given by $\ell_p + h_{a,p}$, and $\{x:\;v_{a}(x, p) = \ell_p(x) + h_{a,p}\}$ is the coincidence set. 

\end{itemize}

We will show that they serve as extremal configurations for the Alexandrov estimate \eqref{eq:abp gene 1/n} in the following precise sense:
 
\begin{Theorem}\label{thm:abp extremal}
Let $a\ge 0$ be a constant, and $\varphi\in C(\overline{\Omega})$ be a convex function. 
For each $y \in \Omega$, we have
\begin{equation}\label{eq:pointwise ext lower}
\inf_{w \in \D_{a, \varphi} } w(y )=u_a(y ,y ),
\end{equation}
and
\begin{equation}\label{eq:pointwise ext upper}
\sup_{w\in \D_{a, \varphi} } w(y )=\sup_{p \in \partial \varphi (\Omega)}v_{a}(y ,p),
\end{equation}
where the supremum on the right-hand side of \eqref{eq:pointwise ext upper} is actually achieved by some $v_{a}(\cdot,p)$ with $p\in \partial \varphi (\Omega)$ and $y $ in its coincidence set.
\end{Theorem}

\begin{Remark}
It is possible that $\sup_{p \in \partial \varphi (\Omega)}v_{a}(y ,p)>\sup_{p \in \partial \varphi (y)}v_{a}(y ,p)$. See Example \ref{exam:supremum for obs}.
\end{Remark}

\subsection{The isolated singularity problem and Alexandrov estimates revisited}\label{sec:isolated singularity}

This subsection is devoted to establishing fundamental properties of the solutions $u_{a}(\cdot,y)$ to the isolated singularity problem \eqref{eq:defnuay}. By the comparison principle, we first have
\begin{equation}\label{eq:isp control by Linfinite}
\varphi(x) - (\varphi(y)-u_{a}(y, y)  ) \leq u_{a}(x, y) \leq \varphi(x), \quad  \forall  x\in \Omega,
\end{equation}

For the special case of $\varphi\equiv 0$, the function $u_{a}(\cdot, y)/a$ reduces to a convex cone function, that is, the solution to
\begin{equation}\label{eq:cone function}
\M w(\cdot, y) =\omega_n \delta_{y} \quad \text{in } \Omega,  \quad w(\cdot, y) =0 \quad \text{on } \partial \Omega.
\end{equation}

For general convex $\varphi\in C(\overline{\Omega})$, the function $\varphi(x) + a w(x,y)$ serves as a subsolution of \eqref{eq:defnuay}, which guarantees the existence of the solution $u_{a}(\cdot, y)$ to \eqref{eq:defnuay}. 
The comparison principle then yields
\begin{equation}\label{eq:isp control by cone}
\varphi(x) +aw(x,y) \leq u_{a}(x, y) \leq \varphi(x), \quad  \forall  x\in \Omega,
\end{equation} 
and
\[
\varphi(x) +aw(x,y) \leq u_{a}(x, y) \leq  \left\|\varphi\right\|_{L^{\infty}(\Omega)} + aw(x,y), \quad  \forall  x\in \Omega.
\] 
This implies that as $a \to \infty$, the rescaled function $u_a(\cdot,y)/a$ converges uniformly on $\overline{\Omega}$ to the cone function $w(\cdot,y)$. Consequently, we obtain the optimality of  the order $a$ in the right-hand side of \eqref{eq:abp simplify} for general functions $\varphi$.

\begin{Theorem}\label{thm:abp 1/n} 
Let $\varphi \in C(\overline{\Omega})$ be a given convex function. Let $\D_{a,\varphi}$ be defined in \eqref{eq:abp class Da}. 
For each $y \in \Omega$, let $w(\cdot,y)$ denote the cone function defined in \eqref{eq:cone function}. Then we have
\begin{equation}\label{eq:alexandrov estimate}
|u(y) - \varphi(y)| \leq  a |w(y,y)|,\quad \forall u \in \D_{a,\varphi}.
\end{equation}
Consequently, we obtain both the Alexandrov estimate  \eqref{eq:abp simplify} and
\begin{equation}\label{eq:alexandrov maximum principle}
|u(y)-\varphi(y) |\leq  C(n)\operatorname{diam}(\Omega)^{\frac{n-1}{n}}\operatorname{dist}(y ,\partial \Omega)^\frac{1}{n} a,  \quad \forall\ y \in \Omega.
\end{equation}
\end{Theorem}
\begin{proof}
Let $u_a(\cdot,y)$ be the one defined in \eqref{eq:defnuay}. For every $u\in \D_{a,\varphi}$ and $\varepsilon>0$, from \cite[Lemma 1.4.1]{gutierrez2016monge} or \cite[Lemma 2.7]{figalli2017monge}, we obtain 
$\M u_{a+\varepsilon}(E) \leq \M u(E)$ for the set $E = \{u < u_{a+\varepsilon}(\cdot,y)\}$. This inequality implies $y \notin E$, since otherwise, we have
\[
\M u_{a+\varepsilon}(E)\ge \M\varphi (E)+ \omega_n(a+ \varepsilon)^n > \M u(E).
\]
Therefore, $u(y) \geq u_{a+\varepsilon}(y,y)$.  Sending $\varepsilon\to 0$, we obtain 
\begin{equation}\label{eq:ubiggerua}
u(y)\ge u_a(y,y),
\end{equation}
and thus,
\[
u(y)-\varphi(y)\ge u_a(y,y)-\varphi(y) \ge a w(y,y),
\]
where the last inequality follows from \eqref{eq:isp control by cone}.
By interchanging the roles of $u$ and $\varphi$, this also yields $\varphi(y) -u(y)\geq aw(y,y)$.

By estimating the cone function $w(y,y)$ following Theorem 1.4.2 in \cite{gutierrez2016monge}, we derive both the inequalities \eqref{eq:abp simplify} and  \eqref{eq:alexandrov maximum principle}.
\end{proof}

\begin{Lemma}\label{lem:finitemodulus}
Let $u,\varphi \in C(\overline{\Omega})$ be two convex functions such that $\varphi = u$ on $\partial \Omega$. Suppose the positive part $\mu_+$ of the measure $\M u - \M \varphi$ satisfies $\mu_+(\Omega)<+\infty$. Then 
\begin{equation}\label{eq:holder continuity}
\omega_{u}(r) \leq \omega_{\varphi}(r) + C(n,\Omega) \mu_{+}(\Omega)^{\frac{1}{n}}r^{\frac{1}{n}},
\end{equation}
where $\omega_{u}(r)$ and $\omega_{\varphi}(r)$ denote the modulus of continuity of $u$ and $\varphi$ on $\overline{\Omega}$, respectively.
\end{Lemma}
\begin{proof}
Let us denote $a_+ = \omega_n^{-\frac{1}{n}} \mu_{+}(\Omega)^{\frac{1}{n}}$. The proof of \eqref{eq:ubiggerua} actually leads to
\[
u(y)\ge u_{a_+}(y,y).
\]
Therefore,
\[
\varphi(y) - u(y) \leq \varphi(y) - u_{a_+}(y,y)    \leq C\,  a_+ |w(y,y)| \leq  C \, a_+\operatorname{dist}(y ,\partial \Omega)^{\frac{1}{n}}.
\]
Since $\varphi = u$ on $\partial \Omega$,
it then follows that
\[
u(x) -u(y) =\varphi(x) -u(y)  \leq  \omega_{\varphi}(|x-y|) + C\, a_+ |x-y|^{\frac{1}{n}}, \quad \forall x\in \partial \Omega.
\] 
Together with the convexity of $u$, this implies the desired estimate for  $\omega_{u}(r)$ on $\overline{\Omega}$. 
\end{proof}

\begin{Remark}
The asymptotic behavior of $w(p, p)$ as $p \to \partial\Omega$ depends on the domain geometry; a detailed discussion can be found in \cite{griffin2023alexandrov}. In particular, if $\partial\Omega \in \C^{1,1}_+$, the exponent $\frac{1}{n}$ in \eqref{eq:holder continuity} can be improved to $\frac{1}{2} + \frac{1}{2n}$.
\end{Remark}

\begin{Proposition}\label{prop:finitetoinfinite}
Let $\varphi \in C(\overline{\Omega})$ be a convex function. Let $\mu$ be a locally finite Borel measure, and $\mu:=\mu_+-\mu_-$ be its Jordan decomposition. Suppose $\mu_{+}(\Omega) < \infty$ and the measure $\M \varphi + \mu$ is nonnegative. Then there exists a unique convex solution $u \in C(\overline{\Omega})$ to the Dirichlet problem
\begin{equation}\label{eq:dirichlet problem}
\M u = \M \varphi + \mu \quad \text{in } \Omega,\quad u = \varphi \quad \text{on } \partial \Omega.
\end{equation}
\end{Proposition} 
 \begin{proof}
For $\epsilon > 0$, define $\Omega_{\epsilon} = \{x \in \Omega : \operatorname{dist}(x, \partial \Omega) > \varepsilon\}$.
Since the measure $(\M\varphi+\mu)(\Omega_{\epsilon})$ is finite, it follows from \cite{gutierrez2016monge,hartenstine2006dirichlet} that there exists a unique convex solution $u_{\epsilon} \in C(\overline{\Omega_{\epsilon}})$ to
\[
\M u_{\epsilon} = \M \varphi + \mu \quad \text{in } \Omega_{\epsilon},\quad u_{\epsilon} = \varphi \quad \text{on } \partial \Omega_{\epsilon}.
\]
By the estimate  \eqref{eq:holder continuity}, the functions $u_{\epsilon}$ are uniformly continuous in $\overline{\Omega}_{\epsilon}$. Thus, as $\epsilon \to 0$, we may extract a subsequence that converges locally uniformly to a convex function $u \in C(\overline{\Omega})$ satisfying $u = \varphi$ on $\partial \Omega$.
\cite[Lemma 1.2.2]{gutierrez2016monge} then implies that $u$ solves the Dirichlet problem \eqref{eq:dirichlet problem}, and the uniqueness of the solution follows from the standard comparison principle given in \cite[Theorem 1.4.6]{gutierrez2016monge}. 
\end{proof}

\subsection{The obstacle problem}\label{sec:comparison principle}\label{sec:obstacle problem rough}
 
Next, we discuss some results concerning convex solutions to the (hyperplane) obstacle problem. This problem has been studied by Savin \cite{savin2005obstacle} and Huang-Tang-Wang \cite{huang2024regularity}, and subsequently extended in \cite{jin2025regularity}. The results in \cite{savin2005obstacle, huang2024regularity} remain valid if the obstacle $0$ is replaced by another linear function $\ell$. 

Let $\varphi\in C(\overline{\Omega})$ be a convex function. 
Let $p \in \partial \varphi(\Omega)$ be a subgradient, and let $\ell_p(x) = \varphi(x_p) + p\cdot(x-x_p)$ be a support function of $\varphi$ at some $x_p \in \Omega$.
For each $0 \leq h\leq   \inf_{\partial \Omega} \left\{ \varphi - \ell_{p}\right\}$,
consider the pointwise infimum 
\[
w_h(x) := \inf_{w \in  \E_{h}}w(x),  \quad x\in \Omega,
\]
where $\E_{h}$ denotes the class of admissible supersolutions:  
\[
\E_{h} = \left\{   w \geq  \ell_p+h\text{ is convex}:\;  \M w\leq  \M \varphi \text{ in } \Omega, \ \mbox{and }  w=\varphi \text{ on } \partial \Omega \right\}.
\] 
Following \cite{savin2005obstacle} (c.f. Proposition 2.2 and Lemma 2.3 in \cite{jin2025regularity}), the infimum $w_h$ is the unique convex solution to the obstacle problem:
\begin{equation*}\label{eq:obse ap 1}
\M w_{h}(\cdot, p) = \M \varphi(\cdot, p) \quad \text{on } \left\{ w_{h}(\cdot, p) >  \ell_{p}+h\right\},\quad  
\M w_{h}(\cdot, p) \leq \M \varphi
\end{equation*} 
under the Dirichlet boundary condition $w_{h}=\varphi$ on $\partial \Omega$.
For the case where $\M\varphi$ has a bounded density, this reduces to: 
\[
\M w_h= \M \varphi \cdot  \chi_{\{w_h> \ell_p + h\}}\quad \text{in } \Omega,\quad w_{h}=\varphi \quad \text{on } \partial \Omega.
\]

When $h=0$, the obstacle solution $w_0=\varphi$.
Applying the comparison principle (Lemma 2.3 in \cite{jin2025regularity}) yields
\begin{equation}\label{eq:xpinK}
\varphi \leq w_{h}  \leq \varphi+h \quad \text{in } \Omega ,\quad w_{h}(x_p)=\ell_p(x_p)+h=\varphi(x_p)+ h,
\end{equation}
and the following

\begin{Lemma}\label{lem:comparison principle va on h}
Assuming  $h_1 > h_2 \geq 0$, it follows that $\M w_{h_1} \not\equiv \M w_{h_2}$,  
\[
w_{h_1} \geq w_{h_2} \geq w_{h_1} - h_1 + h_2, \quad \left\{ w_{h_2} = h_2 \right\} \subset \left\{ w_{h_1} = h_1 \right\},\quad  \M w_{h_1} \leq \M w_{h_2}.
\]
Moreover, for $\M\varphi$ with bounded density, we have $\left\{ w_{h_2} = h_2 \right\} \subsetneq \left\{ w_{h_1} = h_1 \right\}$.
\end{Lemma}
\begin{proof}
By the uniqueness of solutions to the Monge-Amp\`ere equation, the relation $\M w_1 \equiv \M w_2$ would imply $w_{h_1} = w_{h_2}$, contradicting $h_1 > h_2$. 

By the comparison principle for the Monge-Ampère obstacle problem (Lemma 2.3 in \cite{jin2025regularity}), we have $w_{h_1} \geq w_{h_2} \geq w_{h_1} - h_1 + h_2$ and $\{w_{h_2} = h_2\} \subset \{w_{h_1} = h_1\}$.  Since $w_{h_2}$ touches $w_{h_1} - h_1 + h_2$ from above on $\{w_{h_2} = h_2\}$, it follows that $\M w_{h_1} \leq \M w_{h_2}$.
As $\M w_{h_1} \not\equiv \M w_{h_2}$,  we obtain $\left\{ w_{h_2} = h_2 \right\} \subsetneq \left\{ w_{h_1} = h_1 \right\}$ when $\M\varphi$ has bounded density. 
\end{proof}

\begin{Lemma}\label{lem:uniqueha}
Let us fix $p \in \partial \varphi(\Omega)$ and define $c_p\ge 0$ by
\[
\omega_nc_p^n=|\M w_{\inf_{\partial \Omega } (\varphi-\ell_p) }(\cdot,p) - \M \varphi|(\Omega).
\]
For any  $a \in [0,c_p] $, there exists a unique $h_{a,p}\in[0, \inf_{\partial \Omega } (\varphi-\ell_p)]$ such that \eqref{eq:defnvaq 2} holds. 
\end{Lemma}
\begin{proof}
For $h_1 > h_2 \geq 0$, we have $\M w_{h_1} \leq \M w_{h_2}$ with $\M w_{h_1} \not\equiv \M w_{h_2}$, which implies that $|\M w_{h}(\cdot,p) - \M \varphi|(\Omega)$ is strictly increasing in $h$. Noting that $w_{h_1} \geq w_{h_2} \geq w_{h_1} - h_1 + h_2$ and using the stability of Monge-Amp\`ere measures under uniform convergence of solutions, we conclude that the total variation $|\M w_h - \M \varphi|(\Omega)$ is also continuous in $h$. By the invertibility property of strictly increasing continuous functions,  for every $a \in [0,c_p]$, there exists a unique $h_{a,p}\in[0, \inf_{\partial \Omega } (\varphi-\ell_p)]$ such that \eqref{eq:defnvaq 2} holds.
\end{proof}

\subsection{Proof of Theorem \ref{thm:abp extremal} and further discussions}\label{sec:abp classical}

In this subsection, we first prove Theorem \ref{thm:abp extremal}, which provides a novel interpretation of \eqref{eq:abp gene 1/n}. We then conduct a further analysis of this interpretation, focusing on the extremal configurations of Alexandrov estimates that underlie its specific form. 

The proof of Theorem \ref{thm:abp extremal} relies on a nonstandard comparison principle. We state this key tool next, as it will be frequently employed in our subsequent analysis.
\begin{Lemma}\label{lem:comparison gene} 
Let $u, \varphi \in C(\overline{\Omega})$ be convex functions, and let $W\subset \Om$ be a closed set such that
\[
\M u(E) \leq \M \varphi(E) \quad \text{ for every Borel set } E \text{ satisfying } W\subset E \subset \Om.
\] 
Then
\[
\max_{x \in W}\{u(x)-\varphi(x)\} \geq \min _{x \in \partial \Omega}\{u(x)-\varphi (x)\} .
\]	
\end{Lemma}
\begin{proof}
Without loss of generality, we assume $\min _{x \in \partial \Omega}\{u(x)-\varphi (x)\}=0$. 

If $\M \varphi(W) =0$, our result follows from the comparison principle \cite[Theorem 1.4.6]{gutierrez2016monge}. 

If $\M \varphi(W) >0$, let us consider for each small $\epsilon>0$ the convex solution to
\[
\M \varphi_{\epsilon}= \M \varphi + \epsilon\ud x \quad \text{in } \Omega,\quad \varphi_{\epsilon}=\varphi  \quad \text{on }\partial \Omega.
\]
From \cite[Lemma 1.4.1]{gutierrez2016monge} or \cite[Lemma 2.7]{figalli2017monge}, we have $\M \varphi_{\epsilon}  \left(\left\{\varphi_{\epsilon} >u\right\}\right) \leq \M u  \left(\left\{\varphi_{\epsilon} >u\right\}\right) $. This implies that $W\not\subset \left\{\varphi_{\epsilon} >u\right\}$, since otherwise,
\[
\M \varphi_{\epsilon}  \left(\left\{\varphi_{\epsilon} >u\right\}\right) \leq \M u  \left(\left\{\varphi_{\epsilon} >u\right\}\right) \le \M \varphi\left(\left\{\varphi_{\epsilon} >u\right\}\right),
\]
which is a contradiction. This implies  $\max_{x \in W}\{u(x)-\varphi_{\epsilon}(x)\} \geq 0$.
We conclude the proof by sending $\epsilon\to 0 $.
\end{proof}

\begin{proof}[Proof of Theorem \ref{thm:abp extremal}.]
Let $u_a(\cdot,y)$ be the one defined in \eqref{eq:defnuay}. So $u_a(\cdot,y)\in \D_{a,\varphi}$. For every $w \in \D_{a,\varphi}$, applying Lemma \ref{lem:comparison gene} to the functions $w(\cdot), u_a(\cdot,y)$ and the singleton set $W_1 = \{y\}$ yields $w(y) \geq u_a(y,y)$. This proves \eqref{eq:pointwise ext lower}.

We now establish \eqref{eq:pointwise ext upper}. 
Consider $w\in \D_{a, \varphi}$ that satisfies $w(y)\ge\varphi(y)$. Let $p \in \partial w(y)$. Then $
p \in \partial \varphi (x_p) $ for some $x_p \in \Omega$. Let $\ell_p$ be a support function of $\varphi$ at $x_p$. Then $\ell_p+h$ is a support function of $w$ at $y$ for some $h \geq 0$. 
Consider the solution  $v_a(\cdot,p)$ defined in \eqref{eq:defnva}. If $a \geq c_p$, then we have $v_a (y,p) \geq \ell_p (y)+  \inf_{\partial\Omega}(\varphi-\ell_p) \geq  w(y)$. If $ a< c_p $, noting that 
\[
w \geq \ell_p+h =v_a +h-h_{a,p} \quad \text{on } W: =\{v_a=\ell_p+h_{a,p}\}.
\]
An application of Lemma \ref{lem:comparison gene} yields
\[
0 \leq \sup_{x \in W} \big\{v_a(x,p) - w(x)\big\} \leq h_{a,p} - h .
\]
This implies $h \leq h_{a,p}$, and thus, $w(y)=\ell_p(y)+h\le\ell_p(y)+h_{a,p}  \leq v_a(y,p)$. This proves \eqref{eq:pointwise ext upper}.

Next, we select a sequence of $p_i \in \partial \varphi(\Omega)$ such that $v_{a}(y,p_i)$ approaches the supremum in the right-hand side of \eqref{eq:pointwise ext upper}. Since these convex functions satisfy $v_a(\cdot,p_i) \geq \varphi$, they have uniform modulus of continuity up to the boundary. Therefore, we  can extract a subsequence converging to some $w_0 \in \D_{a,\varphi}$.
Applying the previous analysis for $w$ to $w_0$, we find a specific $p_0\in \partial w_0(y)\subset \partial \varphi(\Omega)$ such that $v_a(y,p_0) \geq w_0(y) =\sup_{p  \in \partial \varphi(\Omega)}v_{a}(y,p)\ge v_a(y,p_0)$.  Therefore,
\[
v_a(y,p_0)\ge \ell_{p_0}(y)+h_{a,p_0} \ge w_0(y)=v_a(y,p_0).
\]
Namely, $y \in \{x : v_a(x,p_0) = \ell_{p_0}(x) + h_{a,p_0}\}$ and $v_a(y,p_0) =\sup_{p  \in \partial \varphi(\Omega)}v_{a}(y,p)$.
\end{proof}

Note that there may hold $\sup_{p \in \partial \varphi (\Omega)}v_{a}(y ,p)>\sup_{p \in \partial \varphi (y)}v_{a}(y ,p)$,  as demonstrated by the following example.

\begin{ex}\label{exam:supremum for obs}
Let $\Omega = B_2(0)$ and define $\tilde{\varphi} = \frac{1}{2}(|x|^2 - 4)$.  
Consider the obstacle problem solutions $\tilde{v}_1(\cdot, \pm \varepsilon e_n)$ associated with $\tilde{\varphi}$, and let  $\ell_{\pm}$ denote the corresponding obstacle functions. 
Then, we have 
\[
\tilde{v}_1(\cdot,  \varepsilon e_n)> \tilde{v}_1(\cdot,- \varepsilon e_n) \quad \text{in } \{ x_n >0\}.
\]
Note that if $\varepsilon=0$, the coincidence set of $\tilde{v}_1(\cdot,  0)$ is $B_1$. Hence, we can fix $\varepsilon \in (0, \frac{1}{4})$ sufficiently small, so that $B_{1/2}(0) \subset \{\tilde{v}_1(\cdot, \pm \varepsilon e_n) = \ell_{\pm}\}$. Let $\tilde{y}=\frac{1}{8}e_n$. For a sufficiently large parameter $a > 0$, let $\varphi_a$ denote the solution to
\[
\M \varphi_a= \chi_{\{ \tilde{v}_1(\cdot,\varepsilon e_n)> \ell_+\}\cap \{\tilde{v}_1(\cdot,-\varepsilon e_n) >\ell_-\} } +\omega_na^n\delta_{\tilde{y} }  \quad \text{in } B_2(0),  \quad \varphi_a=0 \quad \text{on } \partial B_2(0).
\]
The function $\varphi_a$ possesses a cone singularity at $\tilde{y} $. As $a \to \infty$, the normalized function $a^{-1}\varphi_a$ converges uniformly to the cone function $w(\cdot,\tilde{y} )$. This convergence implies $\partial(a^{-1}\varphi_a(\tilde{y} )) \to \partial w(\tilde{y} ,\tilde{y} )$, and consequently $\pm\varepsilon e_n \in \partial\varphi_a(\tilde{y} )$.

Fix a large $a$. Consider a smooth approximation of $\omega_n a^n \delta_{\tilde{y} }$, denoted by $f$, with $\operatorname{supp} f \subset B_{1/2}(0)$ such that $\log f$ is bounded and positive in $B_{1/4}(0)$ and $\int f = \omega_n a^n$.
The perturbed solution $\varphi$, derived from $\varphi_a$, is almost a cone in $B_{1/2}(0)$, is $C^2$-regular and strictly convex near $\tilde y$. Hence, the  points $x_{\pm \varepsilon e_n}$ satisfying $\nabla\varphi(x_{\pm \varepsilon e_n}) = \pm \varepsilon e_n$ are distinct and close to  $\tilde{y}$.

For this $\varphi$, let us consider the solution $v_{\tilde a_{\pm}}(\cdot, \pm \varepsilon e_n)$ of \eqref{eq:defnva}, where $\tilde a_{\pm}\ge0$ are chosen so that the obstacles are given by $\ell_{\pm}$, respectively. Then we have $\tilde{v}_1(\cdot, \pm \varepsilon e_n) = v_{\tilde a_{\pm}}(\cdot, \pm \varepsilon e_n)$, and
\begin{align*}
\omega_n (\tilde a_{\pm})^n&= |\M \tilde{v}_1(\cdot, \pm \varepsilon e_n)-\M\varphi |(B_2) \\
&= \omega_n - |\{ \tilde{v}_1(\cdot,\varepsilon e_n)= \ell_+\}\cap \{\tilde{v}_1(\cdot,-\varepsilon e_n) =\ell_-\} |+\omega_n a^n.
\end{align*}
Therefore, $\tilde a_+=\tilde a_-$, and denote them as $\tilde a$. With this construction, 
\[
v_{\tilde a}(y,  \varepsilon e_n ) > v_{\tilde a}(y,-\varepsilon e_n)\quad \text{for }y=x_{-\varepsilon e_n} \in \{x_n>0\}.
\]
\end{ex}

We will use the following geometric estimate for the coincidence set.

\begin{Lemma}\label{lem:coincidencesetcontrol}
Assume that $\varphi\in(\overline\Omega)$ is convex and satisfies the non-degeneracy condition \eqref{eq:varphi equation} in $\Omega$. Let $v_a(\cdot,p)$ be the solution to the obstacle problem \eqref{eq:defnva}, where $p \in \partial \varphi(x_p)$ for some $x_p \in \Omega$. Set
\[
K_{a,p} := \{ x \in \Omega : v_a(x,p) = \ell_p(x) +h_{a,p}\},
\]
and let $y \in K_{a,p}$. Then there exist constants $C_1 = C_1(n,\lambda,\Lambda)$ and $C_2 = C_2(n,\lambda,\Lambda)$ with $C_1 > C_2 > 0$ such that, if $S_{C_1 a^2,q}^{\varphi}(y)   \subset \Omega$ for some $q\in\partial\varphi(y)$, we have
\begin{equation}\label{eq:coincidence inclusion}
K_{a,p} \subset S_{C_2 a^2,p}^{\varphi}(x_p) \subset S_{C_1 a^2,q}^{\varphi}(y).
\end{equation} 
\end{Lemma}
\begin{proof}
Note that if $S_{C_1 a^2,q}^{\varphi}(y)   \subset \Omega$, then, since $\varphi$ satisfies the non-degeneracy condition \eqref{eq:varphi equation} in $\Omega$, $\varphi$ is locally $C^{1,\alpha}$ in $S_{C_1 a^2,q}^{\varphi}(y)$. Therefore, in the proof below we drop the subscript from the section notation, since the subgradient at each point is unique.

The constants $C_2$ and $C_1/C_2$ will be large and determined in the end. The following argument holds if $S_{C_1 a^2}^{\varphi}(y)   \subset \Omega$.

We consider two cases:

Case 1: $h_{a,p} \geq C_2 a^2$. Then $S_{C_2 a^2}^{\varphi}(x_p) \subset S_{h_{a,p}}^{\varphi}(x_p) \subset \Omega$. For simplicity, assume $x_p = 0$ and $\ell_p +h_{a,p}= 0$. 
Since $\M\varphi(K_{a,p})=\omega_n a^n$, after a unimodular affine transformation, we may assume  
\[
B_{c a}(0) + m_{a,p} \subset K_{a,p} \subset B_{C a}(0),
\]
where $m_{a,p}$ denotes the mass center of $K_{a,p}$. Note that $\varphi + h_{a,p}$ touches $v_a(\cdot,p)$ from above at $x_p$ by \eqref{eq:xpinK}.  Applying \cite[Lemma 2]{savin2005obstacle}, we find that
\[
\varphi(x) + h_{a,p} \geq v_a (x, p) \geq c \operatorname{dist}(x,K_{a,p})^2 \quad \text{in } B_{2C a}(0),
\]
which implies $S_{c a^2}^{\varphi}(0) = S_{c a^2}^{\varphi+h_{a,p}}(0) \subset B_{2 C a}(0)$.
By the regularity theory in \cite{caffarelli1990ilocalization}, the section $S_{c a^2}^{\varphi}(0)$ is balanced around $0$ and satisfies $|S_{c a^2}^{\varphi}(0)| \approx a^n$. Combining this with $S_{c a^2}^{\varphi}(0) \subset B_{2 C a}(0)$, we obtain
\[
B_{c_1 a}(0) \subset S_{c a^2}^{\varphi}(0) \subset B_{2 C a}(0)
\]
for some $c_1=c_1(n,\lambda,\Lambda)>0$.
Consequently, $K_{a,p} \subset C_3  S_{c_1 a^2}^{\varphi}(0)$ for a sufficiently large $C_3$. Moreover, applying \cite{caffarelli1990ilocalization}, we further deduce that $C_3S_{c_1 a^2}^{\varphi}(0) \subset S_{C_2 a^2}^{\varphi}(0)$ for a suitably large ratio $C_2/C_3$. This establishes that $K_{a,p} \subset S_{C_2 a^2}^{\varphi}(x_p)$.

Case 2: $h_{a,p} \leq C_2 a^2$. Note that $S_{h_{a,p}}^{\varphi}(x_p) := \{\varphi \leq \ell_p + h_{a,p}\} \subset \Omega$. Then it is clear that $K_{a,p} \subset S_{h_{a,p}}^{\varphi}(x_p) \subset S_{C_2 a^2}^{\varphi}(x_p)$.

In each case, we have established the first inclusion. Therefore, $y\in K_{a,p}  \subset S_{C_2 a^2}^{\varphi}(x_p)$. Then the second inclusion $S_{C_2 a^2}^{\varphi}(x_p) \subset S_{C_1 a^2}^{\varphi}(y)$ follows from the engulfing property of sections provided that $C_1/C_2$ is large (see \cite{caffarelli1990ilocalization,gutierrez2016monge}).
\end{proof}
 
\section{An optimal Alexandrov estimate}\label{sec:alexandrov estimate affine}
 
When $\partial \Omega \in \C_+^{3,\alpha}$ and $u, \varphi \in \C_+^{3,\alpha}(\overline{\Omega})$, we can use a divergence identity that relates the difference of their Monge-Amp\`ere measures to the boundary data. Indeed, consider the divergence-form equation
\begin{equation}\label{eq:div equation}
\partial_i \big( A^{ij} \partial_j (u - \varphi) \big)
= A^{ij} \partial_{ij} (u - \varphi)
= \det D^2 u - \det D^2 \varphi
=: \mu \quad \text{in } \Omega,
\end{equation}
where the coefficient matrix ${ A^{ij} } = \int_0^1 \operatorname{Cof} \left( t D^2 u + (1-t) D^2 \varphi \right) dt$ is defined via the cofactor matrix operator $\operatorname{Cof}$.
Then, for any $C^1$ subdomain $\Omega_2 \subset \Omega$, integration by parts yields
\begin{equation}\label{eq:div identity smooth}
\int_{\partial \Omega_2} A^{ij} \partial_i (u - \varphi)\, \nu_j \, dS \;=\; \mu(\Omega_2),
\end{equation}
where $\nu = (\nu_1,\dots,\nu_n)$ is the outward unit normal to $\partial \Omega_2$.

In general, we may invoke \eqref{eq:div identity smooth} via a standard smooth-approximation (see below) whenever the resulting estimates depend only on structural properties and not on the precise smoothness of the domain or the functions. 

Let $\{\Omega_\epsilon\}_{\epsilon>0}$ be a family of $\C_+^{3,\alpha}$ convex domains that are compactly contained in $\Omega$ and exhaust $\Omega$ as $\epsilon \to 0$.
Let $\varphi_\epsilon \in \mathcal{C}_+^{3,\alpha}(\overline{\Omega})$ converge uniformly to $\varphi$ in $\overline{\Omega}$, and assume that $\M\varphi_\epsilon + \mu_\epsilon \in C^{1,\alpha}(\overline{\Omega})$ converges weakly in $\Omega$ to $\M\varphi + \mu$, where, if necessary, $\mu_\epsilon$ is chosen to satisfy the same structural assumptions as $\mu$.

Suppose $u_\epsilon \in C(\overline{\Omega_\epsilon})$ solves
\[
\M u_{\epsilon} =\M \varphi_{\epsilon}+\mu_{\epsilon}  \quad \text{in } \Omega_{\epsilon},\quad u_{\epsilon}= \varphi_{\epsilon} \quad  \text{on } \partial \Omega_{\epsilon}.
\]
By the global regularity results in \cite{caffarelli1984dirichlet,trudinger2008boundary}, we have $u_\epsilon \in \mathcal{C}_+^{3,\alpha}(\overline{\Omega_\epsilon})$. As $\epsilon \to 0$, the sequence $u_\epsilon$ converges locally uniformly to a function $u$ in $\Omega$. The desired estimates for $u - \varphi$ then follow by deriving estimates for $u_\epsilon - \varphi_\epsilon$ that are uniform in $\epsilon$ and passing to the limit. This smooth-approximation scheme justifies both the regularity assumptions and the use of the divergence identity \eqref{eq:div identity smooth} in the general setting.

\subsection{Proof of estimate \eqref{eq:abp gene 2/n}}\label{subsection:2/n}
 
As outlined in Section \ref{app:comparison reduction}, the comparison principle allows us to reduce the proof of estimate \eqref{eq:abp gene 2/n} to the cases in which $\mu = \M u - \M\varphi$ is either nonnegative or nonpositive. Moreover, it suffices to establish the interior estimate at the origin, so we may restrict to the following setting:
\begin{Assumption**}
Let $\Omega \subset \R^n$ ($n \geq 2$) be a convex domain with $0 \in \Omega \subset B_{4n}(0)$. Let $\varphi \in C(\overline{\Omega})$ be a convex function satisfying the non-degeneracy condition \eqref{eq:varphi equation} in $\Omega$ 
and
\[
0 \in \{ \varphi <0\} \subset \Omega,\quad  \inf_{\Omega} \varphi=\varphi (0) = -1.
\]  
The sections of $\varphi$ at $0$ are defined as
\[
S_h := S_h^\varphi(0) = \{ x \in \overline{\Omega} : \varphi(x) < - 1+h  \}.
\]
\end{Assumption**}

The assumption (G) implies that $B_c(0) \subset S_1 :=\{ \varphi < 0 \}\subset B_{4n}(0)$ holds for some $c(n,\lambda,\Lambda)>0$.

We will estimate $|u(0)-\varphi(0)|$ under the assumption (G).
The following lemma is inspired by \cite[Lemma 3.2]{tian2008classsobolev}.

\begin{Lemma}\label{lem:divergence relation}
Assume (G).  Let $u \in C(\overline{\Omega})$ be convex and satisfy $u=\varphi$ on $\partial\Omega$. Suppose
\[
\operatorname{supp}(\mu) \subset S_{1/16}, 
\]
where $\mu := \M u - \M \varphi$.
Then the following estimate holds:
\[
|u-\varphi| \;\le\; C(n,\lambda,\Lambda)\, |\mu|(\Omega)
\quad \text{in } \Omega \setminus S_{1/2}.
\]
\end{Lemma}

\begin{proof} 
The classical Alexandrov estimate \eqref{eq:abp gene 1/n} implies 
\[
\|u - \varphi\|_{L^\infty(\Omega)} \leq C|\mu|(\Omega)^{\frac{1}{n}}.
\]
This bound allows us to focus on the case in which the total variation $|\mu|(\Omega)$ is sufficiently small, ensuring the difference between $u$ and $\varphi$ remains sufficiently small.  

By the comparison principle and the reduction in Section \ref{sec:signed}, we may assume without loss of generality that $\mu \geq 0$ (the case $\mu \leq 0$ being analogous), and hence 
\[
u \leq \varphi \quad \text{in }\Omega.
\]
Given that $\varphi - u = 0$ on $\partial \Omega$, the comparison principle yields
\[
\varphi-u \leq  \sup_{\partial S_{1/2}} (\varphi-u  ) \quad \text{in } \Omega \setminus S_{1/2}.
\]
Our main objective reduces to establishing the estimate
\[
\varphi - u \leq C\mu(\Omega)\quad \text{on }\partial S_{1/2}.
\]

Since $\varphi$ is normalized in $S_1 \subset B_{4n}(0)$ and the difference between $u$ and $\varphi$ is small,  $u$ remains strictly convex in the annular region $S_{7/8} \setminus S_{1/8}$. 
This strict convexity allows us to apply the Harnack inequality established in \cite{caffarelli1997properties} to subsolutions and supersolutions of the linearized Monge–Ampère equation. Specifically, by linearizing the equation around $u$ and $\varphi$, we obtain the following:
\[
L_{u}(\varphi-u)\geq 0,\quad    L_{\varphi}(\varphi- u) \leq 0 \quad \text{in }  \Omega\setminus S_{1/8},
\]
which implies, by \cite{caffarelli1997properties}, the existence of a constant $\kappa > 0$ such that
\[
\varphi-u  \approx \kappa  \quad \text{in } S_{3/4} \setminus S_{1/4}.
\]  
It therefore suffices to show that
\[
\kappa \leq C(n,\lambda,\Lambda) \mu(\Omega).
\]

By approximation, we may assume $\partial \Omega \in \C_+^{3,\alpha}$ and $u, \varphi \in \C_+^{3,\alpha}(\overline{\Omega})$.
Since $\varphi$ is normalized in $S_1 \subset B_{4n}(0)$ and satisfies $\lambda \leq \det D^2\varphi \leq \Lambda$, 
the annulus $S_{3/4} \setminus S_{1/4}$ is non-degenerate and bounded away from the origin. 
We can therefore construct a convex domain $\Omega_0$ such that $\partial\Omega_0 \subset S_{3/4} \setminus S_{1/4}$ 
and $\|\partial\Omega_0\|_{\C_+^{1,1}} \leq C(n,\lambda,\Lambda)$.
Let $w$ solve
\begin{align*}
L[w] &= A^{ij} w_{ij} = (A^{ij} w_i)_j = 0 \quad \text{in } \Omega \setminus \Omega_0,\\
w &= 0 \quad \text{on } \partial\Omega, \quad w = -c\kappa \geq u - \varphi \quad \text{on } \partial\Omega_0,
\end{align*}
where the coefficient matrix ${ A^{ij} } = \int_0^1 \operatorname{Cof} \left( t D^2 u + (1-t) D^2 \varphi \right) dt$ is defined via the cofactor matrix operator $\operatorname{Cof}$. 
Observing that $L[u - \varphi] = \det D^2 u - \det D^2 \varphi = 0$ in $\Omega \setminus \Omega_0$, the comparison principle implies $w \geq u - \varphi$. Combining this with the divergence formula and the boundary condition $w = u - \varphi = 0$ on $\partial \Omega$, we conclude that
\[
\int_{\partial \Omega_0}  A^{ij}(x) w_i \nu_j= \int_{\partial \Omega}  A^{ij}(x) w_i \nu_j  
\leq \int_{\partial \Omega}  A^{ij}(x) (u-\varphi)_i \nu_j \overset{\eqref{eq:div identity smooth}}{=}\mu(\Omega),
\]
where $\nu$ denote the outer normal on $\partial \Omega$ and $\partial  \Omega_0$, respectively.

We now define the barrier function $\psi$ in terms of the domain $\Omega_0$ constructed above. Let $\psi \in \C_+^{1,1}(B_{4n}(0))$ be the convex function obtained via a homogeneous extension of degree 2 and a suitable constant adjustment, such that 
\[
\Omega_0=\{\psi <  0\}, \quad  
\|\psi\|_{\C_+^{1,1}(B_{4n}(0))} \le C(n,\lambda,\Lambda), \quad  \text{and} \quad \psi \le 1 \quad \text{in }  \Omega\subset B_{4n}(0).
\]  
Since $\psi$ is convex, we observe that 
\[
L[c\kappa \psi]  \geq 0 = L[w+c\kappa]  \quad \text{in } \Omega \setminus \Omega_0.
\]
Together with the boundary conditions, the comparison principle implies
\[
c\kappa \psi \leq w+c\kappa  \quad \text{in } \Omega\setminus \Omega_0.
\]
Then, it follows from $c\kappa \psi=0= w+c\kappa$ on $\partial \Omega_0$ that
\begingroup
\allowbreak
\begin{align*}
\int_{\partial \Omega_0}  A^{ij}(x) w_i \nu_j 
& \geq c\kappa\int_{\partial \Omega_0}  A^{ij}(x) \psi_i   \nu_j\\
&=c\kappa\int_{\Omega_0}(A^{ij}(x) \psi_i )_j\\
&=c\kappa\int_{\Omega_0}A^{ij}(x) \psi_{ij}\\
& \geq  \frac{c\kappa}{n}\int_{\Omega_0} \left(\det A^{ij} \cdot \det D^2 \psi\right)^{\frac{1}{n}} \\
&\geq  c(n,\lambda)|\Omega_0|\kappa\\
&\geq c(n,\lambda,\Lambda) \kappa,
\end{align*}
\endgroup
where we used that $ A^{ij} \geq \int_{0}^1 \operatorname{Cof}\left(tD^2{\varphi} \right)  \ud t  = \frac{1}{n}\operatorname{Cof}D^2{\varphi} $.
We therefore conclude that
\[
\kappa  \leq C(n,\lambda,\Lambda)\int_{\partial \Omega_0}  A^{ij}(x) w_i \nu_j  \leq C(n,\lambda,\Lambda)\mu(\Omega) .
\]
\end{proof}

In Lemmas \ref{lem:interior estimate u} and \ref{lem:interior estimate v} below, we establish the decay estimate $a^n h^{(n-2)/2}$ in $\Omega \setminus S_h^{\varphi}$ for $h \geq C a^2$.

\begin{Lemma}\label{lem:interior estimate u}
Assume (G). 
Let $u_a(x) := u_a(x,0)$ denote the solution to the isolated singularity problem \eqref{eq:defnuay} with $y=0$.  
If $a \leq \frac{1}{4}$, then for every $h\in(0,1/2)$, the following holds:
\begin{itemize}
\item if $n=2$, then
\begin{equation}\label{eq:ua-vaphi n=2}
\varphi(x) \geq u_{a}(x)  \geq
\varphi(x) -C(n,\lambda,\Lambda)  \min\left\{ |\log h|+1, |\log a| \right\}  a^2  \quad  \text{in }\Omega \setminus S_h.
\end{equation}  
\item if $n \geq 3$, then
\begin{equation}\label{eq:ua-vaphi}
\varphi(x) \geq u_{a}(x)  \geq
\varphi(x) -C(n,\lambda,\Lambda)\min\left\{ a^{n} h^{-\frac{n-2}{2}}, a^2 \right\}\quad  \text{in }\Omega \setminus S_h.
\end{equation}  
\end{itemize}
\end{Lemma}
\begin{proof}
In view of the comparison principle, it suffices to estimate the following quantity:
\[
\bar{t}(h):=\sup_{\partial S_h} (\varphi-u_a)  =\sup_{\Omega \setminus S_h} (\varphi-u_a)  .
\] 

Lemma \ref{lem:divergence relation} implies the existence of a structural constant $C(n,\lambda,\Lambda) > 1$ such that   
\[
\bar{t}(1/2)=\sup_{\Omega \setminus S_{1/2}} ( \varphi -u_a) \leq C a^n.
\] 
Furthermore, since $u_a + \bar{t}(2h) \geq \varphi$ on $\partial S_{2h}$, 
we apply a variant of Lemma~\ref{lem:divergence relation}—obtained by combining it with the comparison principle—to the normalization of $\varphi$ and $u_a + \bar{t}(2h) $ over the section $S_{2h}$, as defined in \eqref{eq:simplify normalization 1} and \eqref{eq:simplify normalization 2}, to obtain
\[
\varphi - (u_a+\bar{t}(2h)) \leq  C a^nh^{-\frac{n-2}{2}} \quad \text{in } S_{2h}\setminus S_{h},
\]
which yields that
\[
\bar{t}(h)-\bar{t}(2h)  \leq  C a^nh^{-\frac{n-2}{2}}.
\]
Summing over $h_k=2^{-k}$ for $k=1,2,\dots$, we obtain,
\[
 \bar{t}(h_k) \leq     
\begin{cases}
Ca^2 k & \text{if }  n=2,\\
Ca^n2^{\frac{(n-2)k}{2}}  & \text{if }  n \geq 3.
\end{cases}   
\]  
Since $\bar{t}(h)$ is monotonically decreasing in $h$, this yields in $\Omega\setminus S_h$ that
\[
\varphi-u_{a} \leq \bar{t}(h) \leq     
\begin{cases}
Ca^2(|\log  h|+1)& \text{if }  n=2,\\
Ca^nh^{-\frac{n-2}{2}}  & \text{if }  n \geq 3.
\end{cases}   
\] 

In addition, we apply the rescaled classical Alexandrov estimate \eqref{eq:scaledA} in $S_{2a^2}$ and combine it with the comparison principle to obtain
\[
\varphi - u_a \leq \bar{t}(2a^2)+ Ca^2 \leq     
\begin{cases}
Ca^2|\log  a|& \text{if }  n=2\\
Ca^2  & \text{if }  n \geq 3
\end{cases}   
\quad \text{in } S_{2a^2}.
\]

In conclusion, we obtain  \eqref{eq:ua-vaphi n=2} and \eqref{eq:ua-vaphi}.
\end{proof}

\begin{Lemma}\label{lem:interior estimate v} 
Assume (G). Let $v_a(x) := v_a(x,p)$ denote the solution to the obstacle problem \eqref{eq:defnva}, where $p\in\partial\varphi(\Omega)$. Suppose the coincidence set $K_a$ of $v_a$ contains the origin. If $a \leq \frac{1}{4}$, then for every $h\in(0,1/2)$,  
the following holds:
\begin{itemize}
\item if $n=2$, then
\begin{equation}\label{eq:va-vaphi n=2}
\varphi(x) \leq v_a(x)  \leq
\varphi(x) +C(n,\lambda,\Lambda)  \min\left\{ |\log h|+1, |\log a| \right\}  a^2  \quad  \text{in }\Omega \setminus S_h.
\end{equation}  
\item if $n \geq 3$, then
\begin{equation}\label{eq:va-vaphi}
\varphi(x) \leq v_{a}(x)  \leq
\varphi(x) +C(n,\lambda,\Lambda)\min\left\{ a^{n} h^{-\frac{n-2}{2}}, a^2 \right\}\quad  \text{in }\Omega \setminus S_h.
\end{equation}  
\end{itemize}
\end{Lemma}
\begin{proof} 
By the classical Alexandrov estimate \eqref{eq:abp gene 1/n}, it suffices to consider the case of small $a$. The argument parallels that of Lemma~\ref{lem:interior estimate u}. The only difference is that, invoking \eqref{eq:coincidence inclusion}, we obtain $K_a \subset S_{C a^2}$, which in turn justifies the application of Lemma~\ref{lem:divergence relation}.

Let
\[
\hat{t}(h):=\sup_{\partial S_h} (v_a- \varphi)  =\sup_{\Omega \setminus S_h} (v_a-\varphi).
\] 
Lemma \ref{lem:divergence relation} implies the existence of a structural constant $C(n,\lambda,\Lambda) > 1$ such that   
\[
\hat{t}(1/2)=\sup_{\Omega \setminus S_{1/2}} ( v_a-\varphi  ) \leq C a^n.
\] 
Assume $h \geq Ca^2$. 
Since $\varphi \geq v_a-\hat{t}(2h)$ on $\partial S_{2h}$ and $K_a \subset S_{Ca^2} \subset S_h$,  
we apply a variant of Lemma~\ref{lem:divergence relation}—obtained by combining it with the comparison principle—to the normalization of $\varphi$ and $v_a- \hat{t}(2h)$ over the section $S_{2h}$, as defined in \eqref{eq:simplify normalization 1} and \eqref{eq:simplify normalization 2}, to obtain 
\[
v_a-\hat{t}(2h)- \varphi \leq  C a^nh^{-\frac{n-2}{2}} \quad \text{in } S_{2h}\setminus S_{h},
\]
which yields that
\[
\hat{t}(h)-\hat{t}(2h)  \leq  C a^nh^{-\frac{n-2}{2}}.
\]
Summing over $h_k=2^{-k} \geq Ca^2$ for $k=1,2,\dots$, we obtain,
\[
 \hat{t}(h_k) \leq     
\begin{cases}
Ca^2 k & \text{if }  n=2,\\
Ca^n2^{\frac{(n-2)k}{2}}  & \text{if }  n \geq 3.
\end{cases}   
\]  
Since $\hat{t}(h)$ is monotonically decreasing in $h$, this yields in $\Omega\setminus S_h$ that
\[
v_{a} -\varphi \leq \hat{t}(h) \leq     
\begin{cases}
Ca^2(|\log  h|+1)& \text{if }  n=2,\\
Ca^nh^{-\frac{n-2}{2}}  & \text{if }  n \geq 3.
\end{cases}   
\] 

In addition, we apply the rescaled classical Alexandrov estimate \eqref{eq:scaledA} in $S_{Ca^2}$ and combine it with the comparison principle to obtain
\[
v_a- \varphi  \leq \hat{t}(Ca^2)+ Ca^2 \leq     
\begin{cases}
Ca^2|\log  a|& \text{if }  n=2\\
Ca^2  & \text{if }  n \geq 3
\end{cases}   
\quad \text{in } S_{Ca^2}.
\]

In conclusion, we obtain \eqref {eq:va-vaphi n=2} and \eqref{eq:va-vaphi}.
\end{proof}

These two lemmas lead to the following interior estimate.
 
\begin{Theorem}\label{thm:localized abp}
Let $\Omega \subset \R^n$, $n \geq 2$, be a convex domain satisfying $B_1(0) \subset \Omega \subset B_n(0)$, and let $\varphi \in C(\overline{\Omega})$ be a strictly convex function satisfying \eqref{eq:varphi equation} in $\Omega$. Let $u \in C(\overline{\Omega})$ be convex and satisfy $u=\varphi$ on $\partial\Omega$. Denote $\mu := \M u - \M \varphi$ and define $a = \omega_n^{-1/n} |\mu|(\Omega)^{1/n}$. Then for any open set $\Omega_1 \subset\subset \Omega$, we have
\[
\left\| u-\varphi\right\|_{L^{\infty}(\Omega_1)}\leq
 \begin{cases}
C_1 a^2 \left( \left|\log a\right| + 1 \right) & \text{if } n = 2, \\
C_1 a^2 & \text{if } n \geq 3,
\end{cases}
\]
where $C_1$ depends only on $n$, $\lambda$, $\Lambda$ and $\displaystyle \inf_{p \in \partial \varphi(\Omega_1)} \inf_{x\in \partial \Omega} (\varphi(x) - \ell_p(x))$.
\end{Theorem}
\begin{proof}
Let $y \in \Omega_1$, and let $\ell$ be a support function of $\varphi$ at $y$, and set 
\[
b= \min_{p \in \partial \varphi (\Omega_1) } \min_{x\in \partial  \Omega} (\varphi(x) - \ell_p(x))  \leq C(n,\lambda,\Lambda).
\]
Define the translated domain $\Omega_y = \Omega - y$ and  denote
\[
\varphi_y(x)=\frac{\varphi(x+ y) - \ell (x+ y)}{b},  \quad u_y(x)=\frac{u(x+ y) - \ell (x+ y)}{b}.
\]  
Under this transformation, $\varphi_y$ and $\Omega_y$ satisfy Assumption~(G)  with modified ellipticity constants $\lambda_b = b^{-n}\Lambda$ and $\Lambda_b = b^{-n}\lambda$, depending only on $n$, $b$, $\lambda$, and $\Lambda$. The desired conclusion then follows by combining Theorem~\ref{thm:abp extremal} with the interior estimates established in Lemmas~\ref{lem:interior estimate u} and~\ref{lem:interior estimate v}, applied to $\varphi_y$. Specifically, Theorem~\ref{thm:abp extremal} yields the two-sided bound
\[
u_a(\cdot,0)- \varphi_y(0) \;\leq\; u_y(0) - \varphi_y(0) \;\leq\; v_a(\cdot,p)-\varphi_y(0).
\]
The conclusion follows from Lemmas~\ref{lem:interior estimate u} and~\ref{lem:interior estimate v} by sending $h\to 0^+$, respectively. 
\end{proof}

\begin{proof}[Proof of \eqref{eq:abp gene 2/n} in Theorem~\ref{thm:ordera2}]
Since $\varphi$ is defined on a larger domain $\widetilde{\Omega}\supset\supset \Omega$, an extension argument analogous to the one used in the reduction outlined in Section \ref{app:comparison reduction}, combined with the interior estimates established in Theorem~\ref{thm:localized abp} applied on $\widetilde{\Omega}$, yields the desired estimate in $\Omega$.
\end{proof}

\subsection{Proof of estimate \eqref{eq:abp gene 2/n log}}

\begin{Lemma}\label{lem:varphi sc from u}
Let $\Omega \subset \R^n$ be a convex domain, $a\ge0$, and let $ \varphi \in C(\overline{\Omega})$ be a convex function satisfying \eqref{eq:varphi equation} in $\Omega$. Let $\tilde{u} $ be a convex function satisfying
\[
\M \tilde{u} =\M \varphi+\omega_na^n \delta_y \quad \text{in } \Omega.
\] 
Suppose $\tilde{u} \ge \varphi$ and that $\tilde{u}$ touches $\varphi$ from above at $y \in \Omega$.
Let $p\in\partial \varphi(y)$. 
Then there exist positive constants $C$ and $\widetilde{C} $, both of which depend only on $n$, $\lambda$ and $\Lambda$, such that if
$h \geq Ca^2$ and $S_{h,p}^{\tilde{u}}(y)   \subset \Omega$,  we have $S_{h/2,p}^\varphi(y)  \subset S_{h,p}^{\tilde{u}}(y)$ and 
\[
\varphi \leq \tilde{u} \leq \varphi+ 
\begin{cases}
\widetilde{C}  a^2 \left( \left|\log \frac{a}{h^{1/2}}\right| + 1 \right) & \text{if } n = 2, \\
\widetilde{C}  a^2 & \text{if } n \geq 3,
\end{cases}
\quad 
\text{in } S_{h/2,p}^\varphi(y).
\]
Furthermore, $\varphi$ is strictly convex and  $C^{1,\alpha}$ in $S_{h/2,p}^\varphi(y)$, and thus, $p=\nabla \varphi(y)$. 
\end{Lemma}
\begin{proof}
For notational simplicity, we assume $y=0$, $\varphi(0)=0$, $p=0$,  and write
\[
S_h^{\tilde{u}} := \{ \tilde{u} < h\}\quad \text{and}\quad S_{h}^\varphi := \{\varphi < h\} \quad\text{for } h>0.
\] 

For $h \geq a^2$, let $\tilde{\varphi}=\tilde{\varphi}_h$ be the solution of
\[
\M \tilde{\varphi} =\M \varphi  \quad \text{in }  S_{h}^{\tilde{u}}, \quad \tilde{\varphi} =\tilde{u}(=h) \quad \text{on } \partial  S_{h}^{\tilde{u}}.
\] 
Then $\tilde{u} \leq \tilde{\varphi}$, and $ \ S_{h}^{\tilde{u}}=\{ \tilde{\varphi} < h\}=S_{\tilde h,0}^{\tilde{\varphi}}(\tilde x)$ is a section of $\tilde{\varphi}$ of some height $\tilde{h}>0$ at some point $\tilde x$. Furthermore, it satisfies that
\[
h -\| \tilde{u}-\tilde{\varphi}\|_{L^{\infty}(S_{h}^{\tilde{u}})}  \leq \tilde{h} \leq  h, \quad \M \tilde{\varphi}(S_{\tilde h}^{\tilde{\varphi}}) \approx |S_{\tilde h}^{\tilde{\varphi}}| \approx  \tilde{h}^{\frac{n}{2}}. 
\]
Then, the Alexandrov estimate \eqref{eq:scaledA} implies $\| \tilde{u}-\tilde{\varphi}\|_{L^{\infty}(S_{h}^{\tilde{u}})}  \leq  Ca\tilde{h}^{\frac{1}{2}}$.
We conclude that $h - \tilde{h} \leq Ca\tilde{h}^{\frac{1}{2}}$, which yields $\tilde{h} \approx h$ for $h \geq a^2$.

Hence, $|S_h^{\tilde{u}}|=|S_{\tilde h}^{\tilde{\varphi}}| \approx  \tilde{h}^{\frac{n}{2}} \approx h^{\frac{n}{2}}$. Therefore, we may apply the affine transformation defined in \eqref{eq:simplify normalization 1} and \eqref{eq:simplify normalization 2} to normalize the section $S_h^{\tilde{u}}$ and thus, we may assume $h = 1$ and $a$ is sufficiently small if $C$ is sufficiently large.

Under this normalization, we claim that
\[
\|\tilde{u} - \varphi\|_{L^{\infty}(S_{7/8}^{\tilde{u}})} \leq \sigma(n,a) ,
\]
where $\sigma(n,a)  \to 0$ as $a \to 0$.
Suppose, for contradiction, that there exist $\sigma_0> 0$ and a sequence $a_k \to 0$ with corresponding functions $\tilde{u}_k$, $\varphi_k$ such that
$\|\tilde{u}_k - \varphi_k\|_{L^{\infty}(S_{7/8}^{\tilde{u}_k})} \geq \sigma_0$.
After passing to a subsequence, $\tilde{u}_k$ and $\varphi_k$ converge locally uniformly to convex functions $\tilde{u}_{\infty}$ and $\varphi_{\infty}$ in $S_1^{\tilde{u}_{\infty}}$ satisfying
\[
\lambda \leq \det D^2 \tilde{u}_{\infty} = \det D^2 \varphi_{\infty} \leq \Lambda \quad \text{and} \quad 0 \leq \varphi_{\infty} \leq \tilde{u}_{\infty} \quad \text{in } S_1^{\tilde{u}_{\infty}},
\]
with $\tilde{u}_{\infty}$ and $\varphi_{\infty}$ touching at the origin.
Since $\tilde{u}_{\infty}$ is strictly convex in $S_1^{\tilde{u}_{\infty}}$, the strong maximum principle \cite{jian2025strong} implies $\tilde{u}_{\infty} \equiv \varphi_{\infty}$ in $S_1^{\tilde{u}_{\infty}}$. This contradicts the inequality
\[
\|\tilde{u}_{\infty} - \varphi_{\infty}\|_{L^{\infty}(S_{7/8}^{\tilde{u}_{\infty}})} =\lim_{k \to \infty} \|\tilde{u}_k - \varphi_k\|_{L^{\infty}(S_{7/8}^{\tilde{u}_k})} \geq \sigma_0 > 0,
\]
thus establishing the claim.

Then we have $S_{7/8-\sigma}^\varphi \subset S_{7/8}^{\tilde{u}}$. When $\sigma(n,a)$ is sufficiently small, the strict convexity of $\varphi$ (see \cite{caffarelli1990ilocalization}) then implies a uniform strict convexity of $\tilde u$ and $\varphi$ in the annular region $S_{3/4}^{\varphi} \setminus S_{1/4}^{\varphi}$.
Applying the Harnack inequality from \cite{caffarelli1997properties} to the linearized Monge–Ampère inequalities  
\[
L_{\tilde u}(\varphi - \tilde{u}) \geq 0 \quad \text{and} \quad L_\varphi(\varphi - \tilde{u}) \leq 0,\quad\text{in } S_{3/4}^{\varphi} \setminus S_{1/4}^{\varphi},
\]
we obtain  
\[
\sup_{\partial S_{1/2}^{\varphi}}(\tilde{u} - \varphi) \leq C\inf_{\partial S_{1/2}^{\varphi}}(\tilde{u} - \varphi).
\]
Since $\tilde{u}(0)=\varphi(0)$, the comparison principle implies
\[
0 \leq \tilde{u}- \varphi \leq \sup_{\partial S_{1/2}^{\varphi}}(\tilde{u} - \varphi)\leq C\inf_{\partial S_{1/2}^{\varphi}}(\tilde{u} - \varphi)\quad \text{in }S_{1/2}^{\varphi}.
\]
Therefore, it suffices to estimate $\inf_{\partial S_{1/2}^{\varphi}}(\tilde{u}-\varphi)$.

Let $u$ be the solution to
\[
\M u =\M \varphi+\omega_na^n \delta_y \quad \text{in } S_{1/2}^{\varphi},\quad 
u=\varphi \quad \text{on } \partial S_{1/2}^{\varphi}.
\] 
By the comparison principle, we obtain 
\[
u + \inf_{\partial S_{1/2}^{\varphi}}(\tilde{u}-\varphi) =u + \inf_{\partial S_{1/2}^{\varphi}}(\tilde{u}-u)\leq \tilde{u}  \quad \text{in } S_{1/2}^{\varphi}.
\]
Combining this with Lemma~\ref{lem:interior estimate u}, we derive that
\[
\inf_{\partial S_{1/2}^{\varphi}}(\tilde{u}-\varphi)  \leq \tilde{u} (0)- u(0)  = \varphi (0) - u(0) 
\leq
\begin{cases}
Ca^2|\log a|& \text{if } n = 2, \\
Ca^2& \text{if } n \geq 3. 
\end{cases}  
\] 
The desired estimate in the original setting then follows by reversing the normalization.

The interior regularity of $\varphi$ in $S_{h/2,p}^\varphi(y)$ then follows from the equation \eqref{eq:varphi equation} and Caffarelli's result \cite{caffarelli1990ilocalization}.
\end{proof}

The following lemma is established similarly.
\begin{Lemma}\label{lem:varphi sc from v}
Let $\Omega \subset \mathbb{R}^n$ be a convex domain, and let $\varphi \in C(\overline{\Omega})$ be a convex function satisfying \eqref{eq:varphi equation}. 
Fix $y \in \Omega$, and let $\ell$ be a support function to $\varphi$ at $y$.
Suppose $\tilde{v}$ is a convex function such that
\[
\M\tilde{v} = \M\varphi \cdot \chi_{\{\tilde{v}>\ell\}} \quad \text{in } \Omega, 
\quad \tilde{v} \ge \ell, 
\quad \M\varphi\big(\{\tilde{v}=\ell\}\big) \le \omega_n a^n.
\] 
Assume furthermore that $\tilde{v} \le \varphi$ in $\Omega$ and that $\tilde{v}$ touches $\varphi$ from below at $y$.
Then there exist positive constants $C$ and $\widetilde{C}$, depending only on $n$, $\lambda$, and $\Lambda$, such that if
$h \geq Ca^2$ and $S_{h, \nabla\ell}^{\varphi}(y) \subset \Omega$,  we have $S_{h/2, \nabla\ell}^{\tilde{v}}(y)  \subset S_{h, \nabla\ell}^\varphi(y)$ and    
\[
\tilde{v} \leq \varphi \leq \tilde{v} +  
\begin{cases}
\widetilde{C} a^2 \left( \left|\log \frac{a}{h^{1/2}}\right| + 1 \right) & \text{if } n = 2, \\
\widetilde{C} a^2 & \text{if } n \geq 3,
\end{cases}
\quad \text{in } S_{h/2, \nabla\ell}^{\tilde{v}}(y).
\] 
\end{Lemma}
\begin{proof}
The proof follows an argument analogous to that of Lemma~\ref{lem:varphi sc from u}; we therefore only sketch the main steps.
For notational convenience, we still assume $y=0$ and $l=0$, and write $S_h^{\tilde{v}} := \{\tilde v<h \}$ and $ S_{h}^\varphi := \{\varphi<h\}$.

First, noting that $|S_h^{\varphi}| \approx h^{\frac{n}{2}}$, we apply the affine transformations in \eqref{eq:simplify normalization 1} and \eqref{eq:simplify normalization 2} to normalize the section $S_h^{\varphi}$ so that we can assume $h = 1$. It then suffices to treat the case when the parameter $a$ is sufficiently small. 

A similar argument to that of Lemma~\ref{lem:varphi sc from u} shows that  $\|\tilde{v} -\varphi\|_{L^{\infty}(S_{7/8}^{\varphi})} \leq \sigma(n,a)$,
where $\sigma(n,a) \to 0$ as $a \to 0$. 
It follows that $S_{7/8-\sigma}^{\tilde{v}} \subset S_{7/8}^{\varphi}$ and $S_{0}^{\tilde{v}}:=\{\tilde{v}=0\} \subset S_{\sigma}^{\varphi}$.
For small $a$, an application of the Harnack inequality from \cite{caffarelli1997properties}, combined with the comparison principle, now yields 
\[
0 \leq \varphi-\tilde{v} \leq \sup_{\partial S_{1/2}^{\tilde v}}(\varphi-\tilde{v} )\leq C\inf_{\partial S_{1/2}^{\tilde v}}(\varphi-\tilde{v} ) \quad  \text{in }  \partial S_{1/2}^{\tilde v},
\]  
and the proof reduces to estimate $\inf_{\partial S_{1/2}^{\tilde v}}(\varphi-\tilde{v} ) $.

To this end, let $\tilde{\varphi}$ be the solution to
\[
\M \tilde{\varphi}=\M \varphi  \quad \text{in } S_{1/2}^{\tilde v},\quad 
\tilde{\varphi}=\tilde v \quad \text{on } \partial S_{1/2}^{\tilde v}
\] 
By the comparison principle, we obtain 
\[
\tilde{\varphi}+\inf_{\partial S_{1/2}^{\tilde v}}(\varphi-\tilde{v} ) =\tilde{\varphi}+\inf_{\partial S_{1/2}^{\tilde v}}(\varphi-\tilde{\varphi} ) \leq \varphi  \quad \text{in } S_{1/2}^{\tilde{v}}.
\]
Combining this with Lemma~\ref{lem:interior estimate v}, we derive that 
\[
\inf_{\partial S_{1/2}^{\tilde v}}(\varphi-\tilde{v} )  \leq \varphi(0) - \tilde{\varphi}(0) =\tilde{v}(0) - \tilde{\varphi}(0) 
\leq
\begin{cases}
Ca^2|\log a|& \text{if } n = 2, \\
Ca^2& \text{if } n \geq 3. 
\end{cases}   
\] 
\end{proof}

\begin{proof}[Proof of \eqref{eq:abp gene 2/n log} in Theorem \ref{thm:ordera2}.]
By the Alexandrov estimate \eqref{eq:abp gene 1/n}, it suffices to consider the case when $a > 0$ is sufficiently small. 
The subsequent analysis will focus on the case $n \geq 3$; the two-dimensional case $n=2$ can be handled similarly with minor modifications.
 
\textbf{Step 1.}
We first consider the case $\mu:=\M u-\M\varphi \geq 0$ and prove 
\[
\varphi(x)\leq  u(x) + C(n,\lambda,\Lambda)a^2(|\log a|+1) \quad \text{in } \Omega.
\] 
An application of Theorem~\ref{thm:abp extremal} shows that it suffices to prove that
\begin{equation}\label{eq:abp gene 2/n log positive}
\varphi(x)-u_a(x,x)\leq  C(n,\lambda,\Lambda)a^2(|\log a|+1) \quad \text{in } \Omega.
\end{equation}
Suppose, for contradiction, that there exists a point $x_0 \in \Omega$ at which the maximal deviation
\[
\Theta:= \frac{h_0}{a^2} ,\quad \text{where }h_0 := \|u_a(\cdot,x_0) - \varphi\|_{L^\infty(\overline{\Omega})} = \varphi(x_0) - u_a(x_0,x_0),
\]
satisfies 
\[
\frac{\Theta}{|\log a|+1} >> 1.
\]

Let $p\in\partial\varphi(x_0)$ and $\tilde{u}_0 := u_a(\cdot,x_0) + h_0$. Then $\tilde u_0$ touches $\varphi$ from above at $x_0$. This implies that $p\in\partial \tilde u_0(x_0)=\partial u_a(x_0,x_0)$ and $S_{h_0,p}^{\tilde{u}_0}=S_{h_0,p}^{u_a(\cdot,x_0)}\subset\Omega$. Given that $h_0 \geq C(n,\lambda,\Lambda) a^2$, applying Lemma \ref{lem:varphi sc from u} to $\tilde{u}_0$ and $\varphi$ yields $p=\nabla\varphi(x_0)$ and $S_{h_0/2, p}^{\varphi}(x_0) \subset \Omega$. From now on, in the proof below we drop the subscript from the section notation, since the subgradient at each point we consider is unique. Let $S_{t_0}^{\varphi}(x_0)$ be the maximal section of $\varphi$ at $x_0$; that is, $t_0$ is the supremum of all $t>0$ with $S_{t}^{\varphi}(x_0) \subset \Omega$. 
Then $S_{h_0+t_0}^{\tilde{u}_0}(x_0)=S_{h_0+t_0}^{u_a(\cdot,x_0)}(x_0) \subset \Omega$. Applying Lemma \ref{lem:varphi sc from u} once more then gives $S_{(h_0+t_0)/2}^{\varphi}(x_0) \subset \Omega$. From the maximality of $t_0$, it follows that $t_0 \geq h_0$.

Let us choose $z \in \partial\Omega \cap \partial S_{t_0}^{\varphi}(x_0)$. After a rotation, we assume that the inner normal direction of $\partial\Omega$ at $z$ is given by $e_n$.
Let $\sigma(n,\lambda,\Lambda)>0$ be a small constant, and $C(n,\lambda,\Lambda)>$ be a large constant. 
We proceed by induction on $0 \leq k \leq c\Theta- 2C$ to construct points $x_k \in \Omega$ such that $x_k$ and the corresponding maximal section $S_{t_k}^{\varphi}(x_k)$ of $\varphi$ at $x_k$ satisfy $z \in \partial \Omega \cap \partial  S_{t_k}^{\varphi}(x_k)$,
\[
(x_k-z)\cdot e_n  \leq (1-\sigma)(x_{k-1}-z)\cdot e_n,
\]
and 
\[
h_k:=\left\|u_k - \varphi \right\|_{L^{\infty}(\overline{\Omega})} =\varphi(x_k)-u_k(x_k) \geq h_{k-1}-Ca^2 \geq Ca^2,
\]
where $u_k(x) := u_{a}(x,x_k) \in \mathcal{D}_{a,\varphi}$.

Assuming the inductive hypothesis holds for $k$, we now prove the case for $k+1$. Since $h_k \geq C(n,\lambda,\Lambda) a^2$ and $S_{h_k+t_k}^{u_k}(x_k)\subset \Omega$, we apply Lemma \ref{lem:varphi sc from u} to the functions $\tilde{u}_k := u_k + h_k$ and $\varphi$, which yields $S_{(h_k+t_k)/2}^{\varphi}(x_k)\subset \Omega$ so that $t_k \geq  h_k \ge C(n,\lambda,\Lambda) a^2$, 
and 
\[
\varphi - u_k = \varphi - \tilde{u}_k + h_k \geq h_k - C a^2 \quad \text{in }S_{t_k/2}^{\varphi}(x_k).
\]

From the equation \eqref{eq:varphi equation} of $\varphi$, we observe that 
\[
y_{k+1}:= \sigma z + (1-\sigma)x_k \in S_{t_k/2}^{\varphi}(x_k)
\]
holds whenever $\sigma(n,\lambda,\Lambda) > 0$ is sufficiently small. We then select the point
\[
x_{k+1} \in \partial S_{t_k/2}^{\varphi}(x_k)
\]
with the property that $\nabla\varphi(x_{k+1}) - \nabla\varphi(x_k)$ points along the negative $e_n$-direction. 
Theorem \ref{thm:abp extremal} then implies $h_{k+1} =\varphi(x_{k+1}) - u_{k+1}(x_{k+1}) \geq \varphi(x_{k+1}) - u_k(x_{k+1}) \geq h_k -Ca^2$.
Furthermore, this choice of $x_{k+1}$ yields
\[
(x_{k}-x_{k+1})\cdot e_n  =\sup_{x\in S_{t_k/2}^{\varphi}(x_k)} (x_{k}-x)\cdot e_n  \geq (x_{k}-y_{k+1})\cdot e_n  \geq \sigma (x_{k}-z)\cdot e_n.
\]
Consequently, 
\[
\operatorname{dist}(x_{k+1},\partial \Omega)\le (x_{k+1}-z)\cdot e_n  \leq (1-\sigma)(x_{k}-z)\cdot e_n \leq \dots \leq (1-\sigma)^{k+1}(x_{0}-z)\cdot e_n.
\]
Furthermore, the maximal section $S_{t_{k+1}}^{\varphi}(x_{k+1})$ of $\varphi$  at $x_{k+1}$ also contacts $\partial\Omega$ at $z$, thereby completing the inductive step.

By iteration, we obtain the height estimate
\[
h_{k+1} \geq h_0- C(n,\lambda,\Lambda)ka^2 .
\]
Consequently, the induction process can be continued as long as $Ck \leq   \Theta- 2C$, and we have
\[
Ca^2 \leq \varphi(x_k)- u_k(x_k)  \leq C\operatorname{dist}(x_k,\partial \Omega)^{\frac{1}{n}}a \leq (1-\sigma)^{\frac{k}{n}}\operatorname{diam}( \Omega)^{\frac{1}{n}}a,
\]
where we used \eqref{eq:alexandrov maximum principle} in the second inequality. This implies for a larger $C(n,\lambda,\Lambda)$ that
\[
\Theta \leq C(1+n|\log_{(1-\sigma)}a|) ,
\]  
which completes the proof of \eqref{eq:abp gene 2/n log positive}.

\textbf{Step 2.}
We consider the case $\mu \leq 0$ and prove 
\[
u(x) \leq  \varphi(x)+  C(n,\lambda,\Lambda)a^2(|\log a|+1) \quad \text{in } \Omega.
\] 
An application of Theorem~\ref{thm:abp extremal} shows that it suffices to prove that 
\begin{equation}\label{eq:abp gene 2/n log negative}
v_a(x,p)-\varphi(x)\leq  C(n,\lambda,\Lambda)a^2(|\log a|+1) \quad \text{in } \Omega
\end{equation}
holds for any $v_{a}(\cdot,p)$ having $x$ in its coincidence set.
As in the case of $\mu \geq 0$, we argue by contradiction. Suppose there exist a point $x_0 \in \Omega$ and $p_0 \in \partial \varphi(\Omega)$ at which the maximal deviation
\[
\Xi:= \frac{h_0}{a^2} ,\quad \text{where }h_0 := \|v_a(\cdot,p_0) - \varphi\|_{L^\infty(\overline{\Omega})} = v_a(x_0,p_0)- \varphi(x_0) ,
\]
satisfies 
\[
\frac{\Xi}{|\log a|+1} >> 1.
\]

Let $p_0 \in \partial \varphi(\tilde{x}_0)$ for some $\tilde{x}_0 \in \Omega$. Since $v_a(\cdot,p_0) = \varphi$ on $\partial \Omega$ and $v_a(\cdot,p_0)$ touches $\varphi + h_0$ from below at $\tilde{x}_0$ (see \eqref{eq:xpinK}),
it follows that $S_{h_0,p_0}^{\varphi}(\tilde{x}_0) \subset \Omega$, $\varphi$ is strictly convex and $C^{1,\alpha}$ in  $S_{h_0,p_0}^{\varphi}(\tilde{x}_0)$, and hence, $p_0 = \nabla\varphi(\tilde{x}_0)$. Note that, by \eqref{eq:coincidence inclusion}, the coincidence set of $v_a(\cdot,x_0)$ is contained in $S_{C a^2}^{\varphi}(\tilde{x}_0)$. So $\varphi$ is  $C^{1,\alpha}$ near $x_0$ as well.
Again, from now on, we will drop the subscript from the section notation, since the subgradient at each point we consider is unique. By \eqref{eq:coincidence inclusion} again, the coincidence set of $v_a(\cdot,p_0)$ is contained in $S_{C a^2}^{\varphi}(\tilde{x}_0) \cap S_{C a^2}^{\varphi}(x_0)$. 
Since $h_0 \ge C a^2$, the engulfing property implies
$S_{c h_0}^{\varphi}(x_0) \subset S_{h_0}^{\varphi}(\tilde{x}_0) \subset \Omega$.
Let $S_{t_0}^{\varphi}(x_0)$ be the maximal section of $\varphi$ at $x_0$; that is, $t_0$ is the supremum of all $t>0$ with $S_{t}^{\varphi}(x_0) \subset \Omega$. We conclude that 
\[
t_0 \geq ch_0 \geq Ca^2 \quad \text{and} \quad S_{c^2t_0}^{\varphi}(x_0)\subset S_{ct_0}^{\varphi}(\tilde{x}_0) \subset S_{t_0}^{\varphi}(x_0)\subset\Omega.
\]

Let us choose $z \in \partial\Omega \cap \partial S_{t_0}^{\varphi}(x_0)$. After a rotation, we assume that the inner normal direction of $\partial\Omega$ at $z$ is given by $e_n$.
We proceed by induction on $0 \leq k \leq c\Xi - 2C$ to construct points $x_k \in \Omega$ and $p_k \in \partial \varphi(\Omega)$ such that:  $x_k$ and the corresponding maximal section $S_{t_k}^{\varphi}(x_k)$ of $\varphi$ at $x_k$ satisfy $z \in \partial \Omega \cap \partial  S_{t_k}^{\varphi}(x_k)$,
\[
(x_k-z)\cdot e_n  \leq (1-\sigma)(x_{k-1}-z)\cdot e_n,
\]
and 
\[
h_k:=\left\|v_k - \varphi \right\|_{L^{\infty}(\overline{\Omega})} =v_k(x_k)-\varphi(x_k)\geq h_{k-1}-Ca^2 \geq Ca^2,
\]
where $v_k(x) := v_a(x, p_k) \in \mathcal{D}_{a,\varphi}$ denotes an extremal obstacle solution at $x_k$.

We now verify the inductive step by an argument analogous to the case $\mu \geq 0$. Let $p_k \in \partial\varphi(\tilde{x}_k)$ for some $\tilde{x}_k \in \Omega$. Since $h_k \ge C(n,\lambda,\Lambda) a^2$, $v_k = \varphi$ on $\partial \Omega$, and $v_k$ touches $\varphi + h_k$ from below at $\tilde{x}_k$, repeating the argument used to show $t_0 \ge c h_0$ yields $p_k = \nabla\varphi(\tilde{x}_k)$,
\[
t_k \geq ch_k \geq Ca^2 \quad \text{and} \quad S_{c^2t_k}^{\varphi}(x_k)\subset S_{ct_k}^{\varphi}(\tilde{x}_k) \subset S_{t_k}^{\varphi}(x_k)\subset\Omega.
\]
Then, we apply Lemma~\ref{lem:varphi sc from v} to the functions $\tilde{v}_k := v_k - h_k$ and $\varphi$, which yields
\[
v_k -\varphi =h_k -( \varphi- \tilde{v}_k) \geq h_k - C a^2 \quad \text{in } S_{ct_k}^{\varphi}(\tilde{x}_k)  \supset S_{c^2t_k}^{\varphi}(x_k).
\]
With these facts in hand, the inductive step for $\mu \le 0$ is identical to that for $\mu \ge 0$, upon replacing $S_{t_k/2}^{\varphi}(x_k)$ by $S_{c^2 t_k}^{\varphi}(x_k)$.
Combining this inductive construction with the Alexandrov estimate \eqref{eq:alexandrov maximum principle} yields, for a suitably large constant $C(n, \lambda, \Lambda)$, that
\[
\Xi \leq C(1+n|\log_{(1-\sigma)}a|) ,
\]  
which completes the proof of \eqref{eq:abp gene 2/n log negative}.

\textbf{Step 3.} Finally, the estimate \eqref{eq:abp gene 2/n log} follows by combining the results of Steps 1 and 2 via the Jordan decomposition argument outlined in Section \ref{app:comparison reduction}.
\end{proof}
 
\begin{Remark}\label{rem:pogorolev}
Assume $n \geq 3$. In the proof of estimate \eqref{eq:abp gene 2/n log} above, we have shown that if at some point $x \in \Omega$ the deviation $h:=|u(x)-\varphi(x)|>C a^2$ for a sufficiently large $C$, then $\varphi$ is $C^{1,\alpha}$ and strictly convex at $x$, and the section $S_{h/2,\nabla \varphi(x)}^{\varphi}(x)\subset\Omega$ . 
Consequently, at any point $x$ where $\varphi$ is not strictly convex, we deduce that
\[
\sup_{u\in \D_{a,\varphi}} |u(x) - \varphi(x)|  \leq C(n,\lambda,\Lambda) a^2 .
\] 
\end{Remark}

The following boundary estimate will be used in proving Theorems \ref{thm:ordera2 more} and \ref{thm:rigidity}. Its proof is deferred to Appendix \ref{app:boundary regularity}.

\begin{Proposition}\label{prop:dirichlet problem main} 
Let $\Omega \in C^{2,\alpha}$ be a bounded convex domain and $\Omega_1 \subset \subset \Omega$ a convex subdomain. 
Let $\varphi \in C(\overline{\Omega}) \cap \C_+^{2,\alpha}(\overline{\Omega\setminus\Omega_1})$ be a convex function. 
Let $u \in C(\overline{\Omega})$ be a convex function satisfying:
\[
u = \varphi \quad\text{on } \partial\Omega \quad \text{and} \quad \operatorname{supp}\mu \subset\subset \Omega_1,
\]
where $\mu=\M u-\M\varphi$. Then there exists a sufficiently small positive constant $c_1$ depending only on $n$, $\alpha$,
$\operatorname{dist}(\Omega_1,\partial\Omega)$, $\operatorname{diam}(\Omega)$,
$\|\partial\Omega\|_{C^{2,\alpha}}$,
$\left\|D^2\varphi\right\|_{C^{\alpha}(\overline{\Omega\setminus\Omega_1})}$
and 
$\left\|(D^2\varphi)^{-1}\right\|_{L^{\infty} (\Omega\setminus\Omega_1)}$,
such that if $|\mu|(\Omega) \leq c_1$, then $u \in \C_+^{2,\alpha}(\overline{\Omega} \setminus \overline{\Omega_1})$ and the divergence formula \eqref{eq:div equation} holds. Moreover, for any open convex set $\Omega_2$ satisfying $\Omega_1 \subset\subset \Omega_2 \subset \subset \Omega$, we have
\[
|u-\varphi|+|Du-D\varphi|+|D^2 u-D^2 \varphi | \leq C_2|\mu|(\Omega)  \quad \text{in } \Omega \setminus \overline{\Omega_2},
\]
where $C_2>0$ depends only on $n$, $\alpha$,
$\operatorname{dist}(\Omega_1,\partial\Omega_2)$, $\operatorname{dist}(\Omega_2,\partial\Omega)$, $\operatorname{diam}(\Omega)$,
$\|\partial\Omega\|_{C^{2,\alpha}}$,    $\left\|D^2\varphi\right\|_{C^{\alpha}(\overline{\Omega\setminus\Omega_1})}$ and 
$\left\|(D^2\varphi)^{-1}\right\|_{L^{\infty} (\Omega\setminus\Omega_1)}$.
\end{Proposition}

\section{Asymptotic analysis}\label{sec:asymptotic behavior}

In this section, we always assume that $\partial\Omega \in C^{2,\alpha}$ and $\varphi \in \C_+^{2,\alpha}(\overline{\Omega})$ for some $\alpha\in (0,1)$. Under these regularity conditions, we give a detailed analysis of the asymptotic behavior of the solutions to the isolated singularity and the obstacle problem as the parameter $a  \to 0^+$, and complete the proof of Theorem \ref{thm:ordera2 more}.

All constants in the subsequent estimates depend only on $n$, $\alpha$,  $\operatorname{diam}(\Omega)$, $\|\partial\Omega\|_{C^{2,\alpha}}$,
$\left\|D^2\varphi\right\|_{C^{\alpha}(\overline{\Omega})}$ and $\left\|(D^2\varphi)^{-1}\right\|_{L^{\infty} (\Omega)}$, and are independent of the small parameter $a$ defined in \eqref{defn:a}. Furthermore, by the Alexandrov estimate \eqref{eq:abp gene 1/n}, we may restrict our analysis to the case where $a$ is small.

\subsection{Asymptotic analysis as $a\to 0^+$}\label{sec:singularity problem}

Without loss of generality, we assume that
\begin{equation}\label{eq:varphi normalized s7}
B_{2c}(0) \subset \Omega \subset B_C(0), \quad \varphi(0) = 0, \quad D \varphi(0) = 0, \quad D^2 \varphi(0) = I_n.
\end{equation}
Under these assumptions, we study the asymptotic behavior as $a \to 0^+$ of the solutions $u_a(\cdot,0)$ and $v_a(\cdot,0)$, which correspond to the isolated singularity problem and the obstacle problem, respectively.

We recall that the family convex functions $W_a$ defined in \eqref{eq:isolated global solution} is increasing in $a > 0$, and satisfies $ \det D^2 W_a(a)= 1 + \omega_n a^n \delta_0$,   $W_a(ax)=a^2W_1(x)$, 
with the following asymptotic expansion as $x$ approaches to infinity,
\[
W_{a}(x) 
= \begin{cases}
\frac{1}{2}|x|^2 +\frac{1}{2}a^2 \log (|x|/a)+ O\left(a^2 \right) & \text{if }  n=2, \\
\frac{1}{2}|x|^2 +d_{n,0} a^2   + O\left(\frac{a^n}{|x|^{n-2}}\right) & \text{if }  n\geq 3,
\end{cases}
\]
where $d_{n,0}$ are explicit constants given in Theorem \ref{thm:ordera2 more}. 
Moreover, we have
\[ 
\left| W_{(1+\sigma)a}(x)  -W_a(x)\right|  \leq C_n|\sigma| a^2 \quad \text{if } n \geq 3,
\] 
provided that  $|\sigma| \lesssim 1$, where $C_n$ depends only on $n$.

Let us consider the solution $u_{a}:=u_a(\cdot,0)$ of \eqref{eq:defnuay}, i.e.,
\begin{equation}\label{eq:sing u a 1}
\M u_{a} =\M \varphi+ \omega_na^n\delta_0  \quad \text{in }   \Omega, \quad u_{a} =\varphi \quad \text{on } \partial \Omega.
\end{equation} 

\begin{Lemma}\label{lem:dna2 uaap}
Let $n \geq 3$,  $\beta = \frac{(n-2)\alpha}{n+\alpha}$, and $0< a \leq c$ be small.
Suppose assumption \eqref{eq:varphi normalized s7} holds.  Let $u_a$ be defined by \eqref{eq:sing u a 1}. 
Then we have
\begin{equation}\label{eq:dna2uaat0}
\left| u_{a}(0)-\varphi(0) + d_{n,0} a^2   \right| \lesssim a^{2+\beta}  .
\end{equation}   
\end{Lemma} 
\begin{proof}  
By the assumption on $\varphi$, we have $B_{ch^{1/2}}\subset S_h^\varphi(0) \subset B_{Ch^{1/2}}$. By applying Lemma \ref{lem:interior estimate u}, we obtain 
\begin{equation}\label{eq:dna2 uasquare}
\varphi(x) \geq u_{a}(x)  \geq
\varphi(x) -C\min\left\{  \frac{a^n}{|x|^{n-2}}, a^2 \right\}\quad  \text{in }\Omega.
\end{equation}
Consequently, for any small $Ca \leq r \leq c$ and $
\sigma_0=\sigma_0(r,a):= 	\frac{a^{n-2}}{r^{n-2}} + \frac{r^{2+\alpha}}{a^2}$, we have
\[
\begin{split}
&\left| u_{a}(x) +d_{n,0} a^2  - W_{a}\left(x\right) \right| \\
&\leq  \left| u_{a}(x) -\varphi(x)\right|+\left| \varphi(x)-\frac{1}{2}|x|^2 \right|+\left|\frac{1}{2}|x|^2+d_{n,0} a^2  - W_{a}\left(x\right) \right| 
\\
&\le  \ C \left(\frac{a^{n}}{r^{n-2}}+r^{2+\alpha}+\frac{a^{n}}{r^{n-2}}\right)
\\
&\lesssim   \ \sigma_0 a^2  \quad \text{on }\partial B_{r}(0),
\end{split} 
\]
where  $W_a$ is given by \eqref{eq:isolated global solution}. 
By selecting a sufficiently large constant $C_1 > 0$, this yields on $\partial B_{r}(0)$ that
\[
\left(1+C_1a^{\alpha}\right)W_{ a}(x)-C_1^2\sigma_0 a^2   \leq u_{a}(x)+d_{n,0} a^2\leq  \left(1-C_1a^{\alpha}\right)W_{a}\left(x\right)+C_1^2\sigma_0 a^2  .
\]
The comparison principle then implies
\[
\left| u_{a}(x) +d_{n,0} a^2- W_{a}\left(x\right) \right| \lesssim a^{\alpha} W_{a}\left(x\right) +\sigma_0 a^2  \lesssim \sigma_0 a^2  \quad \text{in }B_{r}(0).
\]

We now choose $r = a^{\frac{n}{n+\alpha}}$, which leads to
$\sigma_0=2a^{\beta}$. 
Substituting this into the above estimate yields
\begin{equation}\label{eq:dna2uaaround0}
\left| u_{a}(x) + d_{n,0} a^2- W_{a}\left(x\right)  \right| \lesssim a^{2+\beta}  \quad \text{in }B_{a^{\frac{n}{n+\alpha}}}(0).
\end{equation} 
In particular,  \eqref{eq:dna2uaat0} follows from setting $x=0$ in \eqref{eq:dna2uaaround0}.
Moreover, by using \eqref{eq:isp control by Linfinite} and then noticing
$
\|u_a-\varphi\|_{L^{\infty}(\Omega)} = \varphi(0)- u_a(0)\le d_{n,0}a^2+C a^{2+\beta},
$ 
we further obtain
\begin{equation}\label{eq:dna2 uaap}
\varphi(x) \geq u_{a}(x)  \geq \varphi(x) -\min\left\{ \frac{C a^{n}}{\left|x\right|^{n-2}}, d_{n,0}a^2+C a^{2+\beta} \right\}\quad  \text{in }\Omega
\end{equation}   
with the help of \eqref{eq:dna2 uasquare}.
\end{proof}

Let $\Phi^{ij} =\operatorname{Cof} D^2 \varphi =\det D^2 \varphi (D^2 \varphi)^{-1}$ denote the cofactor matrix of $D^2 \varphi$. Consider the linearized Monge-Amp\`ere operator at $\varphi$:
\begin{equation}\label{eq:varphi linearized}
L_{\varphi} w :=\Phi^{ij}  w_{ij}=(\Phi^{ij}  w_{i})_j.
\end{equation}
Let $G(x)$ be the (negative) Green’s function of the linearized operator $L_{\varphi} $ at $0$, that is,   
\[  
L_{\varphi}G =\delta_0\quad \text{in }\Omega, \quad   G =0 \quad \text{on } \partial \Omega.  
\] 
Since  $\{\Phi^{ij}\}$ is $\alpha$-h\"older continuous and uniformly elliptic, we have
\[
G(x)\approx  -\operatorname{dist}\left(x,\partial\Omega\right) \quad \text{in }\Omega \setminus B_{ c }(0) ,
\] 
and we have the decomposition
\[
G(x)=\Q (|x|)+H(x),
\]
where 
\[ 
\Q(x) = \Q(|x|) =
\begin{cases}
\frac{1}{2\pi} \log |x|& \text{if }  n=2,\\
-\frac{1}{(n-2)\omega_n} |x|^{2-n} & \text{if }  n \geq 3,
\end{cases} 
\]
is the Green function for the Laplacian equation in $\R^n$,
and the function $H(x)$ satisfies
\[
\|H\|_{C^{1,\alpha}(\overline\Omega)} \lesssim 1, \quad \text{if } n=2.
\]

\begin{Lemma}\label{lem:dna2 ua n=2}
Let $n = 2$, and $0< a \leq c$ be small. Suppose assumption \eqref{eq:varphi normalized s7} holds. 
Let $u_a$ be defined by \eqref{eq:sing u a 1}.  Then we have
\[
\left|u_a(0)-\varphi(0)+\frac{1}{2} a^2 |\log a| \right| \lesssim a^2.
\]
\end{Lemma} 
\begin{proof}
Since $\varphi(0)=0$, we can define
\[
t_a=\frac{2 u_a(0)}{ a^2\log a}=\frac{2 (u_a(0)-\varphi(0))}{ a^2\log a}=\frac{2\|u_a-\varphi\|_{L^{\infty}(\Omega)} }{  a^2|\log a|}>0.
\]
The comparison principle implies 
\[
\varphi \leq u_a +\frac{1}{2} t_a a^2|\log a|\leq \varphi+\frac{1}{2} t_a a^2|\log a| \quad \text{in }\Omega.
\]
Moreover, by \eqref{eq:ua-vaphi n=2}, we have $t_a\le C$. 

 Applying Lemma~\ref{lem:varphi sc from u} to $\tilde{u}=u_a+\tfrac{1}{2}t_a a^2|\log a|$, we obtain
\[
u_a+\tfrac{1}{2}t_a a^2|\log a| \leq \varphi + C a^2 \quad \text{in } B_a(0).
\]
This implies, for $t_a^- := t_a - \tfrac{C}{|\log a|}$, that
\[
u_a \leq \varphi - \tfrac{1}{2} t_a a^2 |\log a| + C a^2 \leq \varphi + t_a^- \,\pi a^2 G 
\qquad \text{on } \partial B_a(0),
\]
where $G$ denotes the Green’s function of the linearized operator $L_{\varphi}$ defined in~\eqref{eq:varphi linearized}.
For any $t \in \R$, noting that
\[
\det D^{2}(\varphi+t  G) - \det D^2\varphi
= t  L_{\varphi}G + t^2  \det D^2 G
= t^2  \det D^2 G \le 0 \quad \text{in } \Omega\setminus\{0\},
\]
we see that $\varphi + t G$ is a supersolution in $\Omega\setminus\{0\}$. The comparison principle (which remains valid even when the supersolution is not convex, by considering its convex hull) then yields
\begin{equation}\label{eq:ua0 upper bound n=2}
u_a \leq \varphi +t_a^- \pi a^2 G \quad \text{in } \Omega\setminus B_a(0) .
\end{equation}

Let $C_1$ be sufficiently large. For any small $r \ge a$, we compare $u_a$ with the upper barrier
\[
\frac{1}{2} t_a^- a^2 \log r
+ \bigl(1 - C_1 r^\alpha\bigr)\!\left(W_a(x) - \frac{1}{2} a^2 \log (r/a) + C a^2\right)
+ C_1^2 r^{2+\alpha}
\]
on $\partial B_r(0)$ using~\eqref{eq:ua0 upper bound n=2}, and hence in $B_r(0)$ by the comparison principle. Since $u_a(0)=\tfrac{1}{2} t_a a^2 \log a$ and $W_a(0)=0$, choosing $r=a^{\frac{2}{2+\alpha}}$ and evaluating at the origin yields
\[
\frac{1}{2} t_a a^2 |\log a^{\frac{\alpha}{2+\alpha}}|
\geq \frac{1}{2} a^2 |\log a^{\frac{\alpha}{2+\alpha}}|- C a^2.
\]
This yields the lower bound
\[
t_a \geq 1- C |\log a|^{-1}.
\]

Finally, an application of Proposition \ref{prop:dirichlet problem main} yields the estimate $|D u_a - D \varphi| +|D^2 u_a - D^2 \varphi| \lesssim a^2$ near $\partial\Omega$. Consequently, the matrix ${A^{ij}}$ defined in \eqref{eq:div equation} satisfies $|A^{ij} - \Phi^{ij}| \lesssim a^2$ near $\partial\Omega$, where $\Phi^{ij} = \operatorname{Cof}(D^2 \varphi)$.
Combining \eqref{eq:ua0 upper bound n=2} with the divergence identity \eqref{eq:div identity smooth}, we obtain 
\begingroup
\allowbreak
\begin{align*}
\pi a^2
&=\int_{\partial  \Omega}A^{ij} \partial_i({u}_a-{\varphi})  \nu_jdS  \\
&\geq \int_{\partial  \Omega}\Phi^{ij} \partial_i({u}_a-{\varphi})  \nu_jdS-Ca^4 \\
&\geq  t_a^-\pi a^2\int_{\partial  \Omega}\Phi^{ij} G_i \nu_jdS-Ca^4 \\
&\geq  t_a \pi a^2-Ca^2 |\log a|^{-1},
\end{align*}
\endgroup
which leads to the upper bound
\[
t_a \leq 1+C |\log a|^{-1}.
\]
This completes the proof.
\end{proof}

We recall that the family of convex functions $W_a^*$ defined in \eqref{eq:obstacle global solution} is decreasing in $a>0$, and satisfies
$\det D^2 W_a^*=\chi_{\left\{W_a^*>0\right\}}$, 
$\left| \left\{ W_a^*=0\right\} \right|= \omega_n a^n$,   $W_a^*(ax)=a^2W_1^*(x)$, 
with the following asymptotic expansion as $x$ approaches to infinity,
\[
W_a^*(x) 
= \begin{cases}
\frac{1}{2}|x|^2 - \frac{1}{2} a^2 \log (|x|/a)+ O\left(a^2 \right) & \text{if }  n=2, \\
\frac{1}{2}|x|^2 -d_{n,0} a^2   + O\left(\frac{a^n}{|x|^{n-2}}\right) & \text{if }  n\geq 3,
\end{cases}
\]
where $d_{n,0} >0$ is the same constant that arises in the expansion of $W_a$. 
Moreover, we have
\[ 
\left| W_{(1+\sigma)a}^*(x)  -W_a^*(x)\right|  \leq C_n|\sigma| a^2, \quad \text{if } n \geq 3,
\] 
provided that  $|\sigma| \lesssim 1$, where $C_n$ depends only on $n$.

For small $a > 0$, consider the obstacle solution $v_a := v_a(\cdot, 0)$ defined by \eqref{eq:defnva} with obstacle $h_a = h_{a,0}$, i.e.,
\begin{equation}\label{eq:obse ap 3}
\M v_a=\M\varphi \cdot \chi_{\left\{ v_a>h_a\right\}}   \quad\text{in }   \Omega, \quad v_a =\varphi \quad \text{on } \partial \Omega,
\end{equation} 
with 
\[
\omega_na^n=\M \varphi (\{ v_a=h_a\}).
\] 

\begin{Lemma}\label{lem:dna2 obs ap}
Let $n \geq 3$,  $\beta = \frac{(n-2)\alpha}{n+\alpha}$, and $0< a \leq c$ be small.
Suppose assumption \eqref{eq:varphi normalized s7} holds. Let $v_a$ be defined by \eqref{eq:obse ap 3}.
Then we have 
\begin{equation}\label{eq:dna2 obs ap vaball}
(1-C a^{\beta} )B_a(0)\subset \left\{v_{a}= h_a\right\} \subset (1+C a^{\beta}   )B_a(0),
\end{equation}    
and the quantity $v_a(0)=h_a$ satisfies
\begin{equation}\label{eq:dna2vaat0}
\left| v_{a}(0)-\varphi(0) - d_{n,0} a^2  \right| \lesssim a^{2+\beta}   .
\end{equation} 
\end{Lemma}	
\begin{proof} 
By applying Lemma \ref{lem:interior estimate v}, we obtain 
\[
\varphi(x) \leq v_{a}(x)  \leq
\varphi(x) +C\min\left\{  \frac{a^n}{|x|^{n-2}}, a^2 \right\}\quad  \text{in }\Omega.
\]
Consequently, for any small $Ca \leq r \leq c$ and $
\sigma_0=\sigma_0(r,a):= 	\frac{a^{n-2}}{r^{n-2}} + \frac{r^{2+\alpha}}{a^2}$, we have
\begin{equation}\label{eq:estimate on va}
\begin{split}
&\left| v_{a}(x) -d_{n,0} a^2  - W^*_{a}\left(x\right) \right|  \\
& \leq \left| v_{a}(x) -\varphi(x)\right|+\left| \varphi(x)-\frac{1}{2}|x|^2 \right|+\left|\frac{1}{2}|x|^2  -d_{n,0} a^2  - W^*_{a}\left(x\right) \right| 
\\
& \le \ C \left(\frac{a^{n}}{r^{n-2}}+r^{2+\alpha}+\frac{a^{n}}{r^{n-2}}\right)
\\
& \lesssim  \ \sigma_0 a^2  \quad \text{on }\partial B_{r}(0),
\end{split} 
\end{equation}   
where $W_a^*$ is given by \eqref{eq:obstacle global solution}.

We claim that 
\begin{equation}\label{eq:estimate on ha}
\left| h_a - d_{n,0} a^2 \right | \lesssim \sigma_0 a^2 .
\end{equation}
In fact, if $h_a- d_{n,0} a^2 \geq C_1^2\sigma_0 a^2 $ for a sufficiently large constant $C_1 > 0$, 
then  
\[ 
v_{a}(x)-h_a \leq \left(1-C_1a^{\alpha}\right) W_{ \left(1+C_1a^{\alpha}\right) a}^*(x)\quad \text{on } \partial B_{r}(0).
\]
The comparison principle then implies 
\[
v_{a}(x)-h_a \leq  \left(1-C_1a^{\alpha}\right) W_{ \left(1+C_1a^{\alpha}\right) a}^*(x) \quad \text{in }B_{r}(0).
\]
Then we derive the following contradiction
\begin{align*}
\omega_na^n &= \M \varphi (\{ v_a=h_a\}) \\
&\geq \M \varphi  \left\{ W_{ \left(1+C_1a^{\alpha} \right)a}^*(x)=0\right\}\\  
&\geq \left(1-CC_1^\alpha a^{\alpha} \right)\left(1+C_1a^{\alpha} \right)^n\omega_na^n\\
&>\omega_na^n,
\end{align*}
where we used $|D^2\varphi(x)-D^2\varphi(0)|\le C|x|^\alpha$ in the second inequality. 
Similarly, if $h_a - d_{n,0}a^2 \leq -C_1\sigma_0 a^2$, we derive the lower bound
\[
v_{a}(x)-h_a \geq  \left(1+C_1a^{\alpha}\right) W_{ \left(1-C_1a^{\alpha}\right) a}^*(x) \quad \text{in }B_{r}(0),
\]
which again yields a contradiction. This completes the proof of \eqref{eq:estimate on ha}.

Let $C_2>0$ be large.
Combining estimates \eqref{eq:estimate on va} and \eqref{eq:estimate on ha}, we find on $ \partial B_{r}(0)$ that 
\[
\left(1+C_2 a^{\alpha}\right) W_{ \left(1+C_2^2 \sigma_0   \right)a}^*(x)
\leq v_{a}(x)-h_a \leq  
\left(1-C_2 a^{\alpha}\right) W_{ \left(1-C_2^2 \sigma_0   \right)a}^*(x).
\]
The comparison principle then implies
\[
\left| v_{a}(x) -h_a- W_{a}^*\left(x\right) \right| \lesssim a^{\alpha} W_{a}^*\left(x\right) +\sigma_0 a^2  \lesssim \sigma_0 a^2  \quad \text{in }B_{r}(0).
\] 

We now choose $r = a^{\frac{n}{n+\alpha}}$, which leads to
$\sigma_0=2a^{\beta}$. 
Substituting this into the earlier estimate yields  \eqref{eq:dna2 obs ap vaball},
\begin{equation}\label{eq:dna2vaaround0}
\left| v_{a}(x) - d_{n,0} a^2-W_{ a}^*(x) \right| \lesssim a^{2+\beta}  \quad  \text{in } B_{a^{\frac{n}{n+\alpha}}}(0),
\end{equation}
and 
 \begin{equation}\label{eq:dna2 obs ap x}
\varphi(x) \leq v_{a}(x)  \leq \varphi(x) + \min\left\{ \frac{C a^{n}}{\left|x\right|^{n-2}}, d_{n,0}a^2+C a^{2+\beta} \right\} \quad \text{in } \Omega.
\end{equation}  
In particular,  \eqref{eq:dna2vaat0}  follows from setting $x=0$ in \eqref{eq:dna2vaaround0}.
\end{proof}

\begin{Lemma}\label{lem:dna2 va n=2}
Let $n = 2$, $0< a \leq c$ be small. Suppose assumption \eqref{eq:varphi normalized s7} holds. Let $v_a$ be defined by \eqref{eq:obse ap 3}. Then we have
\[
\left| v_a(0)-\varphi(0)-\frac{1}{2} a^2 |\log a| \right| \lesssim a^2.
\]
\end{Lemma}
\begin{proof}
Since $\varphi(0)=0$, we can define
\[
h_a=v_a(0)=v_a(0)-\varphi(0)=\|v_a-\varphi\|_{L^{\infty}(\Omega)} >0.
\]
The comparison principle then implies 
\[
\varphi \leq v_a  \leq \varphi +h_a \quad \text{in }\Omega.
\]
Moreover, by \eqref{eq:va-vaphi n=2}, we have $h_a\le C a^2|\log a|$. 

From \eqref{eq:coincidence inclusion}, we have
\[
\{v_a = h_a\} \subset S_{C a^2}^{\varphi}(0) \subset B_{C a}(0).
\]
Define $h_a^+ := h_a +C a^2$ for a possibly larger constant $C$. Then
\[
v_a \le \varphi + \frac{h_a^+}{\log a}  2\pi G \quad \text{on } \partial B_{C a}(0).
\]
Since $\varphi + \frac{h_a^+}{\log a}  2\pi G$ is a supersolution, 
the comparison principle implies
\begin{equation}\label{eq:va0 lower bound n=2}
v_a \le \varphi + \frac{h_a^+}{\log a} 2\pi G
\qquad \text{in } \Omega \setminus B_{C a}(0).
\end{equation}

We claim that
\begin{equation}\label{eq:ha upper bound}
h_a \leq \frac{1}{2}a^2|\log a| +C a^2.
\end{equation}
Suppose, to the contrary, that $h_a- \frac{1}{2} a^2 |\log a|\geq C_1 a^2 $ for some sufficiently large constant $C_1 > 0$.
Choosing $r = a^{\frac{2}{2+\alpha}}$, we obtain
\[
0\le v_a-h_a \leq   \left(1-C_1r^{\alpha}\right) W_{ \left(1+C_1r^{\alpha}\right) a}^*(x)  \quad \text{on } \partial B_{r}(0),
\]
and hence also in $B_r(0)$ by the comparison principle. This implies $v_a = h_a$ in $B_{(1 + C_1 r^{\alpha}) a}(0)$. However, this contradicts the constraint $\M \varphi (\{ v_a=h_a\})=\omega_na^n$.
Therefore, \eqref{eq:ha upper bound} holds.
  
Finally, an application of Proposition \ref{prop:dirichlet problem main} yields the estimate $|D v_a - D\varphi| +|D^2 v_a - D^2 \varphi| \lesssim a^2$ near $\partial\Omega$. 
Consequently, the matrix $\{A^{ij}\}$ defined in \eqref{eq:div equation} satisfies $|A^{ij}-\Phi^{ij}| \lesssim a^2$ near $\partial\Omega$, where $\Phi^{ij} = \operatorname{Cof} D^2 \varphi$. Combining~\eqref{eq:va0 lower bound n=2} with the divergence identity~\eqref{eq:div identity smooth}, we obtain 
\begingroup
\allowbreak
\begin{align*}
\pi a^2
&=\int_{\partial  \Omega}A^{ij} \partial_i(\varphi-v_a)  \nu_jdS  \\
&\leq \int_{\partial  \Omega}\Phi^{ij} \partial_i(\varphi-v_a)  \nu_jdS+Ca^4 \\
&\leq  2 \pi h_a^+ |\log a|^{-1}\int_{\partial  \Omega}\Phi^{ij} G_i \nu_jdS+Ca^4 \\
&\leq 2 \pi (h_a +Ca^2) |\log a|^{-1}     .
\end{align*}
\endgroup
Therefore, 
\[
 h_a \geq \frac{1}{2}a^2|\log a|-C  a^2 .
\]
This completes the proof.
\end{proof}

\subsection{Proof of Theorem \ref{thm:ordera2 more} and an application}\label{sec:stability a2}

Let us summarize the key estimates obtained in Section \ref{sec:singularity problem}.
\begin{Proposition}\label{prop:pointwise sharp}
Under the assumptions $\varphi(0)=0$ and $D\varphi(0)=0$, define $A=(D^2\varphi(0))^{1/2}$ (which is not necessarily the identity) and $\lambda_0=(\det D^2 \varphi(0))^{1/n}$. Then
\begin{itemize}
\item for $n =2 $, Lemmas \ref{lem:dna2 ua n=2} and \ref{lem:dna2 va n=2} imply that
\[
\left|u_a(0)+\frac{1}{2}\lambda_{0}^{-1} a^2 |\log a| \right| \lesssim   a^2 , \quad \left| v_a(0) -\frac{1}{2}\lambda_{0}a^2 |\log a| \right| \lesssim a^2.
\] 
\item for $n \geq 3$, for $\beta = \frac{(n-2)\alpha}{n+\alpha}$,  Lemmas \ref{lem:dna2 uaap} and \ref{lem:dna2 obs ap}  imply that
\[
\left| u_{a}(0) + \lambda_{0}^{-1}d_{n,0} a^2 \right| \lesssim a^{2+\beta}  ,\quad \left| v_{a}(0)-  \lambda_{0}d_{n,0} a^2\right | \lesssim a^{2+\beta},   
\]
and
\[
(1-C a^{\beta} )B_a(0)\subset A^{-1}\left\{v_{a}= h_a\right\} \subset (1+C a^{\beta}   )B_a(0).
\]  
\end{itemize}

This establishes the optimality of the terms $\frac{1}{2}\lambda_{0}^{-1}$, $\frac{1}{2}\lambda_{0}$, and  $a^2 |\log a|$ in \eqref{eq:perturb-result n=2} for $n=2$, as well as the sharpness of  the terms $\lambda_{0}^{-1}$, $\lambda_{0}$ and  $d_{n,0}a^2$ in \eqref{eq:perturb-result} for $n \geq 3$. 
\end{Proposition} 
\begin{proof}
By substituting $\varphi$ with $\varphi(A^{-1}x)$, the conclusion follows from Lemmas \ref{lem:dna2 uaap} and \ref{lem:dna2 obs ap} when $n \geq 3$, and from Lemmas \ref{lem:dna2 ua n=2} and \ref{lem:dna2 va n=2} when $n=2$. 
\end{proof}

\begin{proof}[Proof of Theorem \ref{thm:ordera2 more}] 
By Whitney's extension theorem, we may extend $\varphi \in \C_+^{2,\alpha}(\overline{\Omega})$ to a convex function $\tilde{\varphi} \in \C_+^{2,\alpha}(\overline{\widetilde{\Omega}}) $ on a larger $C^{2,\alpha}$  uniformly convex domain $\widetilde{\Omega}$ that strictly contains  $\Omega$. 
As a result, in view of the reduction in Section \ref{app:comparison reduction}, we may assume without loss of generality that the point $x_0$ is away from the boundary $\partial \Omega$.

For the left-hand side, Theorem \ref{thm:abp extremal} shows that we only need to consider the solution $u_a(\cdot,x_0)$. The corresponding lower bounds in Theorem \ref{thm:ordera2 more} are provided by Proposition \ref{prop:pointwise sharp}.

For the right-hand side, by applying Theorem \ref{thm:abp extremal}, it suffices to consider the quantity $\sup_{p \in \partial \varphi (\Omega)}v_{a}(x_0,p)$, which we assume is attained at some $p \in \partial \varphi (\Omega)$ by $v_{a}(x_0,p)$, such that $x_0$ is in the coincidence set of $v_{a}(\cdot,p) $. The estimate \eqref{eq:coincidence inclusion}, together with the $\C_+^{2,\alpha}$ regularity of $\varphi$, implies that the point $x_p$ for which $p = D \varphi(x_p)$ is close to $x_0$, satisfying
\[
|x_p-x_0|\le Ca.
\]
Denote $\lambda_{0}=(\det D^2 \varphi(x_0))^{\frac{1}{n}}$ and $\lambda_{p}=(\det D^2 \varphi(x_p))^{\frac{1}{n}}$. Then we have
\[
|\lambda_{0}-\lambda_{p}|\le C |x_p-x_0|^\alpha \le Ca^\alpha.
\]
Therefore, when $n \geq 3$, Proposition \ref{prop:pointwise sharp} yields
\[
\begin{split}
u(x_0)-\varphi(x_0)
& \leq v_{a}(x_0, p) -\varphi(x_0)\\
& \leq v_{a}(x_p, p) -\varphi(x_p) \\
& \leq \lambda_{p} d_{n,0} a^2 +Ca^{2+\beta} \\
& \leq  \lambda_{0}d_{n,0} a^2 +Ca^{2+\alpha}+Ca^{2+\beta}\\
& \leq \lambda_{0} d_{n,0} a^2 +Ca^{2+\beta}.
\end{split}
\]
Similarly, when $n = 2$, Proposition \ref{prop:pointwise sharp} yields
\[
\begin{split}
u(x_0)-\varphi(x_0)
& \leq v_{a}(x_0, p) -\varphi(x_0)\\
& \leq v_{a}(x_p, p) -\varphi(x_p) \\
& \leq \frac{1}{2}\lambda_{p} a^2 |\log a| +Ca^{2} \\
& \leq \frac{1}{2} \lambda_{0} a^2 |\log a| +Ca^{2}.
\end{split}
\]
This concludes the proof. 
\end{proof}

Here, let us discuss the regularity of solutions to the  Monge-Amp\`ere equations with multiple isolated singularities 
\begin{equation}\label{eq:sing multiple sing}
\det D^2 u=\det D^2 \varphi+\omega_n a^n \sum_{i=1}^{m}  b_i \delta_{y_i} \quad \text{in } \Omega , 
\end{equation}  
where $\varphi \in \C_+^{2,\alpha}(\overline{\Omega})$, $y_1,\dots ,y_m \in \Omega$, $a>0$ and $b_1,\dots,b_m > 0$ satisfying $\sum_{i=1}^mb_i=1$. This equation \eqref{eq:sing multiple sing} arises, e.g., in the study of surface tensions associate with dimer models \cite{astala2025dimer}. 
In two dimensions, all convex solutions remain strictly convex. However, for $n \geq 3$, solutions may lose strict convexity and develop polytopal singularities \cite{caffarelli1990ilocalization,caffarelli1993note}. 
Some results concerning the structure of these singularities have been obtained by Mooney and Rakshit \cite{mooney2021solutions,mooney2023singular}. Our Theorem \ref{thm:ordera2 more} implies the following sufficient condition for strict convexity:  

\begin{Corollary}\label{prop:strictly convex cond}
Suppose that $\Omega \subset \mathbb{R}^n$, $n \geq 3$, is a bounded convex domain with $\partial \Omega \in C^{2,\alpha}$, where $\alpha \in (0,1)$. Let $\varphi \in \C_+^{2,\alpha}(\overline{\Omega})$. 
Let $u \in C(\overline{\Omega})$ be the convex solution of \eqref{eq:sing multiple sing} with $u =\varphi $ on $\partial \Omega$.
Then there exists a positive constant $C_0$ depending only on $n$, $\alpha$, $\operatorname{diam}(\Omega)$, $\left\|\partial \Omega\right\|_{C^{2,\alpha}}$,   $\left\|D^2\varphi\right\|_{C^{\alpha}(\overline{\Omega})}$ and $\left\|(D^2\varphi)^{-1}\right\|_{L^{\infty} (\Omega)}$, such that if
\begin{equation}\label{eq:strictly convex cond}
\min_{1\leq i<j\leq m} |y_i-y_j|> C_0 a,
\end{equation}
then $u$ is strictly convex in $\overline\Omega$, and $ C^{2,\alpha}$ in   $ \overline\Omega \setminus \{y_1,\dots,y_m\} $.   
\end{Corollary}

Once the solutions are strictly convex, their regularity away from the singularities can be deduced from the established theory of Monge-Amp\`ere equations (see, e.g., \cite{trudinger2008monge,gutierrez2016monge,figalli2017monge,savin2024monge}), and the regularity of the tangent cone at each singularity has been established G\'alvez-Jim\'enez-Mira 
\cite{galvez2015classification} for $n = 2$, and by Savin \cite{savin2005obstacle} and Huang-Tang-Wang \cite{huang2024regularity} for all $n$.

\begin{proof}[Proof of Corollary \ref{prop:strictly convex cond}]
Since $n \geq 3$, using Theorem \ref{thm:ordera2 more}, we find that $\left\| u - \varphi \right\|_{L^{\infty}(\Omega)} \leq Ca^2$.  Suppose, for the sake of contradiction, that $u$ is not strictly convex. Then, according to \cite{caffarelli1990ilocalization,caffarelli1993note},  and Proposition \ref{prop:dirichlet problem main}, $u$ must be linear on a non-degenerate polytope $E$ generated by some points in $\{y_1,\dots,y_m\}$. 
After a suitable translation and rotation, 
we may assume that $u$ is linear on the segment connecting $y_1 = d e_n$ and $y_2 = -d e_n$ for some $d > 0$. 
Given the $\C^{1,1}_+$ regularity of $\varphi$, we must have 
\[
0 = u(y_1)+u(y_2)-2u(0) \geq \varphi(y_1)+\varphi(y_2)-2\varphi (0)-C a^2 \geq c d^2 -Ca^2,
\]
which implies $d \leq Ca$. 
This contradicts the assumption \eqref{eq:strictly convex cond}.  
\end{proof}

\section{Stability of solutions with small perturbed measure}\label{sec:rigidity a2}

In this section, we focus on the case $n \geq 3$ and prove Theorem \ref{thm:rigidity}. 
Following the proof strategy of Theorem \ref{thm:ordera2 more}, we combine Whitney's extension theorem with the Jordan decomposition theorem to reduce the problem to the case where the measure $\mu = \M u - \M \varphi$ is of compact support in $\Omega_1 \subset \subset \Omega$ and is either non-negative or non-positive.
Under this reduction, to prove Theorem \ref{thm:rigidity}, it suffices to show for all $x_0\in \Omega_1$ that 
\begin{itemize}
\item If $\mu$ is non-negative, then 
\[
 \frac{ u(x_0)-\varphi(x_0)}{d_{n,0}a^2} \geq -\lambda_{0}^{-1}-C_{\rho}a^{\beta}+c_{\rho}\left(\frac{\omega_na^n-\mu \left( E \left(x_0, \rho a\right)\right) }{\omega_na^n} \right)^\frac{n}{2}.
\]
\item If $\mu$ is non-positive, then 
\[
 \frac{ u(x_0)-\varphi(x_0)}{d_{n,0}a^2} \leq \lambda_{0} +C_{\rho}a^{\beta}-c_{\rho}\left(\frac{\omega_na^n+\mu \left( E (x_0, (1+\rho)a)\right)}{ \omega_na^n} \right)^n.
\]
\end{itemize}

Unless specifically emphasized, the constants $c$ and $C$ shall also depend on $\rho$ whenever it is involved.

We will use the following lemma in a general form.
\begin{Lemma}[A boundary contact lemma]\label{lem:extends to boundary}
Let $u, \varphi \in C(\overline{\Omega})$ be convex functions satisfying $u = \varphi$ on $\partial \Omega$. Suppose that near $\partial \Omega$, the function $\varphi$ is strictly convex and non-degenerate in the sense of \eqref{eq:varphi equation}, and $\varphi$ is not identically equal to $u$  there. If the measure $\mu = \M u - \M \varphi$ is of compactly support $\Omega$ and satisfies $\mu(\Omega) = 0$, then both sets $\overline{\{u <\varphi\}}$ and $\overline{\{ \varphi<u\}}$ intersect the boundary $\partial\Omega$. In particular, both $\{u <\varphi\}$ and $\{ \varphi<u\}$ are nonempty sets.
\end{Lemma}
\begin{proof}
Assume, for contradiction, that one of the sets—say, $E = \overline{\{u < \varphi\}}$—is compactly contained in $\Omega$.
Suppose further that both $E$ and $\operatorname{supp}(\mu)$ lie compactly within some $\Omega_1 \subset\subset \Omega_2 \subset\subset \Omega$, and that $\varphi$ is strictly convex and non-degenerate in $\Omega \setminus \Omega_1$.
Since $u \geq \varphi$ in $\Omega \setminus \Omega_1$, the strong maximum principle \cite{jian2025strong} implies that $u - \varphi > 0$ in $\Omega \setminus \Omega_1$.
By approximation, we may assume $\partial \Omega \in \C_+^{3,\alpha}$ and $u, \varphi \in \C_+^{3,\alpha}(\overline{\Omega})$, with $u - \varphi > 0$ on $\partial \Omega_2$. Furthermore, we may maintain that $\mu$ is compactly supported in $\Omega_2$ and satisfies $\mu(\Omega) = 0$. Now, \eqref{eq:div equation} is uniformly elliptic in $\overline{\Omega\setminus \Omega_2}$, and we have $u=\varphi$ on $\partial \Omega$ while $u>\varphi$ on $\partial \Omega_2$.
Applying the classical strong maximum principle and the Hopf boundary lemma, we conclude that the outer normal derivative satisfies
\[
D_{\nu} (u-\varphi)<0\quad \text{on }\partial \Omega.
\]
However, this leads to a contradiction, since it implies
\[
0> \int_{\partial \Omega}  A^{\nu j}(x) D_{\nu} (u-\varphi) \nu_j= \int_{\partial \Omega}  A^{ij}(x) (u-\varphi)_i \nu_j \overset{\eqref{eq:div identity smooth}}{=}\mu(\Omega)=0.
\]
Therefore,   $ \overline{\{u < \varphi\}}$ must intersect $\partial\Omega$. A similar argument shows that $ \overline{\{ \varphi< u\}}$ must intersect $\partial\Omega$ as well.
\end{proof}

The following result on small perturbations of solutions follows from Savin \cite{savin2007perturbation}.
\begin{Theorem}\label{lem:small perturbation}
Let $r > 0$ and $\varphi_1 \in \C_+^{2,\alpha}(B_r(x_0))$ be a convex function satisfying $\left\|D^2\varphi_1\right\|_{C^{\alpha}}+\left\|(D^2\varphi_1)^{-1}\right\|_{L^{\infty} } \leq C_0$.  There exists $c_0 (n,C_0) > 0$ sufficiently small such that if $\varphi_2 \in C(B_r(x_0))$ is a convex function satisfying
\[
\det D^2 \varphi_2 = \det D^2 \varphi_1\quad \text{in }B_r(x_0),\quad \|\varphi_2 - \varphi_1\|_{L^\infty(B_r(x_0))} \leq c_0 r^2,
\] 
then we have $\varphi_2\in \C_+^{2,\alpha}(B_r(x_0))$ with $\left\|D^2\varphi\right\|_{C^{\alpha}(B_{r/2}(x_0))}+\left\|(D^2\varphi)^{-1}\right\|_{L^{\infty} (B_{r/2}(x_0))} \leq 2C_0$.
\end{Theorem}

\subsection{The non-negative case}
We start with a lemma:

\begin{Lemma}\label{lem:subdifferential relation}
Let $\Omega$ be a bounded convex domain in $\R^n$ containing the origin.  Let $\mu_0 \geq 0$ be a locally finite Borel measure, and let $\varphi_1, \varphi_2 \in C(\overline{\Omega})$ be two convex functions satisfying  
\[
\M \varphi_1 =\mu_0 + \omega_n a_1^n \delta_0 \quad \text{and} \quad \M \varphi_2 =\mu_0+ \omega_n a_2^n \delta_0
\]
in $\Omega$ for some $0 \leq a_1 \leq a_2$. Suppose $\varphi_1 \leq \varphi_2$ and $\varphi_1(0) = \varphi_2(0)$. Then we have 
\[
t_0:= \inf \left\{ \varphi_2(x) - \varphi_1(x) : x \in \partial \Omega \right\}  \leq C(a_2-a_1),
\] 
where $C> 0$ depending only on $n$ and $\operatorname{dist}(0, \partial \Omega)$.
\end{Lemma}
\begin{proof}
Suppose $t_0>0$. Let $w$ denote the cone function defined in \eqref{eq:cone function} at $y=0$.  
Suppose 
\[
\varphi_{1}+(a_2-a_1)(w-w(0))+\frac{t_0}{2} \leq \varphi_{2} \quad\mbox{on } \partial \Omega.
\] 
Applying the Brunn-Minkowski inequality $|E + F|^{1/n} \geq |E|^{1/n} + |F|^{1/n}$
with $E = (a_2 - a_1)\partial w(0)$ and $F = \partial \varphi_1(0)$, we obtain
\[
\M \big((a_2 - a_1)w + \varphi_{1}\big)(\{0\}) \geq \left( \omega_n^{1/n}(a_2 - a_1) + (\omega_n a_1^n + \mu_0(\{0\}))^{1/n} \right)^n \geq \M \varphi_{2}(\{0\}),
\]
where we used the triangle inequality in $\ell^n(\R^2)$ in the last inequality. 
The comparison principle then implies $\varphi_{1}+(a_2-a_1)(w-w(0))+\frac{t_0}{2}  \leq \varphi_{2}$ in $\Omega$.
However, this contradicts the equality $\varphi_1(0) = \varphi_2(0)$. Consequently, 
\[
\varphi_{1}+(a_2-a_1)(w-w(0))+\frac{t_0}{2} \geq \varphi_{2} \quad\mbox{on } \partial \Omega,
\] 
which implies that $t_0 \leq C(a_2 - a_1)$.
\end{proof}

To study the model for the non-negative case, we characterize the solution behavior of the Monge-Ampère equation with two isolated singularities.

\begin{Lemma}\label{lem:two singularity}
Assuming that $n\geq 3$, $\partial\Omega \in C^{2,\alpha}, \varphi \in \C_+^{2,\alpha}(\overline\Omega)$, and \eqref{eq:varphi normalized s7} is satisfied. Let $a, \varepsilon, \rho > 0$ be  small constants, and let $u_a$ solve \eqref{eq:sing u a 1}. For each fixed $y \in \partial B_{\rho a}(0)$, consider the solution $u$ of
\begin{equation}\label{eq:two singularity}
\M u=\M \varphi+ \omega_n (1-\varepsilon) a^n\delta_0+ \omega_n \varepsilon a^n\delta_y \quad \text{in }  \Omega, \quad u  =\varphi \quad \text{on } \partial \Omega,
\end{equation}
then we have   for all small $\varepsilon \gtrsim a^{\beta}$ with $\beta=\frac{(n-2)\alpha}{n+\alpha}$ that
\begin{align} 
\kappa_1:=&\sup_{x\in \Omega} \left\{u (x)-u_a(x)\right\}  
=u (0)-u_a(0) \approx \varepsilon a^2 , \label{eq:dna2lower uat0}\\  
\kappa_2:=&\sup_{x\in \Omega} \left\{u_a (x)-u(x)\right\} 
=u_a (y)-u(y) \lesssim \varepsilon^{\frac{2}{n}} a^2. \label{eq:dna2upper uatq} 
\end{align}
\end{Lemma}
\begin{proof} 
By approximation, we may assume $\partial \Omega \in \C_+^{3,\alpha}$ and $u, \varphi \in \C_+^{3,\alpha}(\overline{\Omega})$.  Then the unique continuation theorem for Monge-Amp\`ere equations (see \cite[Theorem 1.3]{mooney2015partial}) implies that $u$ does not identically equal to $u_a$ near $\partial\Omega$. By Lemma \ref{lem:extends to boundary}, we conclude that both
\[
E := \{u > u_a\}\quad \text{and} \quad F := \{u_a > u\}
\]
are nonempty, and extend to $\partial\Omega$. The comparison principle implies that each connected component of the open set $E$ or $F$ must contain the origin $0$ or the point $y$, respectively. Therefore, both $E$ and $F$ are connected, and 
\[
0 \in  E,\quad y \in  F.
\]
Thus, $u_a$ equals to $u$ at some point on $\partial B_{3 a/2} (0)$ (recalling that $\partial B_{3 a/2} (0)\subset\Omega$ since $a$ is small).

By Proposition \ref{prop:dirichlet problem main}, $u_a \in \C_+^{2,\alpha}(\overline\Omega\setminus\{0\})$.  By lifting the function $u_a$ upward until it touches $u$ from above, we know from the strong maximum principle that the touch point must be $0$, and thus, 
\[   
\kappa_1=u (0)-u_a(0) = \sup_{x\in \Omega} \left\{u (x)-u_a(x)\right\} > 0.
\]
Analogously, 
\[   
\kappa_2=u_a(y) - u (y)= \sup_{x\in \Omega} \left\{u_a (x)-u(x)\right\}   > 0.
\] 

We begin by examining the case where
\begin{equation}\label{eq:kappa12 small}
\frac{\kappa_1 + \kappa_2}{a^2} \le \gamma(n,\rho)
\end{equation}  
for some small positive constant $\gamma(n,\rho)$.  From the approximation \eqref{eq:dna2uaaround0}, the rescaled function $\frac{u_a(ax)}{a^2}$ is close to $W_1(x) - d_{n,0}$ in $B_4(0)$. Then, \eqref{eq:kappa12 small} implies that $\frac{u(ax)}{a^2}$ is also close to $W_1(x) - d_{n,0}$ in $B_4(0)$.    
Applying Theorem \ref{lem:small perturbation}, we have that both rescaled functions $\frac{u_a(ax)}{a^2}$ and $\frac{u(ax)}{a^2}$  are $\C_+^{2,\alpha}$ in $B_{2}(0) \setminus  \left(B_{\rho/4 }(0) \cup B_{\rho/8}\left(a^{-1}y\right)\right) $.

We first derive the upper bound estimate for $\kappa_1$. Since $u_a$ equals to $u$ at some point on $\partial B_{3 a/2} (0)$,  the Harnack inequality yields
\[
\frac{u_a(ax)  +\kappa_1- u(ax)}{a^2}\approx  \kappa_1 \quad  \text{in } B_{7/4}(0) \setminus \left(B_{\rho /2 }(0) \cup B_{\rho /4}\left(a^{-1}y\right)\right).
\] 
Applying Lemma \ref{lem:subdifferential relation} to the rescaled functions $\frac{u_a(ax)+\kappa_1}{a^2}$ and $\frac{u(ax)}{a^2}$ in $B_{3\rho/4}(0)$, we obtain  $\kappa_1 \lesssim \varepsilon a^2$.

Secondly, we show the lower bound estimate for $\kappa$ assuming $\varepsilon \gtrsim a^{\beta}$. If $\kappa_1 < \tau \varepsilon a^2 $ for some sufficiently small $\tau>0$, then we recall the approximation \eqref{eq:dna2uaaround0} and compare $u$ with the barrier function
\[
 w(x):= u_a(0)+\left(1+Ca^{\alpha} \right) W_{(1- c(n)\varepsilon  )a}(x) +  \tau \varepsilon a^2.
\]
Note that $u \geq u_a- C\kappa_1 \geq w$ on $\partial B_{\rho a/2}(0)$, the comparison principle implies $u \geq w$ in $B_{\rho a/2}(0)$. However, this contradicts  $u(0) =u_a(0)+ \kappa_1 <u_a(0)+\tau \varepsilon a^2 = w(0)$, thereby establishing $\kappa_1 \ge \tau \varepsilon a^2 $. This concludes the proof of \eqref{eq:dna2lower uat0}.

Thirdly, we show the upper bound estimate for $\kappa_2$. 
Since $u_a$ equals to $u$ at some point on $\partial B_{3 a/2} (0)$, we may apply the Harnack inequality to the non-negative function $u +\kappa_2-u_a$, which yields   
\[
\frac{u(ax) +\kappa_2-u_a(ax)}{a^2} \approx  \frac{\kappa_2}{a^2}  \quad  \text{in } B_{7/4}(0) \setminus  \left(B_{\rho/2}(0) \cup B_{\rho/4}(a^{-1}y)\right).
\]  
By applying \eqref{eq:abp inf 2/n} to the rescaled functions $\frac{u(ax) +\kappa_2}{a^2}$ and $\frac{u_a(ax)}{a^2} \in \C_+^{2,\alpha}$ in $B_{\rho/2}(a^{-1}y)$, we derive $0 \geq c\kappa_2/a^2-C\varepsilon^{\frac{2}{n}}$, which consequently establishes \eqref{eq:dna2upper uatq}.

Finally, we note that the preceding argument yields the estimate $\kappa_1 + \kappa_2  \le C \varepsilon^{\frac{2}{n}} a^2$ provided that \eqref{eq:kappa12 small} holds, and the Alexandrov estimate \eqref{eq:abp gene 1/n} implies that \eqref{eq:kappa12 small} holds if $\varepsilon \le  a^{3n}$. By the continuous dependence of the solution $u$ on the parameter $\varepsilon$, we conclude that \eqref{eq:kappa12 small} remains valid whenever $C\varepsilon^{\frac{2}{n}}<\gamma(n,\rho)$, which completes the proof.
\end{proof}

\begin{proof}[Proof of Theorem \ref{thm:rigidity}. The Non-negative Case.]
For the non-negative case $\mu \geq 0$, we have the inequality $u \leq \varphi$. After a suitable affine transformation, we now assume for simplicity that $x_0=0$ and  \eqref{eq:varphi normalized s7} hold. In this case, $\lambda_{0}=\det D^2\varphi(0)=1$, and $E \left(x_0, \rho a\right)= B_{\rho a}(0)$.

Let us define $\varepsilon_1$ and $\kappa$ by
\[
u(0)  - \varphi(0)= -\left(1-{\varepsilon_1}\right)d_{n,0}a^2,\quad \mu(B_{\rho a}(0))= (1-\kappa)^n \omega_na^n.
\]
Then it is clear that $\kappa\ge 0$, and Theorem \ref{thm:ordera2 more} implies that $\varepsilon_1\ge -Ca^{\beta}$. Let $\varepsilon=\varepsilon_1+2Ca^{\beta}\ge Ca^{\beta}$. It suffices to show that
\[
\kappa \lesssim \varepsilon^{\frac{2}{n}}.
\]

For each $y \in \partial B_{\rho a}(0)$, let $\omega(
\cdot, y)$ be the solution of \eqref{eq:two singularity}. Combining \eqref{eq:dna2uaat0} and \eqref{eq:dna2lower uat0} yields
\[
\omega(0, y)> u_a(0)+c\varepsilon a^2\geq \varphi(0)- \left(1-{\varepsilon_1}\right)d_{n,0}a^2=u(0).
\]
Then, Lemma \ref{lem:comparison gene} (with the set $W=\{0,y\}$) implies 
\[
\omega(y, y) \leq u(y).
\]
In conclusion, for every $y \in \partial B_{\rho a}(0)$, we have
\[
u(y) \geq \omega(y, y)   \overset{\eqref{eq:dna2upper uatq}}{\geq} u_{a}(y) -C\varepsilon^{\frac{2}{n}}a^2 \overset{\eqref{eq:dna2uaaround0}}{\geq} W_{a}(y)-d_{n,0}a^2 -Ca^{2+\beta}-C\varepsilon^{\frac{2}{n}}a^2 .
\]
The graphs of these three functions $u,u_a$ and $\omega(\cdot,y)$ can be illustrated as follows:
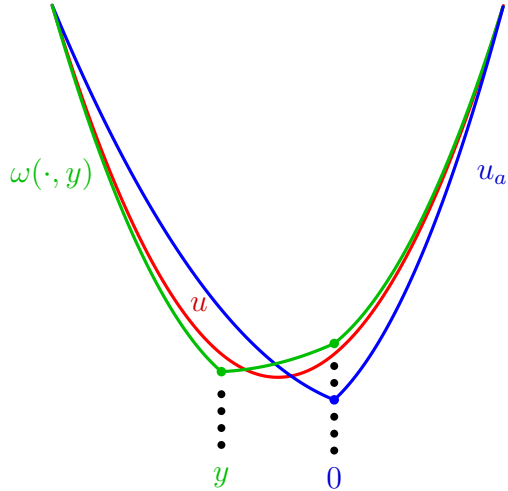
\begin{center}
\begin{tikzpicture}[scale=1.5, line cap=round, line join=round]
  \draw[very thick, red]
    plot[domain=-2:2, samples=200] (\x,{0.825*(\x-2)*(\x+2)+2});
  \node[red!80!black] at (-0.7,-0.65) {$u$};
\draw[very thick, blue]
    plot[domain=-2:0.5, samples=300] (\x,{0.4*(\x - 1)*(\x - 1) - 1.6});
\draw[very thick, green!75!black]
  plot[domain=0.5:2, samples=300] (\x,{(4/5)*\x*\x - (6.0/5.0)});
\draw[very thick, green!75!black]
  plot[domain=-2:-0.5, samples=300] (\x,{(13/15)*\x*\x - (22.0/15.0)});
\draw[very thick, green!75!black]
  plot[domain=-0.5:0.5, samples=300] (\x,{(1/5)*\x*\x +(1/4)*\x- (47.0/40.0)});
\draw[very thick, blue]
  plot[domain=0.5:2, samples=300] (\x,{(14.0/15.0)*\x*\x - (26.0/15.0)});
  \node[blue!80!black] at (1.9,0.5) {$u_a$};
    \node[green!75!black] at (-2,0.5) {$\omega(\cdot,y)$};
\fill[blue] (0.5,-1.5) circle (1.2pt);
\fill[green!75!black] (0.5,-1) circle (1.2pt);
\fill[green!75!black] (-0.5,-1.25) circle (1.2pt);
\foreach \k in {1,2,3} {
  \fill[black] (0.5, {-1.5 - 0.15*\k}) circle (1pt);
}
\foreach \k in {1,2} {
  \fill[black] (0.5, {-1.05 - 0.15*\k}) circle (1pt);
}
\foreach \k in {1,2,3,4} {
  \fill[black] (-0.5, {-1.3 - 0.15*\k}) circle (1pt);
}
\node[blue!80!black, anchor=north] at (0.5,{-1.5 - 0.5}) {$0$};
\node[green!75!black, anchor=north] at (-0.5,{-1.5 - 0.5}) {$y$};
\end{tikzpicture}
\end{center}

 Let 
 \[
w(x):=(1+Ca^{\alpha})\left[W_{(1-\kappa)a}(x) -W_{(1-\kappa)a}( \rho a e_n)\right]+\inf_{y\in \partial B_{\rho a}(0)} u(y).
\]
Then $u\ge w$ on $\partial B_{\rho a}(0)$. Since $\varphi\in C^{2,\alpha}$, $\lambda_{0}=\det D^2\varphi(0)=1$ and $\mu(B_{\rho a}(0))= (1-\kappa)^n \omega_na^n$, we have for every Borel set $E$ satisfying $\{0\}\subset E \subset B_{\rho a}(0)$ that
\[
\begin{split}
\M w(E)  &\geq \M \varphi(E)+  (1-\kappa)^n \omega_na^n\\
&= \M \varphi(E)+\mu( B_{\rho a}(0))\\
&\geq \M \varphi(E)+\mu( E)\\
&=\M u(E).
\end{split}
\] 
By applying Lemma \ref{lem:comparison gene} in $B_{\rho a}(0)$ to $u$ and $w$, along with the fact that  
$\varepsilon^{\frac{2}{n}}>a^{\beta}>a^{\alpha}$, we obtain that
\[
\begin{split}
u(0) \geq w(0) 
& =  -(1+Ca^{\alpha}) W_{(1-\kappa)a}(\rho a e_n)+\inf_{y\in \partial B_{\rho a}(0)} u(y)\\
&\geq  -(1+Ca^{\alpha})W_{(1-\kappa)a}(\rho a e_n)+ W_{a}(\rho a e_n)-d_{n,0}a^2 -C\varepsilon^{\frac{2}{n}}a^2  \\ 
&= \left[ W_{1}\left(\rho e_n\right)-(1+Ca^{\alpha})W_{(1-\kappa)} \left(\rho e_n \right) \right]a^2
-d_{n,0}a^2 -C\varepsilon^{\frac{2}{n}}a^2   \\
& \geq c(n,\rho)\left[\kappa -C a^{\alpha} \right]a^2-d_{n,0}a^2 -C\varepsilon^{\frac{2}{n}}a^2 . 	
\end{split}
\]
Thus, since $\varphi(0)=0$, we have
\[ 
-\left(1-{\varepsilon_1}\right)d_{n,0}a^2 >c(n,\rho)\left[\kappa -C a^{\alpha} \right]a^2-d_{n,0}a^2 -C\varepsilon^{\frac{2}{n}}a^2 , 
\]
which yields 
\[
\kappa \lesssim  {\varepsilon_1}+\varepsilon^{\frac{2}{n}}+  a^{\alpha}\lesssim  \varepsilon^{\frac{2}{n}}.
\]
This establishes \eqref{eq:stabilitypositive} in Theorem \ref{thm:rigidity} for the non-negative case.
\end{proof}

\subsection{The non-positive case} 

To study the model for the non-positive case, we characterize the solution behavior of the Monge-Ampère equation with two linear obstacles. 
Throughout this subsection, we assume for simplicity that $\partial\Omega \in C^{2,\alpha}$, $\varphi \in \C_+^{2,\alpha}(\overline\Omega)$, and that \eqref{eq:varphi normalized s7} holds.

Fix $q \in \partial \varphi(\Omega)$, and let $\ell_{q}$ be the support function of $\varphi$ with slope $q$ at some point $x_q\in\Omega$. Recall that \eqref{eq:varphi normalized s7} implies $\ell_0=0$. 
In the single-obstacle setting, the existence of a solution is established by \cite[Proposition 1.1]{savin2005obstacle} via the Perron method, where the solution is defined as the pointwise infimum of all supersolutions that lie above the obstacle.
Adapting that approach to two linear obstacles, for any sufficiently small  $\kappa_3,\kappa_4\geq 0$, the Perron method yields a unique solution $v_{\kappa_3,\kappa_4}$ to
\begin{equation}\label{eq:double obs 1} 
\M v=\M\varphi \cdot \chi_{  \left\{v> \kappa_3\right\} \cap \left\{v > \ell_{{q}} +\kappa_4\right\} }   \quad \text{in }  \Omega, \quad v=\varphi \quad \text{on } \partial \Omega,
\end{equation}
with the constraint $v \geq \max\left\{ \kappa_3, \ell_{{q}} +\kappa_4\right\}$, where the first equation is interpreted in the same way as \eqref{eq:defnvaq 1} when $\M\varphi$ has unbounded density. The sets $K_3= \left\{v =\kappa_3\right\}$ and $K_4=\left\{v =\ell_{{q}} +\kappa_4\right\}$ are called the  coincidence sets of $v$.

\begin{Lemma}\label{lem:obs monotonicity}
Let $v = v_{\kappa_3,\kappa_4}$ and $v' = v_{\kappa_3',\kappa_4'}$, with coincidence sets denoted by $K_3, K_4$ and $K_3', K_4'$, respectively. Then the following two-sided bound holds:
\begin{equation}\label{eq:obs monotonicity}
v' - \max\left\{ \kappa_3' - \kappa_3, \kappa_4' - \kappa_4, 0 \right\} \leq v \leq v' + \max\left\{ \kappa_3 - \kappa_3', \kappa_4 - \kappa_4', 0 \right\}.
\end{equation}

In the mixed case where $\kappa_3 \geq \kappa_3'$ and $\kappa_4 \leq \kappa_4'$, the coincidence sets satisfy the inclusions $K_3' \subset K_3$ and $K_4 \subset K_4'$.

In the case where $\kappa_3 \geq \kappa_3'$ and $\kappa_4 \geq \kappa_4'$, then $v\geq v'$, and the measure discrepancy satisfies
\[
\left| \M \varphi - \M v \right|(\Omega) \geq \left| \M \varphi - \M v' \right|(\Omega).
\]
\end{Lemma}
\begin{proof}
Suppose $v' < v$ at some point. Then there exists $t > 0$ such that $v' + t \geq v$ touches $v$ from above at some point. Let $E = \{v' + t = v\}$. Then  $E \subset \subset \Omega$. If $E\cap \left( K_3 \cup K_4 \right) = \emptyset$, then for sufficiently small $\epsilon > 0$, the sublevel set $\{v' + t - \epsilon < v\}$ would also be disjoint from $K_3 \cup K_4$. Applying the comparison principle to $v' + t - \epsilon$ and $v$ in this region would then yield $v' + t - \epsilon \ge v$ on $\{v' + t - \epsilon < v\}$, contradicting that $E$ in not empty. Therefore,
\[
E \cap \left( K_3 \cup K_4 \right) \neq \emptyset.
\]
Without loss of generality, assume $E \cap K_3 \neq \emptyset$ and take $x_0 \in E \cap K_3$. Then $0 \in \partial v'(x_0)$, and the convexity of $v'$ implies $v'(x_0) \ge \kappa_3'$ and $t \le \kappa_3 - \kappa_3'$ . Furthermore, if $K'_3\neq\emptyset$, then $v'(x_0) = \kappa_3'$, $t = \kappa_3 - \kappa_3'$ and $K_3' \subset K_3$. Therefore, $t \le \kappa_3 - \kappa_3'$ and $K_3' \subset K_3$. This establishes the right-hand inequality in \eqref{eq:obs monotonicity}. The left-hand inequality follows by a symmetric argument.

In the mixed case where $\kappa_3 \geq \kappa_3'$ and $\kappa_4 \leq \kappa_4'$, the above argument for $\kappa_3 > \kappa_3'$, combined with the fact $v\leq v'$ when $\kappa_3 = \kappa_3'$, shows that $K_3' \subset K_3$. A symmetric argument yields $K_4 \subset K_4'$.

In the case where $\kappa_3 \geq \kappa_3'$ and $\kappa_4 \geq \kappa_4'$, we have $v \geq v'\ge \varphi$ with $v=v'=\varphi$ on $\partial\Omega$, and hence $\partial v(\Omega) \subset \partial v'(\Omega)$. This yields
\[
\left| \M \varphi - \M v \right|(\Omega) \geq \left| \M \varphi - \M v' \right|(\Omega).
\]
\end{proof}

For $q \in \partial \varphi(\Omega)$, recall the obstacle solution $v_a(\cdot,q)$ defined by \eqref{eq:defnva} and \eqref{eq:defnvaq 2} with obstacle $\ell_q + h_{a,q}$, and that \eqref{eq:varphi normalized s7} holds.

\begin{Lemma}\label{lem:exist kappa4}  
Suppose $h_{a,0} < \inf_{\partial \Omega } \varphi$ and $h_{a,q} < \inf_{\partial \Omega } (\varphi-\ell_q)$.
Then for any $0\leq \kappa_3 \leq h_{a,0}$, there exists $0\leq \kappa_4 \leq h_{a,q}$ such that $v:=v_{\kappa_3,\kappa_4}$ solves  \eqref{eq:double obs 1} with 
\begin{equation}\label{eq:double obs 2}  
\left|\M \varphi -\M v \right|(\Omega)=\M \varphi(\Omega) -\M v (\Omega)=\omega_n a^n.
\end{equation}   
Assume in addition that $\partial\Omega \in C^{2,\alpha}$ and $\varphi \in \C_+^{2,\alpha}(\overline\Omega)$, then for each fixed $\kappa_3\in[0,h_{a,0}]$, the corresponding solution $v_{\kappa_3,\kappa_4}$ is unique (even though $\kappa_4$ may vary). Consequently, as $\kappa_3 \to h_{a,0}$, these solutions converge uniformly to $v_a(\cdot,0)$ on $\overline\Omega$.
\end{Lemma}
\begin{proof} 
By Lemma~\ref{lem:obs monotonicity}, the quantity $\left|\M \varphi -\M v_{\kappa_3,\kappa_4}\right|(\Omega)$
is continuous and non-decreasing in $\kappa_4$.
We note that $v_{h_{a,0}, 0} = v_a(\cdot, 0)$ and $v_{0, h_{a,q}} = v_a(\cdot, q)$  by interpreting the obstacle equation in the same way as \eqref{eq:defnvaq 1}.
Fix $\kappa_3 \leq h_{a,0}$. Then we have $v_{\kappa_3,0}\leq v_{h_{a,0},0}=v_a(\cdot,0)$, and hence
\[
\left|\M \varphi -\M v_{\kappa_3,0} \right|(\Omega)\leq\omega_n a^n= \left|\M \varphi -\M v_{0,h_{a,q}} \right|(\Omega)\leq \left|\M \varphi -\M v_{\kappa_3,h_{a,q}} \right|(\Omega).
\]
This implies the existence of $0\leq \kappa_4 \leq h_{a,q}$ such that  $v_{\kappa_3,\kappa_4}$ satisfying \eqref{eq:double obs 2}.

Fix $\kappa_3\in[0,h_{a,0}]$. Let $v = v_{\kappa_3,\kappa_4}$ and $v' = v_{\kappa_3,\kappa_4'}$ be solutions to  \eqref{eq:double obs 1} satisfying \eqref{eq:double obs 2}, with coincidence sets denoted by $K_3, K_4$ and $K_3', K_4'$, respectively. 
Suppose $\kappa'_4 \ge \kappa_4$. Then $v' \ge v$. 
Applying Lemma \ref{lem:obs monotonicity} to $v$ and $v'$, we obtain $K_3' \subset K_3 \subset \{v_a(\cdot,0) = h_{a,0}\}$ and $K_4 \subset K_4' \subset \{v_a(\cdot,q) = \ell_q + h_{a,q}\}$. 
Together with the strict inequalities $h_{a,0} < \inf_{\partial \Omega} \varphi$ and $h_{a,q} < \inf_{\partial \Omega} (\varphi - \ell_q)$, these inclusions imply that all these sets are compactly contained in $\Omega$ and hence stay away from the boundary.
By Proposition \ref{prop:dirichlet problem main}, $v$ belongs to $\C_+^{2,\alpha}$ in a neighbourhood of the boundary. 
Combined with the facts that $v' \geq v$ and $\M v'(\Omega) = \M v(\Omega)$, Lemma~\ref{lem:extends to boundary} yields $v' = v$ near $\partial \Omega$. 
Therefore, by the strong maximum principle \cite{jian2025strong}, we conclude that $v' = v$ in $\Omega \setminus (K_3 \cup K_4')$, and hence also on $\partial (K_3 \cup K_4')$.
Let $F$ be the convex hull of $\partial K_3 \setminus K_4'$. Then $K_3 \setminus K_4' \subset F$. 
Moreover, on $\partial F \cap \partial K_3$ we have $v' = v = \kappa_3$. Since $v'$ is convex and satisfies $v' \ge \kappa_3$, it follows that $v' \equiv \kappa_3$ throughout $F$. Consequently, $(K_3 \setminus K_4') \subset F \subset K_3'$. 
Together with $K_4 \subset K_4'$, this yields $(K_3 \cup K_4) \subset (K_3' \cup K_4')$.
Recalling the measure identity
\[
\M \varphi(K_3' \cup K_4') = \M \varphi(K_3 \cup K_4) = \omega_n a^n,
\]
we obtain $|(K_3' \cup K_4') \setminus (K_3 \cup K_4)| = 0$. Therefore, $\M v' = \M v$, which yields $v' = v$.

Let $\kappa_3 \to h_{a,0}$. Since the convex functions $v_{\kappa_3,\kappa_4} \geq \varphi$ are equicontinuous on $\overline{\Omega}$, the Arzelà–Ascoli theorem implies that, after passing to a subsequence, $\kappa_4 \to \kappa' \leq h_{a,q}$ and $v_{\kappa_3,\kappa_4}$ converges uniformly to a convex function on $\overline{\Omega}$. One can verify that the limit is still an obstacle solution of the form $v_{h_{a,0}, \kappa'}$. Because the measure discrepancy \eqref{eq:double obs 2} is preserved under uniform convergence, we must have $v_{h_{a,0}, \kappa'} = v_{h_{a,0},0} = v_a(\cdot, 0)$. Consequently, the family $v_{\kappa_3,\kappa_4}$ converges uniformly to $v_a(\cdot, 0)$ on $\overline\Omega$ as $\kappa_3 \to h_{a,0}$.
\end{proof}

\begin{Lemma}\label{lem:double obs 1}  
Assume that $n\geq 3$, $\partial\Omega \in C^{2,\alpha}, \varphi \in \C_+^{2,\alpha}(\overline\Omega)$, and \eqref{eq:varphi normalized s7} is satisfied. Let $a, \varepsilon, \rho > 0$ be small constants, and let  $v_{a}(x):=v_{a}(x,0)$ be the obstacle solution of \eqref{eq:obse ap 3} with obstacle $h_a:=h_{a,0}$,
and denote $\kappa_3=h_{(1-\varepsilon)a}$. Let $q$ and $y_q$ be such that $y_q \in \partial B_{(1+\rho)a} (0)$  and  $q=\nabla v_a(y_q)$. 
Suppose $v:=v_{\kappa_3,\kappa_4}$ solves \eqref{eq:double obs 1} with \eqref{eq:double obs 2}. Then we have for all small $\varepsilon \gtrsim a^{\beta}$ with $\beta=\frac{(n-2)\alpha}{n+\alpha}$ that
\begin{align}
h_a-\kappa_3 =\sup_{x\in \Omega} \left\{v_{a}(x)-v (x)\right\}   &= v_{a}(0)-v(0) \approx \varepsilon  a^2, \label{eq:dna2lower vat0} \\
\kappa_5:=\sup_{x\in \Omega} \left\{v(x)-v_{a}(x)\right\}  & = v(y_q)-v_{a}(y_q) \lesssim \varepsilon^{\frac{2}{n}} a^2 . \label{eq:dna2upper vatq}  
\end{align}
Moreover, $v $ is strictly convex and $C^{2,\alpha}$ outside its coincidence set, and we have
\begin{equation}\label{eq:dna2 obs ap 3qta}  
y_q\in \left\{v = \ell_{q}   +\kappa_4\right\} \subset  B_{C\varepsilon^{\frac{1}{n}} a}(y_q) .
\end{equation}
\end{Lemma}
\begin{proof} 
By a standard approximation argument, we may assume $\partial \Omega \in \C_+^{3,\alpha}$ and $\varphi \in \C_+^{3,\alpha}(\overline{\Omega})$. The convergence of the approximating obstacle solutions in this argument is guaranteed by the uniqueness in Lemma \ref{lem:exist kappa4} and the comparison principle. After a further affine transformation, the normalization \eqref{eq:varphi normalized s7} still holds.

\textbf{Step 1.} We claim that $v_a$ and $v$ coincide at some point on $\partial B_{4a}(0)$. The proof proceeds in two cases. 

Case 1. Suppose $v_a$ identically equals to $v$ in a neighborhood of $\partial\Omega$. Then the unique continuation theorem for Monge-Ampère equations (see \cite[Theorem 1.3]{mooney2015partial}) implies that $v_a \equiv v$ outside the union of their coincidence sets. This union is contained in $\{v_{h_{a,0},0} = h_{a,0}\} \cup \{v_{0,h_{a,q}} = \ell_q + h_{a,q}\}$, which, by \eqref{eq:dna2 obs ap vaball}, lies within $(B_{3a/2}(0) \cup B_{3a/2}(y_q)) \subset\subset B_{4a}(0)$.

Case 2. If $v_a$ does not identically equal to  $v$ near $\partial\Omega$, then Lemma~\ref{lem:extends to boundary} implies that both
\[
E := \{v_a> v\}\quad \text{and} \quad F := \{v > v_a\}
\]
are nonempty, and extend to $\partial\Omega$.  
The comparison principle argument, as presented in the first paragraph of Lemma~\ref{lem:obs monotonicity}, implies that each connected component of the open set $E$ or $F$ must intersect $\{v_a = h_a\}$ or $\{v = \ell_q + \kappa_4\}$, respectively, and thus must contain $0$ or $y_q$, respectively. Therefore, both $E$ and $F$ are connected.
It then follows that $v_a$ and $v$ coincide at some point on $\partial B_{4a}(0)$.

\textbf{Step 2.} 
Applying Lemma~\ref{lem:obs monotonicity} yields the two-sided bound
\[
v_a - (h_a - \kappa_3) \le v \le v_a + \kappa_5,
\]
where $\kappa_5 := v(y_q) - v_a(y_q)$.
We shall establish \eqref{eq:dna2lower vat0}, \eqref{eq:dna2upper vatq} and \eqref{eq:dna2 obs ap 3qta} under the assumption that
\begin{equation}\label{eq:kappa5 small}
\frac{h_a-\kappa_3+\kappa_5}{a^2} \leq \gamma(n,\rho)
\end{equation}  
is small.

\textbf{Step 2.1.} 
From the approximation \eqref{eq:dna2vaaround0}, the rescaled function $\frac{v_a(ax)}{a^2}$ is close to $W_1^*(x) + d_{n,0}$ in $B_7(0)$. Then, estimate  \eqref{eq:kappa5 small} show that $\frac{v(ax)}{a^2}$ is also close to $W_1^*(x) + d_{n,0}$ in $B_7(0)$.  
Hence,   the two coincidence sets of $v$ are disjoint, and are contained in $B_{1+\rho a/16 }(0)$ and $B_{\rho a/16}\left( y_q\right)$, respectively.
Applying Theorem \ref{lem:small perturbation}, we conclude that $\frac{v_a(ax)}{a^2}$ is  $\C_+^{2,\alpha}$  in  $B_{6}(0) \setminus  B_{1+\rho/8 }(0) $ and $\frac{v(ax)}{a^2}$  is $\C_+^{2,\alpha}$  in  $B_{6}(0) \setminus  \left(B_{1+\rho/8 }(0) \cup B_{\rho/8}\left(a^{-1}y_q\right)\right) $.

\textbf{Step 2.2.}  Set $\varepsilon_0 = \varepsilon + C a^{\beta}$. From \eqref{eq:dna2vaaround0} we obtain
\[
h_a - \kappa_3 = h_a - h_{(1-\varepsilon)a} \leq C \varepsilon_0 a^2,
\]
and that $h_a - \kappa_3 \approx \varepsilon_0 a^2$ when $\varepsilon \gtrsim a^{\beta}$.

To prove \eqref{eq:dna2lower vat0}, it suffices to show $v_a(0) - v(0) = h_a - \kappa_3$, i.e., $0 \in \{ v = \kappa_3 \}$. This will follow from the measure estimates
\begin{equation}\label{eq:volume coin k4}
\M \varphi \left( \{ v = \kappa_3 \} \right) \geq (1 - C \varepsilon_0)^n \omega_n a^n,\quad  \M \varphi \left(\left\{ v= \ell_{{q}} +\kappa_4\right\}\right)\leq C \varepsilon_0\omega_n a^n.
\end{equation} 
Indeed, since $\{ v = \kappa_3 \} \subset \{ v_a = h_a \}$, combining \eqref{eq:volume coin k4} with \eqref{eq:dna2 obs ap vaball} gives
\[
0 \in B_{a/2}(0) \subset \{ v = \kappa_3 \},
\]
provided $a$ and $\varepsilon$ are sufficiently small.

We now verify \eqref{eq:volume coin k4}. 
Since $v_a$ equals to $v$ at some point on $\partial B_{4a} (0)$,  the Harnack inequality implies
\[
 v(x)+h_a-\kappa_3 -v_a(x)  \approx h_a-\kappa_3  \quad   \text{in } B_{5a}(0) \setminus  \left(B_{(1+\rho/4)a }(0) \cup B_{\rho a/4}( y_q)\right).
\] 
This implies 
\[
v \leq v_a +C(h_a-\kappa_3 ) \leq v_a +C\varepsilon_0 a^2 \leq W_a^*(x)+h_a+C\varepsilon_0 a^2  \quad 
\text{in }B_{(1+\rho/4)a }(0).
\]
Let us consider
\[
w(x):=\left(1-Ca^{\alpha} \right) W_{(1-M \varepsilon_0)a}^*(x) +\kappa_3 \quad \text{on } \partial B_{(1+\rho/2)a }(0).
\]
Thus, for another sufficiently large constant $M$, we have
\[
v(x)-\kappa_3 \leq v(x) -h_a+C\varepsilon_0 a^2 \le w(x) -\kappa_3 \quad \text{on } \partial B_{(1+\rho/2)a }(0).
\]
The comparison principle implies that $w \geq v$ in $B_{(1+\rho/2)a }(0)$, and we have
\[
\left\{ W_{(1-M \varepsilon_0)a}^* =0 \right\} \subset \left\{ v= \kappa_3 \right\}.
\]
Therefore, $\M \varphi \left( \{ v = \kappa_3 \} \right) \geq (1-Ca^\alpha)(1 - M \varepsilon_0)^n \omega_n a^n$, which yields \eqref{eq:volume coin k4} upon noting that
\[
 \M \varphi \left(\left\{ v= \ell_{{q}} +\kappa_4\right\}\right)=  \omega_n a^n-\M \varphi \left( \{ v = \kappa_3 \} \right).
\]

\textbf{Step 2.3.}  
Let us show the upper bound estimate \eqref{eq:dna2upper vatq}  that $\kappa_5 \lesssim \varepsilon_0^{\frac{2}{n}}a^2$. 
Since $v_a$ equals to $v$ at some point on $\partial B_{4a} (0)$,  the Harnack inequality implies
\[
\frac{v_a(ax) +\kappa_5-v(ax) }{a^2}\approx  \frac{\kappa_5}{a^2}\quad   \text{in } B_{5}(0) \setminus  \left(B_{(1+\rho/4) }(0) \cup B_{\rho/4}(a^{-1}y_q)\right).
\] 
Applying \eqref{eq:abp sup 2/n} to the rescaled functions $\frac{v_a(ax) + \kappa_5}{a^2}$ and $\frac{v(ax)}{a^2}$ in $B_{\rho/2}(a^{-1}y_q)$, and noting that $\frac{v(ax)}{a^2} \in \C_+^{2,\alpha}$, we find that
\[
\frac{v_a(y_q) +\kappa_5 }{a^2}-\frac{v(y_q)}{a^2}\geq c\kappa_5-C\left(\M \varphi \left(\left\{ v= \ell_{{q}} +\kappa_4\right\}\right)\right)^{\frac{2}{n}} \geq c\kappa_5-C\varepsilon_0^{\frac{2}{n}}a^2.
\]
From the equality $v_a(y_q) + \kappa_5 = v(y_q)$, it follows that $\kappa_5 \lesssim \varepsilon_0^{2/n} a^2$.

\textbf{Step 2.4.}  
Then, we prove the inclusion \eqref{eq:dna2 obs ap 3qta}.  
Since $v_a + \kappa_5$ is of class $C_+^{2,\alpha}$ in $B_{\rho a}(y_q)$ and touches $v$ from above at $y_q \in \{ v = \ell_q + \kappa_4 \}$, and since 
$b:=\left(\M \varphi \left(\left\{ v= \ell_{{q}} +\kappa_4\right\}\right)\right)^{\frac{1}{n}} \leq C\varepsilon_0^{\frac{1}{n}} a$, 
the same argument used to prove \eqref{eq:coincidence inclusion} implies
\[
\{v = \ell_{q} + \kappa_4\} \subset S_{Cb^2}^{v_a+\kappa_5}(y_q) \subset  B_{C\varepsilon_0^{\frac{1}{n}} a}(y_q).
\] 
This yields \eqref{eq:dna2 obs ap 3qta} provided $\varepsilon_0 \gtrsim a^{\beta}$.

\textbf{Step 3.}  Finally, we note that Lemma \ref{lem:exist kappa4} implies that \eqref{eq:kappa5 small} is satisfied when $\varepsilon$ is sufficiently small. Since the arguments in Steps 2.2 and 2.3 imply the estimate $h_a-\kappa_3+\kappa_5 \leq C (\varepsilon + a^{\beta})^{2/n} a^2$, provided that \eqref{eq:kappa5 small} holds.  By the continuous dependence of the solution $v$ on the parameter $\varepsilon$, we conclude that \eqref{eq:kappa5 small} remains valid whenever $C(\varepsilon + a^{\beta})^{2/n} a^2 < \gamma(n,\rho)$. This proves  \eqref{eq:dna2lower vat0}, \eqref{eq:dna2upper vatq} and \eqref{eq:dna2 obs ap 3qta}, and the proof is completed.
\end{proof}

\begin{proof}[Proof of Theorem \ref{thm:rigidity}. The Non-positive Case.]
For the non-positive case $\mu \leq 0$, we have the inequality $u \geq \varphi$. For simplicity we assume that $0 \in \partial u(x_0)$. After a suitable affine transformation, we can also assume that \eqref{eq:varphi normalized s7} holds. Using the regularity condition $\varphi \in \C_+^{2,\alpha}$, we have
\[
\varphi(0)+c|x_0^2|\le \varphi(x_0)\le u(x_0)\le u(0)\le \varphi(0)+Ca^2,
\]
where we used \eqref{eq:perturb-result} in the last inequality. 
This yields the rough bound $|x_0| \leq C a$, and thus, $|\nabla \varphi(x_0)| \leq C a$ and $\left|\det D^2 \varphi(x_0) -1\right| \le Ca^{\alpha}$.

Let us define $\varepsilon_1$ by
\[
u(x_0)-\varphi(x_0) = \left(1-{\varepsilon_1}\right) d_{n,0}a^2.
\]
Theorem \ref{thm:ordera2 more} implies that $\varepsilon_1\ge -Ca^{\beta}$. Let $\varepsilon=\varepsilon_1+2Ca^{\beta}> Ca^{\beta}$. Then we have
\[
u(0)=u(0)-\varphi(0) \geq u(x_0)-\varphi(x_0) \geq  \left(1-{\varepsilon}\right)d_{n,0}a^2+Ca^{2+\beta}.
\]
Let $v_{a}(x) := v_{a}(x,0)$ be the solution of \eqref{eq:obse ap 3} with obstacle $h_a:=h_{a,0}$, and set $\kappa_3 = h_{(1-\varepsilon)a}$. Let $q$ and $y_q$ be such that $y_q \in \partial B_{(1+\rho)a} (0)$  and  $q=\nabla v_a(y_q)$, and let $\varpi(\cdot, q)$ denote the solution $v$ given by Lemma~\ref{lem:double obs 1}. Note that  on $\left\{ \varpi(\cdot,q) =\kappa_3\right\}$, we have
\[
\varpi(\cdot,q) = \kappa_3 =h_{(1-\varepsilon) a} \leq \left(1-{\varepsilon}\right)^2d_{n,0}a^2+Ca^{2+\beta} <u(0) \leq u .
\]
The graphs of these three functions $u,v_a$ and $\varpi(\cdot,q)$ can be illustrated as follows:
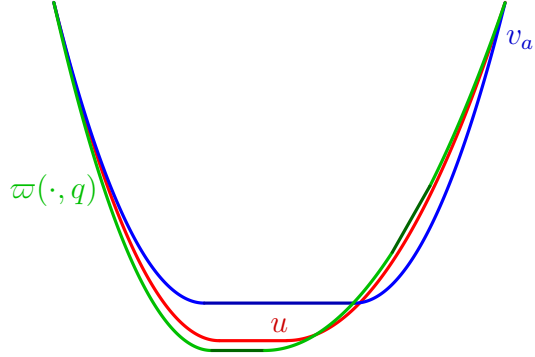
\begin{center}
\begin{tikzpicture}[scale=1, line cap=round, line join=round]
  \draw[very thick, red]
    plot[domain=0.1:3, samples=200] (\x,{(450/841)*\x*\x-(90/841)*\x-84/168+0.007});
  \draw[very thick, red]
    plot[domain=-3:-0.8, samples=200] (\x,{(225/242)*\x*\x+(180/121)*\x+23/242});
\draw[very thick, red]
  plot[domain=-0.8:0.1, samples=200] (\x,{0*\x-0.5});
  \node[red!80!black] at (0,-0.3) {$u$};

\draw[very thick, blue]
    plot[domain=-3:-1, samples=300] (\x,{(\x+1)*(\x+1)});
\draw[very thick, blue]
  plot[domain=1:3, samples=300] (\x,{(\x-1)*(\x-1))});
\draw[very thick, blue!70!black]
  plot[domain=-1:1, samples=300] (\x,{0*\x});
  \node[blue!80!black] at (3.2,3.5) {$v_a$};
  
\draw[very thick, green!75!black]
  plot[domain=-3:-0.9, samples=300] (\x,{1.05*(\x+0.9)*(\x+0.9)-0.6305});
\draw[very thick, green!40!black]
  plot[domain=-0.9:-0.2, samples=300] (\x,{0*\x-0.6305});
\draw[very thick, green!75!black]
  plot[domain=-0.2:1.5, samples=300] (\x,{0.452197265625*\x*\x+0.18087890625*\x-0.612412109375});
\draw[very thick, green!75!black]
  plot[domain=2:3, samples=300] (\x,{0.452197265625*\x*\x+0.18087890625*\x-0.612412109375});
\draw[very thick, green!40!black]
  plot[domain=1.5:2, samples=300] (\x,{1.7635693359375*\x-1.96900390625});
\node[green!75!black] at (-3,1.5) {$\varpi(\cdot,q)$};

\end{tikzpicture}
\end{center}

By applying Lemma \ref{lem:comparison gene} to the set $W=\left\{ \varpi(\cdot,q) =\kappa_3\right\}\cup \left\{ \varpi(\cdot,q) =\ell_{{q}} +\kappa_4\right\}$, we find that  $u(z_q)\leq \varpi(z_q,q) $ for some $z_q \in \left\{ \varpi(\cdot,q) =\ell_{{q}} +\kappa_4\right\}$. Thus, \eqref{eq:dna2 obs ap 3qta} yields $|y_q - z_q| \leq C\varepsilon^{\frac{1}{n}}a$, and  \eqref{eq:dna2upper vatq} yields 
\[
u(z_q)\leq \varpi(z_q,q)\leq v_a(z_q) +C\varepsilon^{\frac{2}{n}}a^2.
\]
Therefore, for sufficiently small $\varepsilon$, by considering all $y_q\in \partial B_{(1+\rho)a} (0)$, we have
\[
(1+\rho-C\varepsilon^{\frac{1}{n}} )B_{a}(0)  \subset
 \operatorname{conv}\left\{ \bigcup_{y_q \in \partial B_{(1+\rho)a}(0)} \{z_q\}\right\}
\subset \left\{ x:\;u (x)\leq  b_1 \right\},
\]
where 
\[
b_1=\sup_{y\in (1+\rho +C\varepsilon^{\frac{1}{n}})B_{a}(0)} v_a(y) +C\varepsilon^{\frac{2}{n}}a^2  \overset{\eqref{eq:dna2vaaround0}}{\leq} W_{a}^*((1+\rho +C\varepsilon^{\frac{1}{n}})ae_n)+d_{n,0}a^2+C\varepsilon^{\frac{2}{n}}a^2.
\] 
Note that we only need to consider $\varepsilon\le c\rho^n$, since otherwise, the estimate \eqref{eq:stabilitynegative} automatically holds. Taking $\sigma =  (1+\rho-C\varepsilon^{\frac{1}{n}} ) a$, this yields for a larger constant $C$ that
\[
u(x) \leq  b_1 \leq   W_{a}^*(\sigma e_n)+d_{n,0}a^2 +C\varepsilon^{\frac{1}{n}}a^2 \quad \text{on } \partial B_{\sigma}(0).
\] 

Let us denote 
\[
-\mu(B_{\sigma }(0))=(1-\kappa)^n \omega_na^n
\]
for some $\kappa\ge 0$, and let
\[
\varpi(x)=\left(1-Ca^{\alpha} \right) \left[W_{(1-\kappa)a}^*(x) -W_{(1-\kappa)a}^*(\sigma e_n)\right]+W_a^*(\sigma e_n)+d_{n,0}a^2 +C\varepsilon^{\frac{1}{n}}a^2 .
\]
Then $u< \varpi$ on $\partial B_{\sigma }(0)$. 
Suppose $u(x_0)>\varpi(0)$, then $u(x) \geq u(x_0)>\varpi(0)$ for all $x$, which yields that
\[
\left\{ W_{(1-\kappa)a}^* =0 \right\}=\{\varpi=\varpi(0)\} \subset E:= \left\{ x \in B_{\sigma}(0):\; u>\varpi \right\}\subset\subset  B_{\sigma}(0).
\]
Since $-\mu(B_{\sigma }(0))=(1-\kappa)^n \omega_na^n$, the first inclusion implies 
\[
\M u(E) = \M \varphi(E) + \mu(E) \ge \M \varphi(E) -(1-\kappa)^n \omega_na^n> \M \varpi(E),
\] 
while the second inclusion yields $\M u(E) \leq \M \varpi(E)$. This is impossible. This yields
\[
\begin{split}
u(x_0) \leq \varpi(0) &= -\left(1-Ca^{\alpha} \right)   W_{(1-\kappa)a}^*(\sigma e_n) +W_a^*(\sigma e_n)+d_{n,0}a^2 +C\varepsilon^{\frac{1}{n}}a^2 \\
&\le (C a^\alpha-c\kappa) a^2 +d_{n,0}a^2 +C\varepsilon^{\frac{1}{n}}a^2.
\end{split}
\]
Combined with
\[
u(x_0) =  \varphi(x_0) +\left(1-{\varepsilon_1}\right)d_{n,0}a^2  \geq  \left(1-{\varepsilon}\right)d_{n,0}a^2-Ca^{2+\beta} ,
\]
we now derive that
\[ 
\kappa \lesssim   \varepsilon^{\frac{1}{n}}.
\] 

Finally, since $\varphi \in \C_+^{2,\alpha}$, we have
\[
\begin{split}
\varpi (0)\geq u(x_0) &\geq  \varphi(x_0)+\left(1-{\varepsilon_1}\right)d_{n,0}a^2 \\
&\geq \frac{1}{2}|x_0|^2 -C|x_0|^{2+\alpha}+\left(1-{\varepsilon}\right)d_{n,0}a^2 -Ca^{2+\alpha} .
\end{split}
\]
We have $|x_0| \leq C \varepsilon^{\frac{1}{2n}}a$, and thus $B_\sigma(0)\subset B_{1+2\rho}(x_0)$. Therefore,
\[
-\mu(B_{1+2\rho}(x_0))\ge -\mu(B_{\sigma }(0))=(1-\kappa)^n \omega_na^n\ge (1-C\varepsilon^{\frac{1}{n}})^n \omega_na^n\ge (1-C\varepsilon^{\frac{1}{n}}) \omega_na^n.
\]
Replacing $\rho$ with $\rho/2$ and noting that $|\det D^2 \varphi(x_0) - 1| \le C a^{\alpha}$, we establish \eqref{eq:stabilitynegative} in Theorem \ref{thm:rigidity} for the non-positive case.

This concludes the proof of Theorem \ref{thm:rigidity}.
\end{proof}

\appendix
\counterwithin*{equation}{section}
\renewcommand\theequation{\thesection\arabic{equation}}

\section{Boundary $C^{2,\alpha}$ regularity under interior perturbation}\label{app:boundary regularity}

In this section, under the assumptions that $\Omega_1 \subset \subset \Omega$ is a convex subdomain and the perturbed measure $\mu=\M u -\M \varphi$ satisfies 
\[
\operatorname{supp}  \mu \subset\subset \Omega_1 \subset \subset \Omega,
\] 
we establish several boundary estimates for the difference $u - \varphi$.

\begin{Lemma}\label{lem:global lipschitz 1}
Suppose $u \in \D_{a,\varphi}$, $\Omega_1 \subset \subset \Omega$ is a convex subdomain and $\operatorname{supp}  \mu \subset\subset \Omega_1 \subset \subset \Omega$. Then we have
\[ 
|u(x)-\varphi(x)| \leq  Ca\operatorname{dist}(x,\partial \Omega) 
\quad \text{in } \Omega,\]
where $C$ is a positive constant depending only on $n$, $\operatorname{diam}(\Omega)$ and $\operatorname{dist}(\Om_1,\partial\Om)$. 
\end{Lemma}
\begin{proof} 
Due to the symmetry between $u$ and $\varphi$, it suffices to prove that
$u(x) \geq \varphi (x) -Ca\operatorname{dist}(x,\partial \Omega)$.
By the Jordan decomposition theorem, we may assume without loss of generality that $\mu$ is nonnegative, which implies $u \leq \varphi$.
An application of \eqref{eq:alexandrov estimate} then yields
\[
\varphi(y)-u(y)\leq \varphi(y)-u_a(y ,y )\leq   C\operatorname{dist}(y,\partial \Omega) a, \quad \forall\  y \in \Omega_1,
\]
where $C$ is a positive constant depending only on $n$, $\operatorname{diam}(\Omega)$ and $\operatorname{dist}(\Om_1,\partial\Om)$. Consequently,
\[ 
\varphi(x)-C a\operatorname{dist}(x,\partial \Omega) \leq   u(x) \leq \varphi(x)\quad \text{on } \partial \Omega \cup \partial \Omega_1.
\]
Since $-\operatorname{dist}(x,\partial \Omega)$ is convex, it follows from the comparison principle that
\[ 
\varphi(x)-C a\operatorname{dist}(x,\partial \Omega) \leq   u(x) \leq \varphi(x) \quad \text{on } \Omega \setminus \Omega_1 . 
\]
This concludes the proof.   
\end{proof}

In \cite{trudinger2008boundary}, Trudinger and Wang investigated the Dirichlet problem for Monge-Amp\`ere equations:
\[
\det D^2 u = f \quad \text{in } \Omega, \quad u = \varphi \quad \text{on } \partial\Omega.
\]
They established global Schauder estimates, proving $u \in C^{2,\alpha}(\overline{\Omega})$ when $\partial\Omega, \varphi \in C^3$ and $\partial\Omega$ is uniformly convex. The $C^3$ regularity assumptions are sharp,  through counterexamples in Wang \cite{wang1996regularity} when any $C^3$ condition is relaxed to $C^{2,1}$. 
Savin \cite{savin2013pointwise} obtained a pointwise Schauder estimate at boundary points $x_0$ under the assumptions that: $\partial\Omega, \varphi \in C^{2,\alpha}$, and the quadratic separation condition 
\begin{equation}\label{eq:sqc}
\varphi (x)  \geq \varphi(x_0)+ p \cdot \left(x-x_0\right)+ \kappa |x-x_0|^2, \quad \forall\  x \in B_c(x_0) \cap \partial \Omega 
\end{equation}
for some $p \in \partial \varphi(x_0)$, $\kappa, c > 0$. Notably, condition \eqref{eq:sqc} is satisfied when either $\partial \Omega , \varphi \in C^{3}$ and $\partial \Omega$ is uniformly convex, or  $\partial \Omega   \in C^{2,\alpha}$ is uniformly convex and  $\varphi\equiv 0$ on $\partial \Omega$.
In Lemma \ref{lem:dirichlet problem} below, we will see that these $C^{2,\alpha}$ boundary estimates remain valid under interior perturbations of the Monge-Amp\`ere measure, including both positive and sufficiently small negative variations. 

In the remainder of this section, we adopt the following assumptions:

\begin{Assumption*}
Let $\Omega \in C^{2,\alpha}$ be a bounded convex domain  and let $\Omega_1 \subset \subset \Omega$ be convex. 
Let $\varphi \in C(\overline{\Omega}) \cap \C_+^{2,\alpha}(\overline{\Omega\setminus\Omega_1})$ be a convex function. 
Let $u \in C(\overline{\Omega})$ be a convex function satisfying:
\[
u = \varphi \quad\text{on } \partial\Omega \quad \text{and} \quad \operatorname{supp}\mu \subset\subset \Omega_1,
\]
where $\mu = \M u - \M \varphi$. 
We consider the Jordan decomposition $\mu = \mu_+ - \mu_-$ into non-negative measures $\mu_+$ and $\mu_-$.
\end{Assumption*}

\begin{Lemma}\label{lem:dirichlet problem} 
Under assumption (H), the function $u$ is strictly convex in ${\Omega}\setminus\overline{\Omega}_1$. Moreover, there exist positive constants $c_1$ and $C_1$ such that if 
\[
\mu_-(\Omega) \leq c_1,
\]
then for any open set $\Omega_2$ satisfying $\Omega_1 \subset\subset \Omega_2 \subset \Omega$, we have
\[   
\left\| D^2 u\right\|_{C^{\alpha} (\overline{\Omega\setminus \Omega_2})}+\left\|(D^2 u)^{-1}\right\|_{L^{\infty} (\Omega\setminus \Omega_2)}\leq C_1,
\]
where $c_1$ depends only on $n$, $\alpha$,
$\operatorname{dist}(\Om_1,\partial\Om)$,  $\operatorname{diam}(\Omega)$,
$\left\|\partial\Omega\right\|_{C^{2,\alpha}}$, 
$\left\| D^2 \varphi\right\|_{C^{\alpha} (\overline{\Omega\setminus \Omega_1})}$ and $\left\|(D^2 \varphi)^{-1}\right\|_{L^{\infty} (\Omega\setminus \Omega_1)}$, 
and $C_1$ depends only on $n$, $\alpha$,
$|\mu|(\Omega) $,
$\operatorname{dist}(\Om_1,\partial\Om_2)$, $\operatorname{dist}(\Om_2,\partial\Om)$, $\operatorname{diam}(\Omega)$,
$\left\|\partial\Omega\right\|_{C^{2,\alpha}}$, $\left\|(D^2 \varphi)^{-1}\right\|_{L^{\infty} (\overline{\Omega\setminus \Omega_1})}$ and $\left\| D^2 \varphi\right\|_{C^{\alpha} (\overline{\Omega\setminus \Omega_1})}$.
\end{Lemma}
\begin{proof}  
From Caffarelli's work \cite{caffarelli1990ilocalization}, we know that the non-strictly convex set of $u$ in $\Omega\setminus \Omega_1$ consists of convex subsets $E$ emanating from $\partial{\Omega}_1 \cup \partial\Omega$. By \cite{caffarelli1993note}, the boundary assumptions prevent $E$ from reaching $\partial\Omega$, which forces $E \subset \overline{\Omega}_1$. Therefore, $u$ is strictly convex in $\Omega \setminus \overline{\Omega}_1$.

First, we verify that $u$ satisfies the quadratic separation condition \eqref{eq:sqc}. Let $u_{-}$ be the convex solution to
\[ 
\M   u_{-} =\M \varphi-\mu_- \quad \text{in } {\Omega}, \quad u_{-}  =\varphi\quad \text{on } \partial \Omega.
\] 
Applying the comparison principle together with Lemma \ref{lem:global lipschitz 1} yields
\[
\varphi (x) -C|\mu|(\Omega)^{\frac{1}{n}}\operatorname{dist}(x,\partial \Omega)\leq u(x) \leq u_{-}(x)\leq \varphi (x)+C\mu_-(\Omega)^{\frac{1}{n}}\operatorname{dist}(x,\partial \Omega).
\]  
Let $\nu$ denote the outer normal vector field on $\partial\Omega$. For any boundary point $x_0 \in \partial\Omega$, we define the canonical gradient $\nabla u(x_0)$ as the unique vector $p \in \partial u(x_0)$ satisfying $p \cdot \nu :=\inf \left\{ q\cdot \nu : q \in \partial u(x_0)  \right\}$.  
Then, $\nabla u(x_0) -\nabla \varphi (x_0)=t   \nu $ for some $t \in \R$, and the previous estimates imply
\[
-C|\mu|(\Omega)^{\frac{1}{n}} \leq (\nabla \varphi (x_0)- \nabla u(x_0)) \cdot \nu \leq C\mu_-(\Omega)^{\frac{1}{n}}.
\] 
Let $\mu_-(\Omega) \leq c_1$ be small,  by the $\C_+^{2,\alpha}$ regularity of $\varphi$ and the $C^{2,\alpha}$ regularity of $\partial\Omega$ at $x_0$, we establish that
\[
\begin{split}
\varphi (x)  
& \geq \varphi(x_0)+ \nabla \varphi(x_0) \cdot \left(x-x_0\right)+ c  |x-x_0|^2 \\ 
& \geq \varphi(x_0)+ \nabla u(x_0) \cdot \left(x-x_0\right)+ \frac{1}{2}c  |x-x_0|^2,
\end{split} 
\]
and
\[
\varphi (x)  \leq \varphi(x_0)+ \nabla u(x_0) \cdot \left(x-x_0\right)+ C_1 |x-x_0|^2 .
\]
  
Applying the pointwise Schauder estimates from \cite[Theorem 7.1]{savin2013pointwise} at boundary points, we obtain $\C^{2,\alpha}$ regularity for $u$ in $\overline\Omega\setminus\Omega_3$ for some $\Omega_3\subset\subset\Omega$ through a standard scaling argument. Moreover, these estimates in $\overline\Omega\setminus\Omega_3$ hold uniformly for all functions $u$ in the class 
\[ 
\C_{M}:= \left\{  u \in C(\overline{\Omega}) :\; u \text{ satisfies assumption (H)}, \quad |\mu|(\Omega) \leq M,\quad   \mu_{-}(\Omega) \leq c_1 \right\}.
\]  

Finally, since $u$ is strictly convex in $\Omega \setminus \overline{\Omega}_1$ and $\M u = \M \varphi$ in $\Omega \setminus \overline{\Omega}_1$ by assumption (H), one can establish the uniform interior $\C_+^{2,\alpha}$ regularity of $u \in \C_M$ in $\overline{\Omega_3}\setminus \Omega_2$ in a standard way, which is sketched as follows.  For any $x \in \overline{\Omega_3} \setminus \Omega_2$ and $u \in \C_M$, define
$
\bar{h}(u,x) := \sup \{ h : S_h(x_0) \subset \Omega \setminus \Omega_1 \} > 0,
$
where $S_h(x_0)$ denotes the section of $u$ at $x_0$ with height $h$.  
The quantity $\bar{h}$ is lower semi-continuous with respect to uniform convergence in both $u$ and $x$. Moreover, the class $\C_M$ is compact due to the uniform Lipschitz estimate. Using compactness arguments, we obtain a uniform positive lower bound for $\bar{h}$ on $\overline{\Omega_3} \setminus \Omega_2$.
Since $u$ is Lipschitz continuous, we have  $B_{c\bar{h}}(x) \subset S_{\bar{h}/2}(x)$ for some constant $c > 0$. Applying the classical regularity theory \cite{caffarelli1990ilocalization,caffarelli1990interiorw2p}, we deduce the $\C_+^{2,\alpha}$ estimate of $u$ in $B_{c\bar{h}}(x)$. A standard covering argument then yields the uniform $C^{2,\alpha}$ estimate in $\overline{\Omega_3 \setminus \Omega_2}$.
\end{proof}

\begin{Remark}
In Lemma \ref{lem:dirichlet problem}, the assumption regarding the smallness of the negative perturbation is necessary for the $C^{2,\alpha}$ regularity of $u$. This can be achieved by considering the function $u$ constructed in \cite[Example 2.1]{wang1996regularity} and resolving $\varphi$ by letting  $\M\varphi=\M u+\delta_{ce_2}$ with $ce_2$ in the interior of $\Omega$.
\end{Remark}

\begin{Lemma}\label{lem:div equation}
Suppose assumption (H) holds. 
Then for any $C^{1}$ domain $\Omega_2$ satisfying $\Omega_1\subset \subset  \Omega_2 \subset \Omega$, the divergence formula \eqref{eq:div identity smooth} still holds.
\end{Lemma}
\begin{proof}
The conclusion follows by constructing smooth approximations of  $u$ and $ \varphi$ that preserve both the boundary values and the behavior of the measure outside $\Omega_1$. Using Lemma \ref{lem:dirichlet problem}, such approximations imply the convergence in the $C^{2,\alpha}$ norm in a neighborhood of $\overline{\Omega\setminus \Omega_2}$. The desired result is then obtained by applying the identity \eqref{eq:div identity smooth} to these smooth approximations and passing to the limit.
\end{proof}

\begin{Lemma}\label{lem:com a G 02}
Suppose assumption (H) holds and $|\mu|(\Omega) \leq c$ is sufficiently small. Then for any open set $\Omega_2$ satisfying $\Omega_1 \subset \subset \Omega_2 \subset \subset \Omega$, we have 
\[
|u-\varphi|+|Du-D\varphi|+|D^2 u-D^2 \varphi | \leq C_2	|\mu|(\Omega)  \quad \text{in } \Omega \setminus \Omega_2,
\]
where all constants depend only on $n$, $\alpha$,
$\operatorname{dist}(\Om_1,\partial\Om_2)$, $\operatorname{dist}(\Om_2,\partial\Om)$, $\operatorname{diam}(\Omega)$,
$\left\|\partial\Omega\right\|_{C^{2,\alpha}}$,
$\left\| D^2 \varphi\right\|_{C^{\alpha} (\overline{\Omega\setminus \Omega_1})}$ and $\left\|(D^2 \varphi)^{-1}\right\|_{L^{\infty} (\Omega\setminus \Omega_1)}$.
\end{Lemma}

\begin{proof}
Let $\Omega_1 \subset \subset \Omega_3 \subset \subset \Omega_2$. By Lemma \ref{lem:dirichlet problem}, we obtain 
\[
\left\| D^2 \varphi\right\|_{C^{\alpha} (\overline{\Omega\setminus \Omega_3})}+\left\|(D^2 \varphi)^{-1}\right\|_{L^{\infty} (\Omega\setminus \Omega_3)} \leq C_2,
\]
which implies that the equation \eqref{eq:div equation} is uniformly elliptic in $\overline{\Omega \setminus\Omega_3}$. 
It suffices to show that 
\begin{equation}\label{eq:u-var an}
|\varphi(x) - u(x) |\leq C_2|\mu|(\Omega) \quad 
\text{in } \Omega \setminus \Omega_2,
\end{equation}
since the proof is then completed by applying standard elliptic regularity theory for the operator in \eqref{eq:div equation} to $(u - \varphi)/|\mu|(\Omega)$.

By the reduction in Section \ref{app:comparison reduction} regarding the sign of the measure, it suffices to consider the case where $\mu$ is either purely negative or purely positive. We present the proof of \eqref{eq:u-var an} for the case $\mu_- = 0$; the case $\mu_+ = 0$ follows analogously.

Since $\mu_-=0$, the comparison principle yields $\varphi \geq u$ in $\Omega$. 
Applying the classical Harnack inequality to \eqref{eq:div equation}, we may assume $\varphi - u \approx \kappa$ on $\partial\Omega_2$ for some  $\kappa \geq 0$. Through the standard boundary Lipschitz estimate and Hopf's lemma for uniformly elliptic equations, we derive the identity 
\[ 
\varphi(x)-u(x) \approx \kappa \operatorname{dist}\left(x,\partial\Omega\right) \quad \text{in } \Omega \setminus \Omega_{3}. 
\]
Combining this result with \eqref{eq:div identity smooth} shows that $\kappa \approx  \mu(\Omega)$, which establishes \eqref{eq:u-var an}. 
\end{proof}

\begin{proof}[Proof of Proposition \ref{prop:dirichlet problem main} ]
The proof follows by combining Lemmas \ref{lem:dirichlet problem}, \ref{lem:div equation} and \ref{lem:com a G 02}.
\end{proof}

\small

\bibliography{references}
%\begin{thebibliography}{99}

%\end{thebibliography}
\smallskip

\noindent T. Jin

\noindent Department of Mathematics, The Hong Kong University of Science and Technology\\
Clear Water Bay, Kowloon, Hong Kong\\
Email: \textsf{tianlingjin@ust.hk}

\medskip

\noindent X. Tu

\noindent Department of Mathematics, The Hong Kong University of Science and Technology\\
Clear Water Bay, Kowloon, Hong Kong\\[1mm]
Email:  \textsf{maxstu@ust.hk}

\medskip

\noindent J. Xiong

\noindent School of Mathematical Sciences, Laboratory of Mathematics and Complex Systems, MOE\\ Beijing Normal University, 
Beijing 100875, China\\
Email: \textsf{jx@bnu.edu.cn}

\end{document}